\documentclass[10pt, reqno]{amsart}
\usepackage{graphicx, amssymb, amsmath, amsthm}
\numberwithin{equation}{section}

\usepackage{amscd}
\def\pa{\partial}

\newcommand{\R}{\mathbb{R}}
\newcommand{\C}{\mathbb{C}}

\newcommand{\eps}{\varepsilon}

\newcommand{\N}{\mathcal{N}}

\newcommand{\abs}[1]{\left\lvert #1\right\rvert}

\newcommand{\al}{\alpha}

\newtheorem{theorem}{Theorem}[section]
\newtheorem{prop}[theorem]{Proposition}
\newtheorem{lemma}[theorem]{Lemma}

\newtheorem{corollary}[theorem]{Corollary}
\newtheorem{conjecture}[theorem]{Conjecture}
\newtheorem{proposition}[theorem]{Proposition}
\newcommand{\la}{\lambda}
\theoremstyle{definition}
\newtheorem{definition}[theorem]{Definition}
\newtheorem{remark}[theorem]{Remark}

\makeatletter
\newcommand{\Extend}[5]{\ext@arrow0099{\arrowfill@#1#2#3}{#4}{#5}}
\makeatother

\begin{document}

\title[Dynamics of  NLW ]{Dynamics of nonlinear wave equations}

\author{Changxing Miao  and  Jiqiang Zheng}
\address{Institute of Applied Physics and Computational Mathematics,
P. O. Box 8009,\ Beijing,\ China,\ 100088}
\email{miao\_changxing@iapcm.ac.cn,  zhengjiqiang@gmail.com}

\begin{abstract}
In this lecture, we will survey the study of dynamics of the  nonlinear wave equation in recent years.
We refer to some lecture notes including such as  C. Kenig \cite{Kenig01,Kenig02}, C. Kenig and F. Merle \cite{KM1} J. Shatah and M. Struwe \cite{SS98},  and C. Sogge \cite{sogge:wave} etc. This lecture was written for LIASFMA School and Workshop on Harmonic Analysis  and Wave Equations in Fudan universty.
\end{abstract}
%
%\begin{abstract}
%
%
%\end{abstract}

 \maketitle

 \tableofcontents %%代表运行目录，如果不要目录，可以把这句话注释掉。

\section{Introduction}

\noindent

\noindent We study the initial-value problem for
nonlinear
 wave equations as following
\begin{align} \label{equ:nlw}
({\rm NLW})\quad \begin{cases}    \partial_{tt}u-\Delta u= \mu|u|^{p-1}u,\quad
(t,x)\in\R\times\R^d,
\\
(u,\pa_tu)(0,x)=(u_0,u_1)(x),
\end{cases}
\end{align}
where $u:\R_t\times\R_x^d\to \R$, $\mu=\pm 1$ with $\mu=-1$ known as the defocusing case and $\mu=1$
as the focusing case.
Equation \eqref{equ:nlw} admits a number of symmetries, explicitly:
\vskip0.15cm

$\bullet$ {\bf Space-time translation invariance:} if $u(t,x)$ solves \eqref{equ:nlw}, then so does $u(t+t_0,x+x_0),~(t_0,x_0)\in\R\times\R^d$;

\vskip0.15cm

$\bullet$ {\bf Lorentz transformations invariance:} if $u(t,x)$ solves \eqref{equ:nlw}, then so does
\begin{equation}
\label{travelling_waves}
u\left(\frac{t-\ell\cdot x}{\sqrt{1-|\ell|^2}},\left(-\frac{t}{\sqrt{1-|\ell|^2}}+\frac{1}{|\ell|^2} \left(\frac{1}{\sqrt{1-|\ell|^2}}-1\right)\ell\cdot x\right)\ell+x\right)
\end{equation}
where $\ell \in \R^d$, $|\ell|<1$.
\vskip0.15cm

$\bullet$ {\bf Scaling invariance:} if $u(t,x)$ solves \eqref{equ:nlw}, then so does $u_\lambda(t,x)$ defined by
\begin{equation}\label{equ:scale}
u_\lambda(t,x)=\lambda^{\frac2{p-1}}u(\lambda t, \lambda x),\quad\lambda>0.
\end{equation}
This scaling defines a notion of \emph{criticality} for \eqref{equ:nlw}. In particular, one can check that
the only homogeneous $L_x^2$-based Sobolev space that is left invariant under \eqref{equ:scale} is $\dot{H}_x^{s_c}(\R^d)$, where the \emph{critical regularity} $s_c$ is given by $s_c:=\frac{d}2-\frac2{p-1}$.
\vskip0.15cm

 $\bullet$ For $s_c=1$, we call the problem \eqref{equ:nlw} \emph{\bf energy-critical};
\vskip0.1cm

 $\bullet$   For $s_c<1,$ we call the problem \eqref{equ:nlw} \emph{\bf energy-subcritical};

\vskip0.1cm
 $\bullet$    For $s_c>1$,  we call the problem \eqref{equ:nlw} \emph{\bf energy supercritical}.

\vskip0.1cm
\noindent From the Ehrenfest law or direct computation, the above symmetries induce invariances in the energy space, namely
\begin{align*}
&\text{\bf Energy:}~E(u)(t)\triangleq\frac12\int_{\R^d}\big(|\nabla u|^2+|\pa_tu|^2\big)\;dx-\frac{\mu}{p+1}\int_{\R^d}|u(t,x)|^{p+1}\;dx=E(u_0,u_1),\\
&\text{\bf Momentum:}~P(u)(t)\triangleq\int\nabla uu_t\;dx=P(u_0,u_1).
\end{align*}

To begin, we need a few definitions.
\begin{definition}[classical, weak, strong, strong Strichartz solution]\label{def:solution}.
\begin{itemize}
\item {\bf Classical solution:} A function $u$ is a {\it classical} solution of \eqref{equ:nlw} on a time interval $I$ containing $0$
 if $u\in C^2(I\times\R^d)$ and solves \eqref{equ:nlw} in the classical sense.
\vskip0.15cm

\item {\bf Weak solution:} A function $u$ is a {\it weak} solution of \eqref{equ:nlw} if $(\pa_tu,\nabla_xu)\in L_t^\infty(\R,L_x^2),~u\in L_t^\infty L_x^{p+1}$, and $u$
 solves \eqref{equ:nlw} in distribution sense, namely
\begin{align}
\nonumber\int_0^{+\infty}\int_{\R^d}\,u\Box\varphi\,dx\,dt&+\int_0^{+\infty}\int_{\R^d}\,f(u)\varphi\,dx\,dt=
-\int_{\R^d}\,u_0(x)\partial_t\varphi(0,x)\,dx\\
&+\int_{\R^d}\,u_1(x)\varphi(0,x)\,dx,\;\;\;\forall\;\;\;\varphi\in{\mathcal D}(\R\times\R^d),\label{weaksense}
\end{align}
and the energy inequality
\begin{equation}
\label{energy1}
E(u(t))\leq E(u(0)),\;\;\; \forall\;\;t\in \R
\end{equation} holds. Here $f(u)=\pm|u|^{p-1}u$ and $\Box=\partial_t^2-\Delta_x$ is the d'Alembertian operator.
\vskip0.15cm

\item {\bf Strong solution:}  $u:I\times\R^d\to\R$ is a \emph{strong  solution} of \eqref{equ:nlw} with data
$(f,g)\in H^s\times H^{s-1}(\R^d)$, $s\in \R$,
if $u$ satisfies the Duhamel formula
\begin{align}u(t,\cdot)=&\cos(t|\nabla|)f
+\frac{\sin(t|\nabla|)}{|\nabla|}g \nonumber \\
&-\mu\int_0^t \frac{\sin((t-s)|\nabla|)}{|\nabla|}
\left [u(s,\cdot)|u(s,\cdot)|^{p-1} \right ] ds \label{eq:Duhamel}
\end{align}
for $t$ in time interval $I$ containing $0$.

\vskip0.15cm
\item {\bf Strong Strichartz solution:} we say that $u$ is a \emph{strong Strichartz solution} of \eqref{equ:nlw}, if $u$ is a strong solution and $u$ belongs to some auxiliary spaces associated with the Strichartz estimate, such as  some spatial-time space $L_t^q(I,L_x^r(\R^d)).$

\end{itemize}

\end{definition}

\begin{remark}
The existence of global weak solutions is proved by Strauss \cite{Strauss}.
\end{remark}

\begin{definition}[Well-posedness]\label{def:wellposed}
We say that the Cauchy Problem \eqref{equ:nlw} for (NLW) is \emph{locally well-posed} ({\bf LWP}) in $H^s$ , if for every $(u_0,u_1) \in H^s_x \times H^{s-1}_x$

\vskip0.15cm
{\rm (i)}\; There exist a time $T = T (u_0,u_1)>0$ and a distributional solution
$u: [-T,T] \times \R^d \to \R$ to \eqref{equ:nlw} which is in
the space $C^0 ([-T,T]; H^s_x)\cap C^1([-T,T];H^{s-1}_x)$;

\vskip0.15cm
{\rm (ii)}\;  The solution map $(u_0,u_1) \mapsto
u$ is uniformly continuous\footnote{Here, we require
the solution mapping to be continuous but not necessarily uniformly continuous,
as is the case for certain equations, such as Burgers' equation, the Korteweg-de Vries
and Benjamin-Ono equations in the periodic case.}
from $H^s\times H^{s-1}$ to $C^0 ([-T,T]; H^s_x)\cap C^1([-T,T];H^{s-1}_x)$;

\vskip0.15cm
{\rm (iii)}\;
Furthermore, there is an additional space $X$ in which $u$ lies, such
that $u$ is the unique solution to the Cauchy problem in $C^0 ([-T,T];
H^s_x)\cap C^1([-T,T];H^{s-1}_x)\cap X$; and
$X\subset L^p_{t,x,\text{loc}}$,
so $|u|^{p-1} u$ is a well-defined spacetime distribution.
\vskip0.15cm

Moreover, if $T=+\infty$, we say that equation \eqref{equ:nlw} is \emph{globally well-posed}  ({\bf GWP}) in $H^s$.
\end{definition}

   \underline{\bf Basic mathematical problems in Nonlinear PDEs}

   \begin{enumerate}
 \item[$\blacksquare$]\; {\bf Wellposedness}:\;Existence, uniqueness, continuous dependence on the data, persistence of regularity.
 The first problem is  to understand this locally in time.

\item[$\blacksquare$] \; {\bf Global behavior}:\;Global existence, or finite time break down(In some norm such as
 $\|\cdot\|_{H^1_x(\R^d)}$, becomes unbounded
in the finite time)?  %: smooth solutions for all times if the data are smooth?

\item[$\blacksquare$]  {\bf Blow-up dynamics}:\;If the solution breaks down in finite time, can we describe the mechanism by which it dose
so?  For example, via energy concentration at the tip of a light cone? Usually, symmetries play a crucial role. How about the qualitative description of singularity formation and blowup rate? Solitary wave conjecture, solitary resolution conjecture.

\item[$\blacksquare$]  {\bf Scattering theory}:\;If the solutions
 exist for all $t\in \mathbb R$, does it approach a free solution?{Critical norm conjecture}.

\item[$\blacksquare$]\; {\bf Special solution}:\;If the solution does not approach  a free solution, does it scatter to something else?
A stationary nonzero solution, for example?  Some physical equations exhibit nonlinear bound states, which represent elementary particles.

\item[$\blacksquare$] \;  {\bf Stability  theory}:\;If special solutions
exist such as stationary or time-periodic solutions, are they orbitally stable?  are they asymptotically stable?

\item[$\blacksquare$]  {\bf Multi-bump solutions}:\;Is it possible to construct solutions which are
asymptotically splited into moving `` solitons "
plus radiation? { Lorentz invariance dictates  the dynamics of single solitons}.

\item[$\blacksquare$]  {\bf Resolution into Multi-bumps}:\;Do all solutions decompose in this fashion?
Suppose solutions exist for all $t\ge0$:\;Either scatter to a free wave, or the energy collects in ``pockets" formed by such ``solitons" ?
Quantization of energy.
 \end{enumerate}

Next, we introduce some basic analysis tools.

\subsection{Behavior of linear solution}\label{subsec:lineasolu}

\begin{lemma}[Dispersive estimate, Shatah-Struwe \cite{SS98}] If $d\geq1,~2\leq
q<\infty$,then
\begin{equation}\label{dispersive-estimate-1}
\Big\|\frac{\sin(t|\nabla|)}{|\nabla|}f\Big\|_{L^q_x(\R^d)}\leq
C|t|^{-(d-1)\big(\frac12-\frac1q\big)}\big\||\nabla|^{\frac{d-3}2-\frac{d+1}q}\nabla
f\big\|_{L^{q'}_x(\R^d)}.
\end{equation}
Especially for $d=$ odd number,  \eqref{dispersive-estimate-1} holds for $2\leq q\leq \infty$.
\end{lemma}
\begin{lemma}[Semi-classical dispersive estimate]\label{lem:semidises}
Let $\psi(r)\in C_c^\infty(\R)$ and ${\rm
supp}~\psi\subset\{r:~1\leq r\leq2\}.$ Then,
\begin{equation}\label{equ:wavedis}
\big\|\psi(h\sqrt{-\Delta})e^{it\sqrt{-\Delta}}\big\|_{L^1_x(\R^d)\to
L^\infty_x(\R^d)}\leq
ch^{-d}\min\Big\{1,\big(\tfrac{h}{t}\big)^{\frac{d-1}2}\Big\}.
\end{equation}
\end{lemma}

\begin{proof}
Denote
$$I_h(t)\triangleq\psi(h\sqrt{-\Delta})e^{it\sqrt{-\Delta}}\delta(x)=h^{-d}\int_{\R^d}
e^{\frac{it}{h}\Phi(x,t,\xi)}\psi(|\xi|)\;d\xi,$$ where
$\Phi(x,t,\xi)=\frac{x}t\cdot\xi+|\xi|$.

It is easy to see that
$$|I_h(t)|\lesssim h^{-d}.$$
Hence, it suffices to consider the case $|t|>h$. Without loss of
generality, we may assume $t>h$.

Since $\nabla_\xi\Phi=\frac{x}{t}+\frac{\xi}{|\xi|},$ we know that
$$|\nabla_\xi\Phi|\geq\frac12, \;\; {\rm when}\;\; |x|\geq 2t\;\; {\rm or}\;\; |x|\leq\tfrac{t}2.$$
In this case, we have by
stationary phase argument
$$|I_h(t)|\lesssim_kh^{-d}\big(\tfrac{t}h\big)^{-k},\quad\forall~k\in\mathbb{N}.$$
If $\frac12t\leq |x|\leq2t$, then, we get
\begin{align*}
|I_h(t)|=&h^{-d}\Big|\int_0^\infty
e^{\frac{it}h\rho}\rho^{d-1}\psi(\rho)\;d\rho\int_{\mathbb{S}^{d-1}}e^{i\frac{x}{h}\cdot
w\rho}\;d\sigma(w)\Big|\\
\lesssim&
h^{-d}\big(\tfrac{|x|}h\big)^{-\frac{d-1}2}\lesssim
h^{-d}\big(\tfrac h{t}\big)^{\frac{d-1}2},\end{align*}
 where we use the fact
$$\int_{\mathbb{S}^{d-1}}e^{i\frac{x}{h}\cdot
w\rho}d\sigma(w)=2\pi\big(\tfrac{|x|\rho}{h}\big)^{-\frac{d-2}2}J_{\frac{d-2}2}\big(2\pi\tfrac{|x|\rho}{h}\big).$$
Thus, we obtain \eqref{equ:wavedis}.

\end{proof}

\underline{\bf Pointwise convergence:} The pointwise convergence problem is  to ask what the minimal $s$ is
to ensure
\begin{equation}\label{equ:converpo}
 \lim_{t\to 0}e^{it\sqrt{-\Delta}}f(x)=f(x),\quad\text{a.e.}\;x,\quad \forall\, f\in H^s(\R^d).
\end{equation}
By a standard process of approximation
and Fatou's lemma, \eqref{equ:converpo} follows from the local maximal inequality
\begin{equation}\label{equ:locmax}
 \Bigl\|\sup_{|t|<1}|e^{it\sqrt{-\Delta}}f|\Bigr\|_{L^2(B(0,1))}\leq C_{s,d}\|f\|_{H^{s}(\R^d)}.
\end{equation}
If $s>\frac{d}{2}$, we obtain \eqref{equ:locmax} immediately from Sobolev's
embedding.  It has been known that \eqref{equ:locmax} holds for $s>\frac12$ in \cite{PSJ,Vega}
and that
this is not true when $s\leq\frac12$ is due to B. G. Walther \cite{Wal}.

 \subsection{Strichartz estimates}\label{sze}

To obtain the well-posedness, the key estimate is the Strichartz estimate.

\begin{definition}[Wave-admissible pair]\label{def:wadpair}
We say that the exponent pair $(q,r)\in\tilde{\Lambda}$ is wave-admissible, if $2\leq q,r\leq\infty$ and
\begin{equation}\label{equ:wapair}
\frac2q\leq(d-1)\Big(\frac12-\frac1r\Big),~(q,r,d)\neq(2,\infty,3).
\end{equation}
If  equality holds in \eqref{equ:wapair}, we say that $(q,r)$ is sharp wave-admissible, otherwise we say that $(q,r)$ is non-sharp wave-admissible.
\end{definition}

Define
\begin{equation}
\beta(r)\triangleq\frac{d+1}2\Big(\frac12-\frac1r\Big),~\delta(r)\triangleq d\Big(\frac12-\frac1r\Big),~\gamma(r)\triangleq(d-1)\Big(\frac12-\frac1r\Big).
\end{equation}

\begin{theorem}[Strichartz estimate]\label{thm;stricest}
Let $d\geq2$ and $u:~I\times\R^d\to\R$ be the solution to $\pa_{tt}u-\Delta u=F$ with initial data $(u,\pa_tu)(0)=(f,g)$. Then, we have
\begin{equation}\label{equ:stricest}
\|u(t,x)\|_{L_t^{q_1}(I,L_x^{r_1}(\R^d))}\lesssim\big\|(f,g)\big\|_{\dot{H}^s\times\dot{H}^{s-1}}+\|F\|_{L_t^{q_2'}(I,L_x^{r_2'})}
\end{equation}
where $(q_1,r_1),~(q_2,r_2)\in\tilde{\Lambda}$  satisfy scaling gap condition
\begin{equation}\label{equ:sgapcon}
\frac1{q_1}+\frac{d}{r_1}=\frac{d}2-s=\frac1{q_2'}+\frac{d}{r_2'}-2.
\end{equation}

In level of frequency, we have for $\forall (q_1,r_1),~(q_2,r_2)\in\tilde{\Lambda}$
\begin{equation*}
\|\Delta_ju\|_{L_t^{q_1}(I,L_x^{r_1}(\R^d))}\leq C2^{j(\mu_1-1)}\big\|(\Delta_jf,\Delta_jg)\big\|_{\dot{H}^1\times L^2}+C2^{j(\mu_{12}-1)}\|\Delta_jF\|_{L_t^{q_2'}(I,L_x^{r_2'})},
\end{equation*}
with
$$\mu_1=\delta(r_1)-\frac1{q_1},\quad \mu_{12}=\delta(r_1)+\delta(r_2)-\frac1{q_1}-\frac1{q_2}.$$
\end{theorem}

\subsection{Finite speed of propagation}

\begin{lemma}[Huygens principal]\label{hgp}

{\rm (i)} Let $d=2k+1,~k\in\mathbb{N}$, then
\begin{equation}\label{odd}
\frac{\sin(t|\nabla|)}{|\nabla|}f(x)=c_d(t^{-1}\pa_t)^{\frac{d-3}{2}}\Big(t^{-1}\int_{|y-x|=t}f(y)d\sigma(y)\Big),
\end{equation}
In particular, when $d=3$, we have
\begin{equation}\label{d3}
\frac{\sin(t|\nabla|)}{|\nabla|}f(x)=\frac1{4\pi
t}\int_{|y-x|=t}f(y)d\sigma(y).
\end{equation}
Let $d=2k,~k\in\mathbb{N}$, then
\begin{equation}\label{even}
\frac{\sin(t|\nabla|)}{|\nabla|}f(x)=c_d(t^{-1}\pa_t)^{\frac{d-2}{2}}\Big(t^{-1}\int_{|y-x|\leq
t}f(y)d\sigma(y)\Big).
\end{equation}
\vskip0.2cm

{\rm (ii)} {\rm(Domain of influence)} Let $d=2k+1,~k\in\mathbb{N}$, then
\begin{equation}\label{yinx}
\frac{\sin(t|\nabla|)}{|\nabla|}\Big(\chi_{|x|\leq
R}f\Big)=\chi_{\big||x|-|t|\big|\leq
R}\frac{\sin(t|\nabla|)}{|\nabla|}\Big(\chi_{|x|\leq R}f\Big).
\end{equation}

\vskip0.2cm

{\rm (iii)} {\rm( Domain of determincy)} Let $d=2k+1,~k\in\mathbb{N}$, then
\begin{equation}\label{jued}
\chi_{|x|\leq
R}\Big(\frac{\sin(t|\nabla|)}{|\nabla|}f\Big)=\chi_{|x|\leq
R}\Big(\frac{\sin(t|\nabla|)}{|\nabla|}\chi_{\big||x|-|t|\big|\leq
R}f\Big).
\end{equation}
\end{lemma}
\begin{proof} (ii) and (iii) are the direct consequence of \eqref{odd}.
\end{proof}

\begin{center}
\includegraphics[width=3in,height=1.75in]{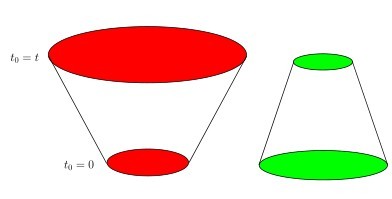}
\end{center}

\vskip0.2cm

\begin{lemma}[Finite propagation speed]\label{fsprop}
Let $u:~I_u\times \R^d\to\R$ and $v:~I_v\times\R^d\to\R$ solve $\phi_{tt}-\Delta\phi+|\phi|^p\phi=0$,
with initial data $(u_0,u_1),~(v_0,v_1)\in\dot{H}^{s_c}\times\dot{H}^{s_c-1}$, respectively, such that for $J\subset\subset I_u\cap
I_v$
$$\|u\|_{L_{tx}^{p_c}(J\times\R^d)}+\|v\|_{L_{tx}^{p_c}(J\times\R^d)}<\infty,~p_c=\frac{d+1}2p.$$
then we have following properties
\begin{enumerate}
\item  If initial data satisfies
$$(u_0,u_1)=(v_0,v_1),~x\in B(x_0,M),$$
then  for any  $t\in J$,
$$u(t,\cdot)= v(t,\cdot),~~x\in B(x_0,M-|t|),\;\;|t|\leq M.$$

\item If initial data satisfies
$$(u_0,u_1)=(v_0,v_1),~x\in B(x_0,M)^c,$$
then
$$u(t,\cdot)= v(t,\cdot),~x\in B(x_0,M+|t|)^c, \;\; \forall\; t\in J.$$
\end{enumerate}
where $(p_c,p_c)\in\tilde\Lambda$ with
$$\delta(p_c)-\frac1{p_c}=s_c\triangleq\frac{d}2-\frac2p.$$
\end{lemma}

%\begin{figure}
%  % Requires \usepackage{graphicx}
%  \includegraphics[width=0.4\columnwidth]{comp.jpg}\\
%\end{figure}

\subsection{Morawetz estimate}

\begin{proposition}[Morawetz estimate]\label{prop:Morawetz}
Let $d\geq3$,  and  u be a solution to \eqref{equ:nlw} with $\mu=-1$  on a spacetime slab $I\times
\mathbb{R}^d$. Then
\begin{eqnarray}\label{morawetzes}
\int_I\int_{\R^d}\frac{~~|u|^{p+1}}{|x|} dxdt\leq C(E),
\end{eqnarray}
where  $E$ is the energy $E(u_0,u_1)$.
\end{proposition}

\subsection{Linear profile decomposition}

Let
$$S(t)(f,g):=\cos(t|\nabla|)f
+\frac{\sin(t|\nabla|)}{|\nabla|}g.$$
Then the operator  $S(t):~\dot{H}^1\times L_x^2(\R^d)\to L_{t,x}^\frac{2(d+1)}{d-2}(\R\times\R^d)$
is not compact. This fact shows the linear profile decomposition  will be useful in the critical or focusing problem.

\begin{theorem}[Linear profile decomposition in energy space, \cite{BG,Bu2010}]\label{thm:inepro}
Let $(u_{0,n},u_{1,n})_{n\in\mathbb{N}}$ be a bounded sequence in $\dot{H}^1\times L^2(\mathbb{R}^d)$ with $d\geq 3$.  Then there exists a subsequence of $(u_{0,n},u_{1,n})$ {\rm (still denoted by $(u_{0,n},u_{1,n})$)}, a sequence $(V_0^j,V_1^j)_{j\in\mathbb{N}}\subset \dot{H}^1\times L^2(\mathbb{R}^d)$ and a sequence of triples $(\lambda_n^j,x_n^j,t_n^j)\in \mathbb{R^+}\times\mathbb{R}^d\times\mathbb{R}$
such that for every $j\neq j'$,
\begin{align}
\label{lab5}\frac{\lambda_n^j}{\lambda_n^{j'}}+\frac{\lambda_n^{j'}}{\lambda_n^j}+\frac{|t_n^j-t_n^{j'}|}{\lambda_n^j}
+\frac{|x_n^j-x_n^{j'}|}{\lambda_n^j}\mathop{\longrightarrow}_{n\rightarrow\infty}
\infty,
\end{align}
and for every $\ell\geq 1$, if $V^j=S(t)(V_0^j,V_1^j)$ and
\begin{align}
\label{lab6}
V_n^j(t,x)\triangleq\frac{1}{(\lambda_n^j)^\frac{d-2}{2}}V^j\left(\frac{t-t_n^j}{\lambda_n^j},\frac{x-x_n^j}{\lambda_n^j}\right),\end{align}
with
\begin{equation}\left\{\begin{aligned}
&u_{0,n}(x)=\sum_{j=1}^\ell V_n^j(0,x)+w_{0,n}^\ell(x),\\
&u_{1,n}(x)=\sum_{j=1}^\ell \partial_t V_n^j(0,x)+w_{1,n}^\ell(x),
\end{aligned}\right.\label{lab7}\end{equation}
satisfying
\begin{align}
\label{lab8}\limsup_{n\rightarrow\infty}\lVert S(t)(w_{0,n}^\ell,w_{1,n}^\ell)\rVert_{L_{t,x}^{\frac{2(d+1)}{d-2}}}\mathop{\longrightarrow}_{\ell\rightarrow\infty}0,
\end{align}
where $S(t)(u_0,u_1)$ denotes the solution of the free wave equation with initial data $u(0)=u_0$ and $\partial_t u(0)=u_1$.
For every $\ell\geq 1$, we also have
\begin{align}
\|u_{0,n}\|_{\dot{H}^1}^2+\|u_{1,n}\|_{L^2}^2=\sum_{j=1}^\ell &\left(\|V^j_0\|_{\dot{H}^1}^2+\|V^j_1\|_{L^2}^2\right)
+\|w_{0,n}^\ell\|_{\dot{H}^1}^2\nonumber\\
&+\|w_{1,n}^\ell\|_{L^2}^2+o(1),\quad n\rightarrow\infty,\label{lab9}
\end{align}
and, for $j\neq k$,
\begin{align}
\label{lab11}\lim_{n\rightarrow\infty} \lVert V_n^jV_n^k\rVert_{L_{t,x}^{\frac{d+1}{d-2}}}=0.
\end{align}
\end{theorem}

Applying the profile decomposition to an appropriate sequence, one can obtain the existence of maximizers for Strichartz estimate.
\begin{theorem}[Existence of maximizers for Strichartz estimate, \cite{Bu2010}]\label{thm:emaxstr}
Let $d\geq 3$, $(q,r)$ be a wave admissible pair with $q,r\in (2, \infty)$ and satisfying the $\dot{H}^1$-scaling condition: $\frac1q+\frac{d}r=\frac{d}2-1.$  Then there exists a maximizing pair $(\phi,\psi)\in \dot{H}^1\times L^2(\mathbb{R}^d)$ such that \begin{align*}\lVert S(t)(\phi,\psi)\rVert_{L_t^qL_x^r(\mathbb{R}\times\mathbb{R}^d)}=C_{q,r}\,\lVert (\phi, \psi)\rVert_{\dot{H}^1\times L^2(\mathbb{R}^d)}\end{align*} where
\begin{align*} C_{q,r}\triangleq\sup \big\{\lVert S(t)(\phi,\psi)\rVert_{L_t^qL_x^r}:(\phi,\psi)\in \dot{H}^1\times L^2,\,\, \lVert (\phi, \psi)\rVert_{\dot{H}^1\times L^2(\mathbb{R}^d)}=1\big\}
\end{align*}
is the sharp constant.
\end{theorem}

\begin{theorem}[Linear profile decomposition in $\dot{H}^s\times\dot{H}^{s-1}$ with $s\geq1$, \cite{Bu2010}]
\label{thm:linprohs}
Let $s\geq 1$ be given and let $(u_{0,n},u_{1,n})_{n\in\mathbb{N}}$ be a bounded sequence in $\dot{H}^s\times \dot{H}^{s-1}(\mathbb{R}^d)$ with $d\geq 3$.  Then there exists a subsequence of $(u_{0,n},u_{1,n})$ (still denoted $(u_{0,n},u_{1,n})$), a sequence $(V_0^j,V_1^j)_{j\in \mathbb{N}}\subset \dot{H}^s\times\dot{H}^{s-1}(\mathbb{R}^d)$, and a sequence of triples $(\lambda_n^j,x_n^j,t_n^j)\in \mathbb{R^+}\times\mathbb{R}^d\times\mathbb{R}$ such that for every $j\neq j'$,
\begin{align*}
\frac{\lambda_n^j}{\lambda_n^{j'}}+\frac{\lambda_n^{j'}}{\lambda_n^j}+\frac{|t_n^j-t_n^{j'}|}{\lambda_n^j}
+\frac{|x_n^j-x_n^{j'}|}{\lambda_n^j}\mathop{\longrightarrow}_{n\rightarrow\infty}
\infty,
\end{align*}
and for every $\ell\geq 1$, if $V^j=S(t)(V_0^j,V_1^j)$ and
\begin{align}
\label{lab15}
V_n^j(t,x)\triangleq\frac{1}{(\lambda_n^j)^{\frac{d-2}{2}-(s-1)}}V^j\left(\frac{t-t_n^j}{\lambda_n^j},\frac{x-x_n^j}{\lambda_n^j}\right),
\end{align}
\begin{equation}\left\{\begin{aligned}
&u_{0,n}(x)=\sum_{j=1}^\ell V_n^j(0,x)+w_{0,n}^\ell(x),\\
&u_{1,n}(x)=\sum_{j=1}^\ell \partial_t V_n^j(0,x)+w_{1,n}^\ell(x),\end{aligned}\right.
\label{lab16}
\end{equation}
with
\begin{align}
\limsup_{n\rightarrow\infty}\lVert S(t)(w_{0,n}^\ell,w_{1,n}^\ell)\rVert_{L_t^qL_x^r}\mathop{\longrightarrow}_{\ell\rightarrow\infty}0\label{lab17}
\end{align}
for every $(q,r)$ a wave admissible pair with $q,r\in (2,\infty)$ and satisfying the $\dot{H}^s$-scaling condition. For every $\ell\geq 1$, we also have
\begin{align}
\|u_{0,n}\|_{\dot{H}^s}^2+\|u_{1,n}\|_{\dot{H}^{s-1}}^2=\sum_{j=1}^\ell &\left(\|V^j_0\|_{\dot{H}^s}^2+\|V^j_1\|_{\dot{H}^{s-1}}^2\right)
+\|w_{0,n}^\ell\|_{\dot{H}^s}^2\nonumber\\
&+\|w_{1,n}^\ell\|_{\dot{H}^{s-1}}^2+o(1),\quad n\rightarrow\infty.\label{lab18}
\end{align}
\end{theorem}

\begin{theorem}[Linear profile decomposition in $\dot{H}^{s_c}$, $s_c=\frac12$, Ramos \cite{Ramos}]\label{thm:lpds1}

Let $\{(u_{0,n},u_{1,n})\}_n$  be a bounded sequence in $\dot{H}^\frac12\times\dot{H}^{-\frac12}(\R^d)$ with $d\geq2$. Then, there exist a subsequence  (still denoted by $\{(u_{0,n},u_{1,n})\}_n$), a sequence
$(\phi_0^j,\phi_1^j)_{j\in\N}\subset\dot{H}^\frac12\times\dot{H}^{-\frac12}$ and a family of orthogonal sequences
 $\{(r_j^n,l_j^n,w_j^n,x_j^n,t_j^n)_{n\in\mathbb N}\}_{j\in\mathbb N}$ in $\R^+\setminus\{0\}\times[1,+\infty)\times\mathbb{S}^{d-1}\times\R^d\times\R$, such that for every $N\geq1$,
\begin{equation}\label{equ:lineqdecomp}
S(t)(u_{0,n},u_{1,n})=\sum_{j=1}^N\Gamma_j^nS(t)(\phi_0^j,\phi_1^j)+S(t)(R_{0,n}^N,R_{1,n}^N),
\end{equation}
with
\begin{equation}\label{equ:reamturnzer}
\lim_{N\to\infty}\limsup_{n\to\infty}\big\|S(t)(R_{0,n}^N,R_{1,n}^N)\big\|_{L_{t,x}^\frac{2(d+1)}{d-1}}=0,
\end{equation}
where
$$\Gamma_j^nF(x,t):=\Big(\frac{r_j^n}{l_j^n}\Big)^\frac{d-1}{2}F\big((T_{w_j^n}^{l_j^n})^{-1}r_j^n(x-x_j^n,t-t_j^n)\big),$$
and $T_{w_j^n}^{l_j^n}$ is the rescaled Lorentz transformation defined as
\begin{align*}
&T_w^{2^j}(w,1)=(w,1),\\
&T_w^{2^j}(w,-1)=2^{2j}(2,-1),\\
&T_w^{2^j}(x,t)=2^j(x,t)\quad \text{if}\;(x,t)\in\R^d\;\text{is orthogonal to}\;(w,1)\;\text{and}\;(w,-1).
\end{align*}
Furthermore, we also have for every $N\geq1$,
\begin{align}
\|u_{0,n}\|_{\dot{H}^\frac12}^2+\|u_{1,n}\|_{\dot{H}^{-\frac12}}^2=\sum_{j=1}^\ell &\left(\|\phi^j_0\|_{\dot{H}^\frac12}^2+\|\phi^j_1\|_{\dot{H}^{-\frac12}}^2\right)
+\|R_{0,n}^\ell\|_{\dot{H}^\frac12}^2\nonumber\\
&+\|R_{1,n}^\ell\|_{\dot{H}^{-\frac12}}^2+o(1),\quad n\rightarrow\infty.\label{lab18123}
\end{align}

\end{theorem}

\section{Main result}

\noindent\underline{\bf History of defocusing energy-subcritical NLW: $\mu=-1,~s_c<1$.}
\vskip0.2cm

For $d=3$ and $1<p<p_c=5$, in 1961, J\"orgen \cite{Jor} proved the global existence of {\bf smooth solution}.
For higher dimensional case such as $4\leq d\leq9$,  Brenner-Wahl and Pecher
 proved the global existence of {\bf smooth solution}, see \cite{BW,Pech,Wa}.
 Ginibre-Velo proved the global well-posedness of {\bf finite enengy solution}  with initial data $(u_0,u_1)\in H^1\times L^2$
 by  the compactness argument in \cite{GV1,GV2}.

\vskip0.15cm
\begin{remark}[Sketch proof of global existence of smoothing solution]\label{rem:glsmsub}   We take $d=3$ for example. For the detailed proof, we refer reader to Chapter 3 in Miao \cite{MC}.

\vskip0.15cm

{\bf Step 1:}  By the finite propagation
speed,  we can  reduce the above problem to the problem
with compact supported data $(u_0(x), u_1(x))$, i.e
$${\rm supp}~u_0, {\rm supp}~u_1\subset\{x:~|x|\leq R\}.$$

\vskip0.15cm
{\bf Step 2:} By local theory, if $T^\ast\triangleq\sup I<+\infty,$ then $u(t,x)\not\in L_{t,x}^\infty([0,T^\ast)\times\R^3)$. This can be further reduced to
 $$T^\ast<+\infty\Longrightarrow\|u(t,x)\|_{L_t^4([0,T^\ast);L_x^{12}(\R^3))}=+\infty.$$
 Hence, we are reduced to show
 \begin{equation}\label{equ:defsubclaim}
 u(t,x)\in L_t^4L_x^{12}([0,T^\ast)\times\R^3).
 \end{equation}
This follows from energy conservation, Strichartz estimate and continuous argument. In fact,  assume $T^\ast<+\infty$, then
for $0\leq t_0<s<T^\ast$, we have by the Strichartz estimate and H\"older inequality
\begin{align}
&\|u\|_{L_t^4L_x^{12}([t_0,s]\times\R^3)}\lesssim\|(u_0,u_1)\|_{\dot{H}^1\times L^2}+\big\||u|^{p-1}u\big\|_{L_t^1L_x^2([t_0,s]\times\R^3)}\nonumber\\
\leq&CE(u_0,u_1)^\frac12+C\|u\|_{L_t^\frac4{5-p}L_x^\frac{12}{7-p}([t_0,s]\times\R^3)}\|u\|_{L_t^4L_x^{12}([t_0,s]\times\R^3)}^{p-1}.\label{equ:l4l12}
\end{align}
On the other hand, by the finite propagation speed: ${\rm supp} u(t,x)\subset\{x:~|x|\leq t+R\},$ we obtain
\begin{align*}
\|u\|_{L_t^\frac4{5-p}L_x^\frac{12}{7-p}([t_0,s]\times\R^3)}\leq&(T^\ast-t_0)^\frac{5-p}4\sup_{t_0\leq t\leq s}\|u(t)\|_{L_x^\frac{12}{7-p}}\\
\lesssim&(T^\ast-t_0)^\frac{5-p}4(T^\ast+R)^{3(\frac{7-p}{12}-\frac1{p+1})}\|u(t)\|_{L_x^{p+1}}\\
\leq&\rho(R,T^\ast,E_0)(T^\ast-t_0)^\frac{5-p}4.
\end{align*}
Hence,
$$\|u\|_{L_t^4L_x^{12}([t_0,s]\times\R^3)}\leq CE(u_0,u_1)^\frac12+C\rho(R,T^\ast,E_0)(T^\ast-t_0)^\frac{5-p}4\|u\|_{L_t^4L_x^{12}([t_0,s]\times\R^3)}^{p-1}.$$
Note that
 $$\lim\limits_{t_0\to T^\ast}\rho(R,T^\ast,E_0)(T^\ast-t_0)^\frac{5-p}4=0, \;\;\; {\rm for}\;\;  p<5,$$
  we obtain \eqref{equ:defsubclaim} by the continuous argument.
\end{remark}

\noindent\underline{\bf History of energy-critical NLW: $s_c=1$.}

\vskip0.2cm

For defocusing case $\mu=-1$,
\begin{equation}\label{equ:nlwdefcri}
{\rm (DNLW)}\;\;\begin{cases}
u_{tt}-\Delta u+|u|^\frac4{d-2}u=0,~(t,x)\in\R\times\R^d,~d\geq3,\\
(u,u_t)(0)=(u_0,u_1).
\end{cases}
\end{equation}
\vskip-0.25cm
\begin{itemize}\rm
\item For $d=3$ and $p_c=5$, in 1988, Rauch \cite{Rauch} proved the global existence of {\bf radial smooth solution}
for (DNLW).  In 1990 Grillakis \cite{Gri90} remove radial assumption and proved the global existence of {\bf smooth solution}
for (DNLW) by the classical Morawetz estimate.
\item  For higher dimensional case such as $3\leq d\leq5$, Grillakis \cite{Gri92}
 proved the global existence of {\bf smooth solution} for (DNLW) by combining Strichartz estimates with the classical Morawetz estimate.
\item  Shatah-Struwe \cite{ShaStr94}  proved the global existence of {\bf smooth solution}
for (DNLW) for $d\le 7$. Moreover, they  proved the global existence of {\bf finite energy  solution} in $H^1\times L^2$
for (DNLW) in  \eqref{equ:nlwdefcri}.
\end{itemize}
\vskip0.2cm

The scattering result is obtained by
Bahouri-G\'{e}rard \cite{BG}, and Tao \cite{T07}. In particular, Tao in \cite{T07} derived a
exponential type spacetime bound as follows
$$ \|u\|_{L^{4}_tL^{12}_x(\R\times\R^3)}\leq C(1+E_0)^{CE_0^{105/2}},~E_0=E(u_0,u_1).$$

\begin{theorem}[Defocusing energy-critical NLW]\label{thm:defenrcri} Let $d\geq3$. Given $(u_0,u_1)\in\dot H^1(\R^d)\times L^2(\R^d)$. Then, there is a unique global strong solution $u$ to
\begin{equation}\label{equ:nlwedef}
\begin{cases}
u_{tt}-\Delta u+|u|^\frac4{d-2}u=0,\\
(u,\pa_tu)(0)=(u_0,u_1).
\end{cases}
\end{equation}
Moreover, the solution $u$  obeys the estimate
\begin{equation}\label{Gophergoal12}
\int_{\R}\int_{\R^d}|u(t,x)|^{\frac{2(d+1)}{d-2}}\,dx\,dt\leq C( \|(u_0,u_1)\|_{\dot H^1_x\times L^2_x} ).
\end{equation}
And so the solution scatters in the sense, there exists $(u_0^\pm,u_1^\pm)\in\dot H^1(\R^d)\times L^2(\R^d)$ such that
\begin{equation}\label{equ:scate}
\lim_{t\to\pm\infty}\big\|u(t,x)-S(t)(u_0^\pm,u_1^\pm)\big\|_{\dot H^1(\R^d)\times L^2(\R^d)}=0,
\end{equation}
where
$$S(t)(f,g)=\cos(t\sqrt{-\Delta})f+\frac{\sin(t\sqrt{-\Delta})}{\sqrt{-\Delta}}g.$$
\end{theorem}

\begin{enumerate}

\item Since there is no  {\bf pointwise criteria} for the
critical problem, it  seems that GWP and
{scattering result} are
simultaneously solved. However,
{the study history of the $\dot H^1$-critical wave equation} shows us  scattering result
is later than global well-posedness!

\item Recently, we give another proof by the concentration-compactness approach  to induction on energy,
where we show the GWP and scattering simultaneously   in \cite{MZZ}.

 \end{enumerate}

\underline{\bf Outline of proof:} ( For detail, see Section \ref{sec:defenecri}).\;

\vskip0.2cm

 By the {\bf finite propagation speed} of NLW, we can
consider the Cauchy problem with compact data.
\vskip0.15cm

\begin{enumerate}
\item  By ruling out the
accumulation of the energy at any time to show the existence of the
global smooth solution.

\vskip0.15cm

\item By {compactness argument}, we can show the global
existence and uniqueness of the energy solution.

\vskip0.15cm

\item By FPS and conformal scaling transform, we can prove the
{scattering result} of the energy solution.
\end{enumerate}

\begin{remark}
The scattering theory with {radial data} is more
easier than the general case because the {\bf energy
accumulation} may only occur near the origin in this case. Indeed, by Morawetz estimate
\begin{align}\label{ClassMoraEst}
\iint_{\R\times \R^d} \frac{|u(t,x)|^{2^*}}{|x|} \; dxdt \leq
C(E)<\infty,\quad 2^\ast=\frac{2d}{d-2}
\end{align}
and {\bf radial Sobolev embedding inequality}
\begin{align}\label{RadSobEmb}
\big\||x|^{\frac{d}{2}-1}u\big\|_{L^{\infty}_x(\R^d)} \leq
\big\|u\big\|_{\dot H^1(\R^d)},
\end{align}
 we have for
{$r_0=\frac{2(d+1)}{d-2}$}
\begin{align*}
\iint_{\R\times\R^d} |u(t,x)|^{r_0} \; dxdt \leq & \iint_{\R\times \R^d} \frac{|u(t,x)|^{2^*}}{|x|} \cdot |x||u|^{\frac{2}{d-2}} \; dxdt \\
\leq &  \big\|u\big\|^{\frac{2}{d-2}}_{\dot H^1(\R^d)} \iint_{\R
\times\R^d} \frac{|u|^{2^*}}{|x|} \; dxdt  \leq  C(E) < \infty,
\end{align*}
which implies that
\begin{align*}
\big\|u\big\|_{L^{q_0}_tL^{r_0}_{x}(\R\times \R^d)}<\infty,\qquad q_0=r_0=\frac{2(d+1)}{d-2}
\end{align*} with
\begin{align*}
\frac{2}{q_0} \leq \gamma(r_0), \quad \delta(r_0) -\frac{1}{q_0}=1 ,
\quad
 \frac{\delta(r_0)}{d} = \frac{\gamma(r_0)}{d-1} =
\frac12 - \frac{1}{r_0}.
\end{align*}

\end{remark}

We refer the reader to the proof of the {\bf nonradial case} in Section \ref{sec:defenecri}.

\vskip 0.2in

The dynamics of the focusing energy-critical equation
\begin{align} \label{equ:nlwecrifoc}
(\text{FNLW})\;\; \begin{cases}    \partial_{tt}u-\Delta u= |u|^{\frac4{d-2}}u,\quad
(t,x)\in\R\times\R^d,
\\
(u,\pa_tu)(0,x)=(u_0,u_1)(x),
\end{cases}
\end{align}is much richer: small data solutions are global and scatter,
however blow-up in finite time may occur  \cite{Le74}, where Levine  showed  that if
$$(u_0, u_1)\in\dot{H}^1\times L^2, \;\; u_0 \in L^2,\;\;  E(u_0, u_1)<0,$$
 there is always break-down in finite time.  Recently, Krieger-Schlag-Tataru \cite{KST09} have constructed explicit radial examples which break down in finite time. On the other hand, it is well-known that there exist solutions with compactly supported smooth initial data which
blow up in finite time.
This is most easily seen by observing that
\[ u(t,x)=\big(\tfrac{d(d-2)}{4}\big)^\frac{d-2}4 \big/(1-t)^{\frac{d-2}2} \]
is an explicit solution which, by finite speed of propagation, can be used to construct a blowup
solution of the aforementioned type.
This kind of breakdown is referred to as \emph{\bf ODE blowup} and it is conjectured to comprise the
``generic'' blowup
scenario \cite{BCT04}.

\begin{remark}
{\bf About the finite blow up solution:} Let $Q$ satisfy
$$-\Delta Q+Q=|Q|^{\frac4d}Q.$$
For the mass-critical focusing Schr\"odinger equation $i\pa_tu+\Delta u+|u|^\frac4du=0,$  the
pseudo-conformal transformation applied to the stationary solution $e^{it}Q$  yields an explicit blowup solution
\begin{equation}\label{equ:stx}
S(t,x)=\frac1{|t|^{d/2}}Q\Big(\frac{x}{t}\Big)e^{-i\frac{|x|^2}{4t}+\frac{i}{t}},~\;\;~\|S(t)\|_{L^2}=\|Q\|_{L^2}
\end{equation}
which scatters as $t\to-\infty$ since
 $$\|S(t,x)\|_{L_{t,x}^\frac{2(d+2)}d((-\infty,-1)\times\R^d)}^\frac{2(d+2)}d=\int_{-\infty}^{-1}|t|^{-2}\;dt
\|Q\|_{L_{x}^\frac{2(d+2)}{d}}^\frac{2(d+2)}d<+\infty,$$
 and  blows up at $T=0$ at the speed
 $$\|\nabla S(t)\|_{L^2}\sim\frac1{|t|}.$$
 An essential feature of \eqref{equ:stx} is compact up to the symmetries of the flow,which shows
that all the mass goes into the singularity formation
  \begin{equation}\label{equ:stdels0}
  |S(t)|^2\rightharpoonup\|Q\|_{L^2}^2\delta_{x=0}~\text{as}~t\to0.
  \end{equation}
\end{remark}

  \underline{\bf Non-scattering global solutions also exist}. \; Examples of such solutions  are the so-called
solitary waves and are  given by solutions of the elliptic equation:
\begin{equation}
 \label{Ell}
 -\Delta W=|W|^{\frac{4}{d-2}}W,\quad W\in \dot{H}^{1},
\end{equation}
 (see \cite{Ding86} for the existence of such solutions),
 \begin{equation}
\label{defW}
W(x)\triangleq\frac{1}{\left(1+\frac{|x|^2}{d(d-2)}\right)^{\frac{d-2}{2}}}\in  \dot{H}^1,
\end{equation}
 but in $L^2$ only if $d\geq 5$. The works of Aubin and Talenti \cite{Au76,Ta76}, give the following elliptic characterization of $W(x)$
\begin{align}
 &\|u\|_{L^{2^*}}\leq C_d\|u\|_{\dot{H}^1}, \quad\;\forall u\in\dot{H}^1;\label{SobolevIn}\\
\label{CarW}
&\|u\|_{L^{2*}}=C_d\|u\|_{\dot{H}^1}\Rightarrow \exists \;\lambda_0,x_0,\theta_0\quad u(x)=e^{i\theta_0}\lambda^{\frac{d-2}2}W\big(\lambda_0(x+x_0)\big)
\end{align}
where $C_d$ is the best Sobolev constant in dimension $d$.

\subsection{Below threshold of ground state}

\begin{theorem}[Scattering/Blowup dichotomy, Kenig-Merle\cite{KM1}]\label{thm:kmw08} Assume $3\leq d\leq5,~(u_0, u_1)\in\dot H^1\times L^2$ with $$E(u_0,u_1)<E(W,0).$$ Then
 \begin{enumerate}
 \item If $\|u_0\|_{\dot H^1}<\|W\|_{\dot H^1},$ then the solution $u$ of \eqref{equ:nlwecrifoc} is global and scatters.
 \vskip0.15cm

 \item If $\|u_0\|_{\dot H^1}>\|W\|_{\dot H^1},$ then the solutions $u$ of \eqref{equ:nlwecrifoc} blows up in finite time in both directions.
 \end{enumerate}
\end{theorem}

\begin{remark}\label{rem:scablw}
{\rm (i)} The condition  $\|u_0\|_{\dot{H}^1}=\|W\|_{\dot{H}^1}$ is not compatible with $E(u_0,u_1)<E(W,0).$
 \vskip0.15cm

{\rm (ii)}   Bulut et.al \cite{BCLPZ} extended the above result to higher dimensions.
  \vskip0.15cm

 \vskip0.15cm
{\rm (iii)} \underline{\bf Sketch proof of scattering part:} We proceed ``concentration compactness"  and profile decomposition developed by  Bahouri-
G\'{e}rard \cite{BG} for $d=3$ (see also Bulut\cite{Bu2010} for $d\geq4$).
 Thus, arguing by contradiction, we find a number $E_c$, with $0<\eta_0\leq
E_c<E(W, 0)$ with the property that if
$$E(u_0, u_1)<E_c, \;~~\; \|\nabla u_0\|_2<\|\nabla W\|_2,$$
then $\|u\|_{S(I)}<\infty$. And $E_c$ is optimal with this property. We will see that this leads to
a contradiction.

 \vskip0.15cm
{\bf Step 1: The existence of critical element with compactness}.  There exists
$$(u_{0,c},u_{1,c})\in\dot H^1\times L^2,~\|\nabla u_{0,c}\|_2<\|\nabla W\|_2,~ E(u_{0,c},u_{1,c})=E_c, $$
such that for the corresponding solution $u_c(t)$, we have
$$\|u_c\|_{L_{t,x}^\frac{2(d+1)}{d-2}(I_+\times\R^d)}=+\infty,~I_+=I\cap[0,+\infty).$$
Moreover, there exist $x(t),~t\in I_+,~\la(t)\in\R^+,~t\in I_+,$ such that
\begin{equation*}
K=\bigg\{\vec{v}(x,t)=\Big(\frac1{\la(t)^\frac{d-2}2}u\Big(\frac{x-x(t)}{\la(t)},t\Big),
\frac1{\la(t)^\frac{d}2}\pa_tu\Big(\frac{x-x(t)}{\la(t)},t\Big)\Big),~t\in I_+\bigg\}
\end{equation*}
has compact closure in $\dot H^1\times L^2.$

 \vskip0.15cm
{\bf Step 2: Rigidity theorem.} $T_+<\infty$ is impossible, and if $T_+=+\infty$, then $u\equiv0$.

Please refer to  Section \ref{sec:focener} for complete proof .

\end{remark}

\begin{remark}\small\rm
{\rm (i)}\;  Kenig-Merle in \cite{KM1} claimed: Assume that $u$ is a radial solution or  $d\in\{3,4\}$ such that
\begin{equation}\label{eque1.3}
\sup_{t\in[0,T_+)}\Big[\|\nabla u\|_2^2+\|\pa_t
u\|_2^2\Big]<\|\nabla W\|_2^2,
\end{equation}
then $T_+=+\infty$ and $u$ is scatter forward in time. But when $d=5$, the condition \eqref{eque1.3}
 does not suffice to ensure the solution of (FNLW) scatter forward in time. Please refer the sharp
 condition in (ii).

 {\rm (ii)}\; In \cite{DKM12JEMS}, Duyckaerts-Kenig-Merle prove that:
assume $d\in\{3,4,5\}$.
Let $u$ be a solution of (FNLW) which satisfies
\begin{equation}
\label{bound_nabla_corol}
\limsup_{t\to T_+(u)} \big[\|\nabla u(t)\|_{L^2}^2+\tfrac{d-2}{2}\|\partial_tu(t)\|_{L^2}^2\big]<\|\nabla W\|_{L^2}^2.
\end{equation}
Then $T_+(u)=+\infty$ and $u$ scatters forward in time. If $u$ is radial, \eqref{bound_nabla_corol} can be replaced by the following bound
\begin{equation}
\label{bound_nabla_corol_radial}
\limsup_{t\to T_+(u)} \|\nabla u(t)\|_{L^2}^2<\|\nabla W\|_{L^2}^2.
\end{equation}
\end{remark}

\subsection{Scattering norm estimate}

Consider
\begin{align} \label{equ:nlwecri123}
(\text{FNLW})\;\;\begin{cases}    \partial_{tt}u-\Delta u= |u|^{\frac4{d-2}}u,\quad
(t,x)\in\R\times\R^d,
\\
(u,\pa_tu)(0,x)=(u_0,u_1)(x),
\end{cases}
\end{align}
where $u:\R_t\times\R_x^d\to \R$.

In Theorem \ref{thm:kmw08}, Kenig-Merle \cite{KM1} have described the dynamics of \eqref{equ:nlwecri123} below the energy threshold $E(W,0)$.
Theorem \ref{thm:kmw08} implies that for $\eps>0$ the following supremum is finite:
$$ \mathcal{I}_{\eps}=\sup_{u\in F_{\eps}} \int_{\R\times \R^d}|u(t,x)|^{\frac{2(d+1)}{d-2}}\,dtdx
=\sup_{u\in F_{\eps}} \|u\|_{S(\R)}^{\frac{2(d+1)}{d-2}}<\infty,$$
where
\begin{align*}
F_{\eps}\triangleq\bigg\{u:\;\; \text{ solution of }\eqref{equ:nlwecri123}\text{ such that }& E(u_0,u_1)\leq E(W,0)-\eps^2 \\
 \text{ and }&\int |\nabla u_0|^2<\int |\nabla W|^2\bigg\}.\end{align*}
Furthermore, the existence of the non-scattering solution $W$ at the energy threshold shows that
\begin{equation*}
\lim_{\eps\rightarrow 0^+} \mathcal{I}_{\eps}=+\infty.
\end{equation*}

 Consider the negative eigenvalue $-\omega^2$ ($\omega>0$) of the linearized operator associated to \eqref{equ:nlwecri123} around $W$:
$$-\omega^2=\inf_{\substack{u\in H^1,\\ \|u\|_2=1}} \Big[\int_{\R^d}|\nabla u|^2dx-\frac{d+2}{d-2}\int_{\R^d} W^{\frac{4}{d-2}}|u|^2dx\Big].$$
 Then
\begin{theorem}[Scattering norm estimate, Duyckaerts-Merle \cite{DuMe10}]\label{thm:scnorest} Let $d\in\{3,4,5\}$. Then, we have
$$\lim_{\eps\rightarrow 0^+} \frac{\mathcal{I}_{\eps}} {|\log \eps|}=\frac{2}{\omega}\int_{\R^d} W^{\frac{2(d+1)}{d-2}}dx.$$
\end{theorem}

\begin{remark}
{\rm (i)} It would be interesting to get an explicit value of the limit
$$\frac{2}{\omega}\int_{\R^d} W^{\frac{2(d+1)}{d-2}}dx.$$
A straightforward computation gives:
\begin{align*}
\int_{\R^d} W^{\frac{2(d+1)}{d-2}}\;dx&=\frac{(d(d-2))^{\frac{d}{2}}}{2^{2d+1}}\times \frac{d!}{\left((\frac d2)!\right)^2}\times \pi&&\text{ if }d\text{ is even,}\\
\int_{\R^d} W^{\frac{2(d+1)}{d-2}}\;dx&=\frac{(d(d-2))^{\frac{d}{2}}}{2}\times \frac{\left(\frac{d-1}{2}\right)!}{d!}&&\text{ if }d\text{ is odd}.
\end{align*}
However we do not know any explicit expression of $\omega$.

{\rm (ii)}  The proofs rely on the compactness argument of Kenig-Merle, on a classification result
at the energy level $E(W,0)$, and on the analysis of the linearized equation around $W$.

{\rm (iii)}  {\bf It seems to be possible to extend the above result to higher dimensions $d\geq6$}.
 The authors in \cite{DuMe10} also obtain the similar result for NLS,  how about the Klein-Gordon equation? and $\varepsilon\to0-?$
\end{remark}

\begin{prop}
\label{P:toW}
Let $u_n$ be a family of solutions of {\rm (FNLW)} such that
\begin{equation}
\label{below_threshold}
E\big(u_n(0),\partial_tu_n(0)\big)<E(W,0),\quad \|\nabla u_n(0)\|_2<\|\nabla W\|_2.
\end{equation}
and $\lim\limits_{n\rightarrow +\infty} \|u_n\|_{S(\R)}=+\infty$. Let
$(t_{n})_n$ be a time sequence.
\begin{enumerate}
\item \label{C:toW} Assume that
$$\lim_{n\rightarrow +\infty} \|u_n\|_{S(-\infty,t_n)}=\lim_{n\rightarrow +\infty} \|u_{n}\|_{S(t_n,+\infty)}=+\infty.$$
Then, up to the extraction of a subsequence there exist $\iota_0\in \{-1,+1\}$ and sequences of parameters $x_n\in \R^d$, $\lambda_n>0$ such that
\begin{equation*}
%\label{uepstoW}
\lim_{n\rightarrow +\infty}
\left\|\frac{\iota_0}{\lambda_n^{d/2}}\nabla u_{n}\left(t_{n},\frac{\cdot-x_n}{\lambda_n}\right)-\nabla  W\right\|_{2}+\left\|\frac{\partial u_{n}}{\partial t}\left(t_{n}\right)\right\|_{2}=0.
\end{equation*}
\item \label{C:toW'} Assume that there exists $C_0\in (0,+\infty)$ such that
$$\lim_{n\rightarrow +\infty} \|u_n\|_{S(-\infty,t_n)}=+\infty \text{ and }\lim_{n\rightarrow +\infty} \|u_{n}\|_{S(t_n,+\infty)}=C_0.$$
Then, up to the extraction of a subsequence there exist $t_0\in \R$, $\iota_0,\iota_1\in\{-1,+1\}$, and sequences of parameters $x_n\in \R^d$, $\lambda_n>0$ such that
\begin{multline*}
%\label{uepstoW-}
\qquad \qquad \lim_{n\rightarrow +\infty} \left\|\frac{\iota_0}{\lambda_n^{d/2}}\nabla u_{n}\left(t_{n},\frac{\cdot-x_n}{\lambda_n}\right)-\nabla  W^-(t_0)\right\|_{2}
\\+\left\|\frac{\iota_1}{\lambda_n^{d/2}}\frac{\partial u_{n}}{\partial t}\left(t_{n},
\frac{\cdot-x_n}{\lambda_n}\right)-\frac{\partial W^-}{\partial t}(t_0)\right\|_{2}=0,\qquad\qquad
\end{multline*}
\end{enumerate}
where $W^-$ is as in Theorem \ref{thm:existec}.
\end{prop}

\subsection{Threshold of solutions}

In this subsection, we consider $E(u_0,u_1)=E(W,0).$

\begin{theorem}[Connecting orbits, Duyckaerts-Merle \cite{DM08} for $3\leq d\leq5$,  Li-Zhang \cite{LZ10} for $d\geq6$]\label{thm:existec}
Let $d\geq3$. There exist radial solutions $W^-$ and $W^+$ of \eqref{equ:nlwecrifoc}, with initial conditions $\left(W^{\pm}_0,W^{\pm}_1\right)\in \dot{H}^1\times L^2$ such that
\begin{align}
&E(W,0)=E(W^+_0,W^+_1)=E(W^-_0,W^{-}_1),\label{ex.energy}\\
&T_+(W^-)=T_+(W^+)=+\infty \text{ and }\lim_{t\rightarrow +\infty} W^{\pm}(t)=W \text{ in } \dot{H}^1,\label{ex.lim}\\\
&\big\|\nabla W^{-}\big\|_{2}<\|\nabla W\|_{2},\; T_-(W^-)=+\infty,
\; \|W^-\|_{L^{\frac{2(d+1)}{d-2}}((-\infty,0)\times\R^d)}<\infty,\label{ex.sub}\\
&\big\|\nabla W^{+}\big\|_{2}>\|\nabla W\|_{2},\quad \;T_-(W^+)<+\infty.\label{ex.super}
\end{align}
\end{theorem}

\begin{remark}
{\rm (i)} The above construction gives a precise asymptotic development of $W^{\pm}$ near $t=+\infty$.
Indeed there exists an eigenvalue $e_0>0$ of the linearized operator near $W$, such that,
if $\mathcal{Y}\in \mathcal{S}$ is the corresponding eigenfunction with the appropriate normalization,
\begin{equation}\label{as.dev}
\big\|\nabla\big(W^{\pm}(t)-W\pm e^{-e_0 t}\mathcal{Y}\big)\big\|_{L^2}
+\big\|\partial_t\big(W^{\pm}(t)-W\pm e^{-e_0 t}\mathcal{Y}\big)\big\|_{L^2}\leq Ce^{-2 e_0t}.
\end{equation}

{\rm (ii)} Similar solutions were constructed for NLS in \cite{DuMe07,LZ09JFA}.
However, in the NLS case, it is not able to prove that $T_-(W^+)<\infty$
except in the case $d\geq5$. We see this fact, in particular in the case $d=3$, as a nontrivial result.
 Note that $W^+$ is not in $L^2$ except for $d\ge 5$, so that case \eqref{theo.super} of Theorem \ref{thm:classif} below does not apply.
\end{remark}

The next result is that $W$, $W^-$ and $W^+$ are, up to the symmetry of the equation, the only examples of new behavior at the critical level.
\begin{theorem}[Dynamical classification at the critical level \cite{DM08}, \cite{LZ10}]\label{thm:classif}
Let $d\geq3$. Let $(u_0,u_1)\in \dot{H}^1\times L^2$ such that
\begin{equation}\label{threshold}
E(u_0,u_1)=E(W,0)=\frac{1}{dC_d^d}.
\end{equation}
Let $u$ be the solution of \eqref{equ:nlwecrifoc} with initial conditions $(u_0,u_1)$
and $I$ its maximal interval of definition. Then the following holds:
\begin{enumerate}
\item \label{theo.sub} If $ \|\nabla u_0\|_2^2<\|\nabla W\|_2^2=1/{C_d^d}$, then $I=\R$.
 Furthermore, either $u=W^-$ up to the symmetries of the equation, or
 $$\|u\|_{L^{\frac{2(d+1)}{d-2}}_{t,x}(\mathbb R\times\mathbb R^d)}<\infty.$$
\item \label{theo.crit} If $\|\nabla u_0\|_2^2=\|\nabla W\|_2^2$, then $u=W$ up to the symmetries of the equation.
\vskip0.15cm

\item \label{theo.super} If $ \|\nabla u_0\|_2^2>\|\nabla W\|_2^2$ and $u_0\in L^2$,
 then either $u=W^+$ up to the symmetries of the equation, or $I$ is finite.
\end{enumerate}
\end{theorem}
 In the above  theorem, by \emph{$u$ equals $v$ up to the sym\-me\-try of the equation},
  we mean that there exists $t_0\in\R$, $x_0\in \R^N$, $\lambda_0>0$, $\iota_0,\iota_1\in \{-1,+1\}$ such that
$$ u(t,x)=\frac{\iota_0}{\lambda_0^{(d-2)/2}}v\Big(\frac{t_0+\iota_1 t}{\lambda_0},\frac{x+x_0}{\lambda_0}\Big).$$

\begin{remark}
{\rm (i)} Case \eqref{theo.crit} is a direct consequence of the variational characterization of $W$ given by \eqref{Ell}.
 Furthermore, using assumption \eqref{threshold}, it shows by continuity of $u$ in $\dot{H}^1$ and the conservation of energy
  that the assumptions
  $$ \|\nabla u(t_0)\|_2^2<\|\nabla W\|_2^2,\;\;\text{or}\;\; \|\nabla u(t_0)\|_2^2>\|\nabla W\|_2^2$$
do not depend on the choice of the initial time $t_0$. Of course, this dichotomy does not persist when $E(u_0,u_1)>E(W,0)$.

{\rm (ii)} Theorem \ref{thm:classif} is also the analoguous, for the wave equation, of Theorem 2 of \cite{DuMe07} for NLS,
but without any radial assumption. The nonradial situation carries various problems partially solved in \cite{KM1},
the major difficulty being a sharp control in time of the space localization of the energy.

{\rm (iii)} In dimension $d=3$ or $d=4$, $W^+$ is not in $L^2$. It is a delicate problem to get rid
of the assumption $u_0\in L^2$.

{\rm (iv)}  The main issue in
high dimensions ($d\geq6$) is the non-Lipschitz continuity of the nonlinearity
which they get around by making full use of the decay property of $W$.

\end{remark}

\subsection{Beyond threshold of ground state}

Next, we consider $E(u_0,u_1)>E(W,0).$ First, we introduce the definition of two type blow up solutions.

\begin{definition}[Type I blowup solutions] Let $u:~[0,T_+)\times\R^d\to\R$ be the maximal life-span solution of (FNLW) satisfying \begin{equation}\label{eque1.2123}
T_+<+\infty,~\text{和}~\sup_{t\in[0,T_+)}\Big[\|\nabla
u\|_2^2+\|\pa_t u\|_2^2\Big]=+\infty,
\end{equation}
then we call $u(t,x)$ to be  Type I-blowup solution for (FNLW)(unbounded solutions blowing up in finite time).
We call it blows up in infinity if
\begin{equation}\label{eque1.2123inf}
T_+=+\infty,~\text{和}~\sup_{t\in[0,T_+)}\Big[\|\nabla
u\|_2^2+\|\pa_t u\|_2^2\Big]=+\infty.
\end{equation}

\end{definition}

\begin{definition}[Type II blowup solutions] 若Let $u:~[0,T_+)\times\R^d\to\R$ be the maximal life-span solution of (FNLW) with
\begin{equation}\label{eque1.212312}
T_+<+\infty,~\text{和}~\sup_{t\in[0,T_+)}\Big[\|\nabla
u\|_2^2+\|\pa_t u\|_2^2\Big]<+\infty,
\end{equation}
then we call $u(t,x)$ to be  Type II-blowup solution for (FNLW) (bounded solutions blowing up in finite time).
\end{definition}

\subsection{Existence of type I blow up solutions}
We consider the initial value problem for the focusing nonlinear wave equation
\begin{equation}\label{Eq:NLW}\left\{\begin{aligned}
& \partial^2_{t} u- \Delta u = |u|^{p-1} u,\;\;  (t,x) \in I \times \R^{d};\\
& u|_{t=0} = u_0, \;\;  \partial_t u|_{t=0}= u_1,\;\; x\in\R^{d}
\end{aligned}\right.\end{equation}
for $d = 2k+1$ with  $k\ge 2$ and $I$ an interval with $0 \in I$.  \eqref{Eq:NLW} is conformally invariant
for $p = \frac{d+3}{d-1}$ and we restrict ourselves to the \textit{superconformal} case
\begin{align}\label{Eq:superconf}
 p > \frac{d+3}{d-1}.
 \end{align}

Explicit examples for singularity formation can be obtained by considering the so called ODE blowup solution
\begin{align}\label{Eq:BlowUpSol}
 	 u_T(t,x)= c_p (T-t)^{-\frac{2}{p-1}},  \quad   c_p := \left [\tfrac{2(p+1)}{(p-1)^2} \right]^{\frac{1}{p-1}},
\end{align}
which is independent of the space dimension and solves the ordinary differential equation
$u_{tt} = |u|^{p-1} u$ for $p >1$. By finite speed of propagation one can use $u_T$
to construct compactly supported smooth initial data such that the solution blows up as $t \to T$.

\begin{theorem}[Existence of type I blow up solution, \cite{DonningerSchorkhuber1}]\label{thm:typei}
 For $d= 2k+1$, $k \in \mathbb N$ with $k \geq 2$, fix $p > \frac{d+3}{d-1}$ and $T_0 > 0$.
 There are constants $M, \delta > 0$ such that if $u_0,u_1$ are radial functions with
\begin{align}\label{Eq:CondData}
\|(u_0,u_1)-  u_{T_0}[0] \|_{H^{\frac{d+1}{2}} \times H^{\frac{d-1}{2}}({B}_{T_0+\delta})} < \tfrac{\delta}{M},
\end{align}
where $B_{T_0+\delta}\triangleq\{x: |x|\le T_0+\delta\}$. Then
\begin{enumerate}
\item The blowup time at the origin $T :=T_{(u_0,u_1) }$ is contained in the interval $[T_0 - \delta, T_0 + \delta]$.
\item The solution
$u: \Gamma_T(\R^d)  \to \R$ satisfies
\begin{equation}\label{Eq:DecayRates_1}\left\{\begin{aligned}
&(T-t)^{\frac{2}{p-1}} \| u(t,\cdot) - u_T(t,\cdot) \|_{L^{\infty}({B}_{T-t})}  \lesssim (T-t)^{\mu_p}, \\
&(T- t)^{-s_p} \|u(t,\cdot) - u_T(t,\cdot) \|_{L^2({B}^d_{T-t}) }  \lesssim (T-t)^{\mu_p}, \\
&(T- t)^{-s_p + 1}  \|u[t] - u_T[t] \|_{\dot H^1 \times  L^{2}({B}_{T-t}) }  \lesssim (T-t)^{\mu_p}, \\
\end{aligned}\right.
\end{equation}
for $s_p = \frac{d}{2} - \frac{2}{p-1}$, $\mu_p = \mathrm{min} \{\frac{2}{p-1},1 \}
- \varepsilon$ and $\varepsilon >0$ small. Furthermore, for $j = 2, \dots, \frac{d+1}{2}$,
\begin{align}\label{Eq:DecayRates_2}
\begin{split}
(T- t)^{-s_p + j}  \|u[t] \|_{\dot H^j \times \dot H^{j-1}({B}_{T-t}) }  \lesssim  (T-t)^{\mu_p}. \\
\end{split}
\end{align}
where $$\Gamma_T(\R^d)\triangleq \big\{(t,r): \; t\in [T_0-\delta, T_0+\delta], \; r=|x|\in (0, T-t)\big\}.$$
\end{enumerate}
\end{theorem}

\subsection{Existence of type II blowup solutions}

Examples of radial type II blow-up solutions of (FNLW) were constructed in space dimension $d=3$
by Krieger, Schlag and Tataru \cite{KST09}, see Theorem \ref{thm:exblup} below.

\begin{theorem}[Existence of type II blowup solution in  $\R^3$,
\cite{KST09}]\label{thm:exblup}
Let $\nu>\frac12$ and $\delta>0$. Then there exists an energy
solution $u$ of {\rm (FNLW)} which blows up precisely at
$r=t=0$  and which has the following property: in the cone $|x|=r\le
t $ and for small times $t$ the solution has the form,
\[
u(x,t) = \lambda^\frac12(t) W(\lambda(t) r) + \eta(x,t), \;\;\lambda(t)=t^{-1-\nu},
\]
where ${E}_{\rm loc}(\eta(\cdot,t))\to0$ as $t\to0$ and outside
the cone $u(x,t)$ satisfies
\[
\int_{[|x|\ge t]} \big[|\nabla u(x,t)|^2+ |u_t(x,t)|^2 +
|u(x,t)|^6\big]\, dx <\delta
\]
for all sufficiently small $t>0$. In particular, the energy of these
blow-up solutions can be chosen arbitrarily close to~${E}(W,0)$,
i.e., the energy of the stationary solution.
\end{theorem}

\begin{remark}
{\rm (i)} Lately, Krieger and Schlag \cite{KrSc14} improved this to $\nu > 0$

{\rm (ii)} If $\nu>1$, then the solutions from Theorem~\ref{thm:exblup}
  belong to $L^\infty(\R^3)$ for all $t>0$ and blow up
at the rate
\[ \|u(\cdot,t)\|_\infty \asymp t^{-(1+\nu)/2}, \; \; {\rm as}\;\; t\to 0.\]
\end{remark}

In dimension five, denote $X^s := \dot H^{s+1}\cap \dot H^1$.
\begin{theorem}[Existence of smooth type II blowup solutions in $d=5$,\cite{moi15p}]
  \label{thm:non-deg}
  Let $(u^*_0, u^*_1) \in X^4\times H^4$ be any pair of radial functions with $u^*_0(0) > 0$.
  Let $(u^*(t), \partial_t u^*(t))$ be the solution of {\rm(FNLW)} with the initial data $(u^*(0), \partial_t u^*(0)) = (u^*_0, u^*_1)$.
  There exists a solution $(u, \partial_t u)$
  of {\rm (FNLW)} defined on a time interval $(0, T_0)$ and a $C^1$ function
  $\lambda(t): (0, T_0) \to (0, +\infty)$ such that
  \begin{equation}
    \label{eq:blowup-1}\left\{\begin{aligned}
    &\|(u(t) - W_{\lambda(t)} - u^*(t),\; \partial_t u(t) +\lambda_t(t) (\Lambda W)_{\lambda(t)}
    -\partial_t u^*(t))\|_{\dot{H}^1\times L^2} = O(t^{9/2});\\
  &\lambda(t) = \big(\frac{32}{315\pi}\big)^2\big(u^*(0, 0)\big)^2 t^4 + o(t^4),\;\;  \text {as }t \to 0^+.\end{aligned}\right.\end{equation}
\end{theorem}
\begin{theorem}[\cite{moi15p}]
  \label{thm:deg}
  Let $\nu > 8$. There exists a solution $(u, \partial_t u)$ of {\rm (FNLW)}
  defined on the time interval $(0, T_0)$  such that
  \begin{equation}
    \label{eq:blowup-2}
    \lim_{t \to 0^+}\big\|\big(u(t) - W_{\lambda(t)} - u^*_0,\ \partial_t u(t) - u^*_1\big)\big\|_{\dot{H}^1\times L^2} = 0,
  \end{equation}
  where $\lambda(t) = t^{\nu+1}$, and $(u^*_0, u^*_1)$ is an explicit radial $C^2$ function.
\end{theorem}
We refer to the situation of Theorem
\ref{thm:non-deg} as the \emph{non-degenerate case} and to the situation of Theorem
\ref{thm:deg} as the \emph{degenerate case}.
Note that in Theorem
\ref{thm:non-deg} we allow any regular $(u_0^*, u_1^*)$ with $u_0^*(0) > 0$.
The above result might be seen as a first step in a possible classification of all blow-up solutions with a non-degenerate asymptotic profile.

Let us mention that radiality is only a simplifying assumption.
A similar construction should be possible also for non-radial $(u^*_0, u^*_1)$.

\subsection{Description of type II blow up solutions}

Duyckaerts-Kenig-Merle (JEMS,\cite{DKM11}) first studied the  type II blow up solutions
when the kinetic energy of solution for (FNLW) is slightly larger than  the kinetic energy of ground state.
\vskip0.15cm

$\blacksquare$\; If $d=3,4$ or $u$ is radial, and
\vskip-0.15cm
 \begin{equation}\label{equ:eque1.3focscon}
\sup_{t\in[0,T_+)}\Big[\|\nabla u\|_2^2+\|\pa_t
u\|_2^2\Big]<\|\nabla W\|_2^2.
\end{equation}
Then $T_+=\infty$ and the solution scattering forward in time.
\vskip0.15cm
$\blacksquare$\; The threshold $\|\nabla W\|_2^2$ is sharp. Krieger, Schlag and Tataru \cite{KST09}'s result shows:  For $\forall \eta_0>0$, there exists
a Type II blow-up solution such that
\begin{equation}\label{equ:eque1.3focscon-1}
\sup_{t\in[0,T_+)}\Big[\|\nabla u\|_2^2+\|\pa_t
u\|_2^2\Big]<\|\nabla W\|_2^2+\eta_0.
\end{equation}

$\blacksquare$\; Duyckaerts-Kenig-Merle investigate the inverse problem under an appropriate small assumption.

\vskip0.15cm

\noindent\underline{\bf Description theorem of type II non radial blow up solutions}:

\begin{theorem}[Duyckaerts-Kenig-Merle,\cite{DKM12JEMS}]\label{thm:altmove}\rm
Let $d=3$ or  $d=5$ and  $\eta_0$ be small,   the maximal solution $u:~[0,T_+)\times\R^3\to\R$ for {\rm (FNLW)} satisfies that  $u:~[0,T_+)\times\R^3\to\R$
 $T_+<+\infty$ and
\begin{equation}\label{eque1.6}
\sup_{t\in[0,T_+)}\Big[\|\nabla
u(t,\cdot)\|_2^2+\frac{d-2}2\|\pa_tu(t,\cdot)\|_2^2\Big]\leq\|\nabla
W\|_2^2+\eta_0.
\end{equation}
Then there exist $(v_0,v_1)\in\dot{H}^1\times L^2,~\iota_0\in\{\pm1\}$,
small parameter $\ell>0$  and  $x(t),~\la(t)$ with
\begin{equation}\label{eque1.7}
\lim\limits_{t\to T_+}\frac{\la(t)}{T_+-t}=0,~\lim_{t\to
T_+}\frac{x(t)}{T_+-t}=\ell\vec{e}_1,~|\ell|\leq C\eta_0^\frac14,
\end{equation}
 such that when  $t\to T_+$,
\begin{equation}\label{eque1.8}
\big(u(t,\cdot),\pa_t(t,\cdot)\big)-(v_0,v_1)-\Big(\frac{\iota_0}{\la(t)^\frac{d-2}2}W_\ell\Big(0,\frac{x-x(t)}{\la(t)}\Big),
\frac{\iota_0}{\la(t)^\frac{d}2}(\pa_tW_\ell)\Big(0,\frac{x-x(t)}{\la(t)}\Big)\Big)\to 0
\end{equation}
in  $\dot{H}^1\times L^2$, where
\begin{equation}\label{eque1.9}\left\{\begin{aligned}
&W_\ell(t,x)=\Big(1+\frac{|x_1-\ell t|^2}{d(d-2)(1-\ell^2)}+\frac{|\bar{x}|^2}{d(d-2)}\Big)^{-\frac{d-2}2},\\
& W_\ell(t,x)\triangleq W\Big(\frac{x_1-\ell t}{\sqrt{1-\ell^2}},\bar{x}\Big),\;\; \text{\rm Lorentz transform},\\
&\|\nabla v_0\|_2^2+\frac{d-2}2\|v_1\|_2^2\leq \eta_0,\;\; (v_0,v_1)\;\text{\rm to be small initial data}.\end{aligned}\right.
\end{equation}
\end{theorem}

\begin{remark}\label{remk1.1.1}
{\rm (i)} The clue of proof: \; By introducing the regular part of  $u(t,x)$:
$$v(t,x)=S_N(t-T_+)(v_0,v_1)\in C(\R;\dot{H}^1\times L^2),$$
 where $(v_0,v_1)$ is a weak convergent point of $(u,\pa_tu)$.  Then \eqref{eque1.8}
can be reduced to prove the singularity part $a(t,x)=u(t,x)-v(t,x)$
such that   when $t\to T_+$, in the sense of $\dot{H}^1\times L^2$,
\begin{center}
\begin{tabular}{|c|}
  \hline
  % after \\: \hline or \cline{col1-col2} \cline{col3-col4} ...
  $\Big(\la(t)^\frac{d-2}2a\big(t,\la(t)x+x(t)\big),\la(t)^\frac{d}2\pa_ta\big(t,\la(t)x+x(t)\big)\Big)\to
\ell_0\big(W_\ell(0),\pa_tW_\ell(0)\big).$\\
  \hline
\end{tabular}\\
$\Uparrow,~~\forall~t_n'\to T_+,$ \\
\begin{tabular}{|c|}
  \hline
  % after \\: \hline or \cline{col1-col2} \cline{col3-col4} ...
  $\Big(\la(t_n')^\frac{d-2}2a\big(t_n',\la(t_n')x+x(t_n')\big),\la(t_n')^\frac{d}2\pa_ta\big(t_n',\la(t_n')x+x(t_n')\big)\Big)\to
\ell_0\big(W_\ell(0),\pa_tW_\ell(0)\big),$  \\
  \hline
\end{tabular}\\$\Uparrow~~\exists~t_n\to T_+,$ \\
\begin{tabular}{|c|}
  \hline
  % after \\: \hline or \cline{col1-col2} \cline{col3-col4} ...
  $\Big(\la(t_n)^\frac{d-2}2a\big(t_n,\la(t_n)x+x(t_n)\big),\la(t_n)^\frac{d}2\pa_ta\big(t_n,\la(t_n)x+x(t_n)\big)\Big)\to
\ell_0\big(W_\ell(0),\pa_tW_\ell(0)\big),$  \\
  \hline
\end{tabular}\\$\Uparrow$ \\
\begin{tabular}{|c|}
  \hline
  % after \\: \hline or \cline{col1-col2} \cline{col3-col4} ...
  $\Big(\la(t_n)^\frac{d-2}2a\big(t_n,\la(t_n)x+x(t_n)\big),\la(t_n)^\frac{d}2\pa_ta\big(t_n,\la(t_n)x+x(t_n)\big)\Big)\rightharpoonup
\ell_0\big(W_\ell(0),\pa_tW_\ell(0)\big),$  \\
  \hline
\end{tabular}
\end{center}
 for details please refer to \cite{DKM12JEMS}.

\vskip0.2cm

{\rm (ii)}  It seems to be interesting for us  to study the description of type II blow up solutions
of focusing energy critical wave equation in higher dimensions.

\end{remark}

\subsection{Bounds on the speed of type II blow-up}

In sum, for type II blowup solutions,
\begin{theorem}[\cite{DKM11}, \cite{DKM12JEMS}, \cite{Kenig4dWave}]
  Let $ u(t)$ be a radial solution of {\rm (FNLW)} which blows up at $t = T_+$
  by concentration of one bubble of energy. Then there exist $ \vec{u}^*_0 \in \dot{H}^1\times L^2$ and $\lambda(t) \in C([t_0, T_+), (0, +\infty))$
    such that
    \begin{equation}\label{eq:DKM1}
      \lim_{t\to T_+}\| \vec{u}(t) - \vec{W}_{\lambda(t)} -  \vec{u}^*_0\|_{\dot{H}^1\times L^2} = 0,\qquad \lim_{t \to T_+} (T_+ - t)^{-1}\lambda(t) = 0.
    \end{equation}
\end{theorem}
In the following, the function $\vec{u}^*_0$ is called the \emph{asymptotic profile}.
Note that in \cite{DKM12JEMS} a more general, non-radial version of the above theorem was proved for $d \in \{3, 5\}$.

Solutions verifying \eqref{eq:DKM1} were first constructed in dimension $d = 3$ by Krieger, Schlag and Tataru \cite{KST09}, who obtained all possible
polynomial blow-up rates $\lambda(t) \sim (T_+ -t)^{1+\nu}$, $\nu > 0$.
For $d = 4$ smooth solutions blowing up at a particular rate were constructed by Hillairet and Rapha\"el \cite{HR12}.
For $d = 5$ they  proved in \cite{moi15p} that for any radially symmetric asymptotic profile $\vec u^*_0 \in H^4 \times H^3$
such that $u^*_0(0) > 0$, there exists a solution $\vec u(t)$ such that \eqref{eq:DKM1} holds.
For these solutions the concentration speed of the bubble is
\begin{equation}
  \label{eq:con-speed}
  \lambda(t) \sim u^*_0(0)^2 (T_+ - t)^4.
\end{equation}

Jendrej continued the investigation of the relationship between the behaviour of $\vec u^*_0$ at $x = 0$
and possible blow-up speeds, still in the special case when the asymptotic profile $\vec u^*_0$ is sufficiently regular.
He proved the following result.
\begin{theorem}[Bounds on the speed of type II blow-up,\cite{Jen}] \label{thm:loiest}
 Let $d \in \{3, 4, 5\}$ and $s > \frac{d-2}{2}$, $s \geq 1$. Let $\vec u_0^* = (u_0^*, \dot u_0^*) \in H^{s+1} \times H^s$ be a radial function.
  Suppose that $\vec u$ is a radial solution of {\rm(FNLW)} such that
  \begin{equation}
    \label{eq:loi}
    \lim_{t \to T_+}\|\vec u(t) - \vec W_{\lambda(t)} - \vec u_0^*\|_{\dot{H}^1\times L^2} = 0, \qquad \lim_{t \to T_+} \lambda(t) = 0,\qquad T_+ < +\infty.
  \end{equation}
 There exists a constant $C > 0$ depending on $\vec u^*_0$ such that:
  \begin{itemize}
    \item If $d \in \{4, 5\}$, then for $T_+ - t$ sufficiently small there holds
      \begin{equation}
        \label{eq:loi-borne}
        \lambda(t) \leq C (T_+ - t)^\frac{4}{6-d}.
      \end{equation}
  \item If $d = 3$, then there exists a sequence $t_n \to T_+$ such that
      \begin{equation}
        \label{eq:loi-borne-N3}
        \lambda(t_n) \leq C (T_+ - t_n)^\frac{4}{6-d}.
      \end{equation}
  \end{itemize}
\end{theorem}

 \begin{theorem}[\cite{Jen}]
    \label{thm:negagt}
    Let $d \in \{3, 4, 5\}$. Let $\vec u_0^* = (u_0^*, \dot u_0^*) \in H^3 \times H^2$ be a radial function such that
    \begin{equation}
      \label{eq:ustar-neg}
        u^*_0(0) < 0.
    \end{equation}
  There exist no radial solutions of {\rm (FNLW)} such that
  \begin{equation}
    \label{eq:neg}
    \lim_{t \to T_+}\|\vec u(t) - \vec W_{\lambda(t)} - \vec u_0^*\|_{\dot{H}^1\times L^2} = 0, \qquad \lim_{t \to T_+} \lambda(t) = 0,\qquad T_+ < +\infty.
  \end{equation}
  \end{theorem}

\subsection{Soliton resolution conjecture}

From the above analysis, we can obtain the dynamics of the energy solutions to the focusing energy-critical wave equation in some restrictions.
 Soliton resolution conjecture describes such dynamics.
In \cite{Sof06}, Soffer first proposed this conjecture, and lately it was modified as following
\begin{conjecture}[Soliton resolution conjecture]\label{conjecture:soliton}
In the context of equation (NLW), the soliton resolution conjecture predicts that any bounded solution should asymptotically decouple into a finite sum of modulated solitons, a regular part in the finite time blow up case or a free radiation in the global case, plus a residue term that vanishes asymptotically in the energy space as time approaches the maximal time of existence.
\end{conjecture}

 Now, we give the precise soliton resolution for focusing energy-critical NLW.
Let Lorentz transforms of $W$
\begin{align}
W_{\ell}(t,x)\triangleq W\bigg(\Big(-\frac{t}{\sqrt{1-|\ell|^2}}+\frac{1}{|\ell|^2} \Big(\frac{1}{\sqrt{1-|\ell|^2}}
-1\Big)\ell\cdot x\Big)\ell+x\bigg).\label{travelling_waves}
\end{align}
 Note that
$$W_{\ell}(t,x)=W_{\ell}(0,x-t\ell),\quad\; \ell \in \R^d, \;|\ell|<1,$$
so  that $W_{\ell}$ is a solitary wave traveling at speed $|\ell|$. The energy of $W_{\ell}$ is given by:
\begin{equation}
\label{EQl}
E(\vec{W}_{\ell}(0))=\frac{1}{\sqrt{1-|\ell|^2}}E(\vec{W}(0))\underset{|\ell|\to 1}{\longrightarrow}+\infty.
\end{equation}
It is conjectured that any bounded  solution of focusing energy-critical NLW is a sum of modulated, decoupled traveling waves and a scattering part. More precisely:

\begin{conjecture}[Soliton resolution for general type II blowup solutions] \label{C:E1}
 Let $u$ be a solution of of {\rm (FNLW)} on $[0,T_+)\times \R^d$ such that
 \begin{equation}\label{E2}
  \sup_{t\in[0,T_+)}\left\|\vec{u}(t)\right\|_{\dot{H}^1\times L^2}<\infty.
 \end{equation}

{\rm (i)} If $T_+=+\infty$,
 then there exist a solution $v_{L}$ of the linear wave equation
\begin{equation}
 \label{E3}
 \left\{\begin{aligned}
&( \partial_t^2-\Delta)v_{L}=0,\\
&\vec{v}_L|_{t=0}=(v_0,v_1)\in \dot{H}^1\times L^2,
\end{aligned}\right.
\end{equation}
an integer $J\geq 0$, and for $j\in \{1,\ldots,J\}$, a (nonzero) traveling wave $W_{\ell_j}^j$ ($|\ell_j|<1$),
 and parameters $x_j(t)\in \R^d$, $\lambda_j(t)\in \R^+$  such that
\begin{equation}
\label{expansion}
\lim_{t\to+\infty}
 \vec{u}(t)-\vec{v}_{L}(t)-\sum_{j=1}^J \left(\frac{1}{\lambda_j(t)^{\frac{N-2}{2}}}W_{\ell_j}^j\left(0,\frac{\cdot-x_j(t)}{\lambda_j(t)}\right),\frac{1}{\lambda_j(t)^{\frac{N}{2}}}
 \partial_t W_{\ell_j}^j\left(0,\frac{\cdot-x_j(t)}{\lambda_j(t)}\right)\right)=0
\end{equation}
in $\dot{H}^1\times L^2$ and
\begin{equation*}\left\{
\begin{aligned}
 &\forall j\in \{1,\ldots,J\},\quad \lim_{t\to\infty}\frac{x_j(t)}{t}=\ell_j,\quad \lim_{t\to\infty}\frac{\lambda_j(t)}{t}=0\\
 &\forall j,k\in \{1,\ldots,J\},\quad j\neq k\Longrightarrow \lim_{t\to+\infty}\frac{|x_j(t)-x_k(t)|}{\lambda_j(t)}+\frac{\lambda_j(t)}{\lambda_k(t)}+\frac{\lambda_k(t)}{\lambda_j(t)}=+\infty\\
 &W_{\ell}(t,x)=W\bigg(\Big(-\frac{t}{\sqrt{1-|\ell|^2}}+\frac{1}{|\ell|^2}
 \Big(\frac{1}{\sqrt{1-|\ell|^2}}-1\Big)\ell\cdot x\Big)\ell+x\bigg).
\end{aligned}\right.\end{equation*}

{\rm (ii)} If $T_+<+\infty$, then the following is true. Define the singular set
\begin{equation}
\mathcal{S}\triangleq\Big\{x_{\ast}\in \R^d:\,\|u\|_{L_t^\frac{d+2}{d-2}L_x^{\frac{2(d+2)}{d-2}}
\big(B_{\epsilon}(x_{\ast})\times [T_+-\epsilon,\,T_+)\big)}=\infty,\; {\rm for} \;\forall\;\epsilon>0\Big\}.
\end{equation}
Then $\mathcal{S}$ is a finite set. Let $x_{\ast}\in\mathcal{S}$ be a singular point. Then there exist an integer $J_{\ast}\ge 1$, $r_{\ast}>0$, $\vec{v}_0\in \dot{H}^1\times L^2$,   scales $\lambda_n^j$ with $0<\lambda_n^j\ll T_+-t_n$, positions $c_n^j\in R^d$ satisfying $c_n^j\in B_{\beta (T_+-t)}(x_{\ast})$ for some $\beta\in(0,1)$ with $\ell_j=\lim\limits_{n\to\infty}\frac{c_n^j-x_{\ast}}{T_+-t}$ well defined, and traveling waves $W_{\ell_j}^j$, for $1\leq j\leq J_{\ast}$, such that inside the ball $B_{r_{\ast}}(x_{\ast})$ we have
\begin{align}
\vec{u}(t_n)=\vec{v}_0+&\sum_{j=1}^{J_{\ast}}\,\Big((\lambda_n^j)^{-\frac{d}{2}+1}\,
W_{\ell_j}^j\Big(\frac{x-c_n^j}{\lambda_n^j},\,0\Big),\,(\lambda_n^j)^{-\frac{d}{2}}\,
\partial_tW_{\ell_j}^j\Big(\frac{x-c_n^j}{\lambda_n^j},\,0\Big)\Big)\nonumber\\
&+\,\,o_{\dot{H}^1\times L^2}(1), \quad\,\,{\rm as}\,\,n\to\infty.\label{eq:maindecompositionhaha}
\end{align}
In addition, the parameters $\lambda_n^j,\,c_n^j$ satisfy the pseudo-orthogonality condition
\begin{equation}
\label{ortho123}
\frac{\lambda_n^j}{\lambda_n^{j'}}+\frac{\lambda_n^{j'}}{\lambda_n^j}+\frac{|c^j_n-c^{j'}_n|}{\lambda^j_n}\to\infty,
\end{equation}
as $n\to\infty$, for each $1\leq j\neq j'\leq J_{\ast}$.

\end{conjecture}

\begin{remark}
The soliton resolution conjecture was proved in \cite{DKM13}, in the radial case for $d=3$ as in Theorem \ref{thm:srreect123} below. For other dimensions in the radial case, soliton resolution is only known along a sequence of times, see \cite{Kenig4dWave,Casey,JiaKenig}. In the nonradial case, the soliton resolution was proved for type II blow up solutions in \cite{DKM12JEMS} for $d=3,\,5$, under an extra smallness condition. For general large data (does not scatter forward in time), it was proved in \cite{DKMprofile} that along a sequence of times, the solution converges locally, after an appropriate rescaling, to a modulated soliton. The proof of \cite{DKMprofile} relies on a compactness/rigidity argument that works for a large class of dispersive equations.

Recently, Duyckaerts-Jia-Kenig-Merle \cite{DJKM} (see Theorem \ref{thm:djkm16} bleow) prove the soliton resolution conjecture for general type II solutions along a sequence of times. This is an important step towards the full soliton resolution in the nonradial case and without any size restrictions.
\end{remark}

If $f$ and $g$ are two positive functions defined in a neighborhood of $\ell\in \R\cup\{\pm\infty\}$, we will write
$$f(t)\ll g(t) \text{ as }t\to \ell \text{ if and only if }
\lim_{t\to \ell}\frac{f(t)}{g(t)}=0.$$
Soliton resolution conjecture for radial solution in $\R^3$ was proved by Duyckaerts-Kenig-Merle\cite{DKM13}.

\begin{theorem}[Soliton resolution conjecture for radial solution in $\R^3$, \cite{DKM13}]\label{thm:srreect123}
 Let $u$ be a radial solution of {\rm(FNLW)} and $T_+=T_+(u)$. Then one of the following holds:
\begin{itemize}
 \item {\rm\bf Type I blow-up:} $T_+<\infty$ and
\begin{equation}
\label{BUp123}
\lim_{t\to T_+} \|(u(t),\partial_tu(t)\|_{\dot{H}^1\times L^2}=+\infty.
\end{equation}
\item {\rm \bf Type II blow-up:} $T_+<\infty$ and there exist $(v_0,v_1)\in\dot{H}^1\times L^2$, an integer
$J\in \mathbb N\setminus\{0\}$, and for all $j\in \{1,\ldots,J\}$, a sign $\iota_j\in \{\pm 1\}$, and a positive function $\lambda_j(t)$ defined for $t$ close to $T_+$ such that
\begin{gather}
\label{hyp_lambda_bup122}
 \lambda_1(t)\ll \lambda_2(t)\ll \ldots \ll\lambda_{J}(t)\ll T_+-t\text{ as }t\to T_+\\
\label{expansion_u_bup122}
\lim_{t\to T_+}
\bigg\|(u(t),\partial_tu(t))-\Big(v_{0}+\sum_{j=1}^J\frac{\iota_j}{\lambda_j^{1/2}(t)}W\Big(\frac{x}{\lambda_{j}(t)}\Big),
v_{1}\Big)\bigg\|_{\dot{H}^1\times L^2}=0.
\end{gather}
\item {\rm\bf Global solution:} $T_+=+\infty$ and there exist a solution $v_{L}$ of the linear wave equation, an integer $\mathbb N\setminus\{0\}$, and for all $j\in \{1,\ldots,J\}$, a sign $\iota_j\in \{\pm 1\}$, and a positive function $\lambda_j(t)$ defined for large $t$ such that
\begin{gather}
\label{hyp_lambda123}
 \lambda_1(t)\ll \lambda_2(t)\ll \ldots \ll\lambda_{J}(t)\ll t\text{ as }t\to +\infty\\
\label{expansion}
\lim_{t\to+\infty}
\bigg\|(u(t),\partial_tu(t))-\Big(v_{L}(t)+\sum_{j=1}^J\frac{\iota_j}{\lambda_j^{1/2}(t)}W\Big(\frac{x}{\lambda_{j}(t)}\Big),
\partial_tv_{L}(t)\Big)\bigg\|_{\dot{H}^1\times L^2}=0.
\end{gather}
\end{itemize}
\end{theorem}

Soliton resolution conjecture along a sequence of times was solved by Duyckaerts-Jia-Kenig-Merle\cite{DJKM}.

\begin{theorem}[Soliton resolution conjecture along a sequence of times,\cite{DJKM}]\label{thm:djkm16}
Let $(u,\pa_tu)\in C([0,T_+),\dot{H}^1\times L^2(\R^d))$ with $u\in L_t^\frac{d+2}{d-2}L_x^{\frac{2(d+2)}{d-2}}(\R^d\times [0,T))$ for any $T<T_+$,
 be a solution to  {\rm (FNLW)} that satisfies
 \begin{equation}
\label{bounded}
\sup_{t\in [0,T_+)}\,\|(u,\pa_tu)(t)\|_{\dot{H}^1\times L^2}<\infty,
\end{equation}
where $T_+$ denotes the maximal existence time of $u$.

{\bf Case I: $T_+<\infty$}.  Define the singular set
\begin{equation}
\mathcal{S}\triangleq\Big\{x_{\ast}\in \R^d:\,\|u\|_{L_t^\frac{d+2}{d-2}L_x^{\frac{2(d+2)}{d-2}}(B_{\epsilon}(x_{\ast})\times [T_+-\epsilon,\,T_+))}
=\infty,\;{\rm for},\,\forall \;\epsilon>0\Big\}.
\end{equation}
Then $\mathcal{S}$ is a finite set. Let $x_{\ast}\in\mathcal{S}$ be a singular point.
Then there exist an integer $J_{\ast}\ge 1$, $r_{\ast}>0$, $\vec{v}_0\in \dot{H}^1\times L^2$,
a time sequence $t_n\uparrow T_+$,  scales $\lambda_n^j$ with $0<\lambda_n^j\ll T_+-t_n$,
 positions $c_n^j\in \R^d$ satisfying $c_n^j\in B_{\beta (T_+-t_n)}(x_{\ast})$
 for some $\beta\in(0,1)$ with $\ell_j=\lim\limits_{n\to\infty}\frac{c_n^j-x_{\ast}}{T_+-t_n}$
 well defined, and traveling waves $W_{\ell_j}^j$, for $1\leq j\leq J_{\ast}$, such that inside the ball $B_{r_{\ast}}(x_{\ast})$ we have
\begin{align}
\vec{u}(t_n)=\vec{v}_0+\sum_{j=1}^{J_{\ast}}&\Big((\lambda_n^j)^{-\frac{d}{2}+1}\, W_{\ell_j}^j\Big(\frac{x-c_n^j}{\lambda_n^j},\,0\Big),
\,(\lambda_n^j)^{-\frac{d}{2}}\, \partial_tW_{\ell_j}^j\Big(\frac{x-c_n^j}{\lambda_n^j},\,0\Big)\Big)\nonumber\\
 &+\,\,o_{\dot{H}^1\times L^2}(1), \quad\,\,{\rm as}\,\,n\to\infty.\label{eq:maindecompositionhaha}
\end{align}
In addition, the parameters $\lambda_n^j,\,c_n^j$ satisfy the pseudo-orthogonality condition
\begin{equation}
\label{ortho}
\frac{\lambda_n^j}{\lambda_n^{j'}}+\frac{\lambda_n^{j'}}{\lambda_n^j}+\frac{|c^j_n-c^{j'}_n|}{\lambda^j_n}\to\infty,
\end{equation}
as $n\to\infty$, for each $1\leq j\neq j'\leq J_{\ast}$.\\

{\bf Case II: $T_+=\infty$.}
There exist a finite energy solution $u^L$ to the linear wave equation
$$ \partial_{tt}u^L-\Delta u^L=0\text{ in }\R^d\times \R,$$
an integer $J_{\ast}\ge 0$ \footnote{If $J_{\ast}=0$, then there is no soliton in the
decomposition below and the solution scatters.}, a time sequence $t_n\uparrow \infty$,
scales $\lambda_n^j$ with $\lambda_n^j>0$ and $\lim\limits_{n\to\infty}\frac{\lambda^j_n}{t_n}=0$,
positions $c_n^j\in \R^d$ satisfying $c_n^j\in B_{\beta t_n}(0)$ for some $\beta\in(0,1)$
with $\ell_j=\lim\limits_{n\to\infty}\frac{c_n^j}{t_n}$ well defined,
and traveling waves $W_{\ell_j}^j$, for $1\leq j\leq J_{\ast}$, such that
\begin{align}
\vec{u}(t_n)=\vec{u}^L(t_n)+\sum_{j=1}^{J_{\ast}}&\Big((\lambda_n^j)^{-\frac{d}{2}+1}\, W_{\ell_j}^j
\Big(\frac{x-c_n^j}{\lambda_n^j},\,0\Big),\,(\lambda_n^j)^{-\frac{d}{2}}\,
\partial_tW_{\ell_j}^j\Big(\frac{x-c_n^j}{\lambda_n^j},\,0\Big)\Big)\nonumber\\
&+\,\,o_{\dot{H}^1\times L^2}(1), \quad\,\,{\rm as}\,\,n\to\infty.\label{eq:maindecompositionhahaGlobal}
\end{align}
In addition, the parameters $\lambda_n^j,\,c_n^j$ satisfy the pseudo-orthogonality condition \eqref{ortho}.
\end{theorem}

\subsection{Blowup at infinity for the critical wave equation}\label{subsec:binf}

\begin{theorem}[\cite{DoKr12P}]\label{thm:notscrsl}
There exists an $\varepsilon_0>0$ such that for any $\delta>0$ and $\mu \in \R$ with $|\mu|\leq \varepsilon_0$ there exists
a $t_0\geq 1$ and an energy class solution $u: [t_0,\infty)\times \R^3\to \R$ of {\rm (FNLW)}
of the form
\[ u(t,x)=t^{\frac{\mu}{2}}W(t^\mu x)+\eta(t,x),\quad |x|\leq t,\: t\geq t_0 \]
and for all $t\geq t_0$
\begin{equation*}\left\{\begin{aligned}
&\|\partial_t u(t,\cdot)\|_{L^2(\R^3\backslash B_t)}+\|\nabla u(t,\cdot)\|_{L^2(\R^3\backslash B_t)}&\leq \delta,
\\
&\|\partial_t \eta(t,\cdot)\|_{L^2(B_t)}+\|\nabla \eta(t,\cdot)\|_{L^2(B_t)}\leq \delta,
\end{aligned}\right.\end{equation*}
 where $B_t:=\{x\in\R^3: |x|<t\}$.
\end{theorem}

It seemed plausible
to expect a strong version of the \emph{soliton resolution conjecture} to hold.

\subsection{Two-bubble solutions}

\begin{theorem}[Two-bubble solutions,\cite{Jen1}]
  \label{thm:deux-bulles} Let $d=6$.
There exists a solution $\vec u: (-\infty, T_0] \to \dot{H}^1\times L^2$ of {\rm(FNLW)} such that
  \begin{equation}
    \label{eq:mainthm}
    \lim_{t\to -\infty}\|\vec u(t) - (\vec W + \vec W_{\frac{1}{\kappa}e^{-\kappa|t|}})\|_{\dot{H}^1\times L^2} = 0,
    \;\;\; \text{with }\kappa \triangleq\sqrt{5/4},
  \end{equation}
  where $\vec{W}=(W,0)$ and
  $$\vec W_\lambda=(W_\lambda,0),~W_\lambda(x)=\frac1{\lambda^2}W(x/\lambda).$$
\end{theorem}

\begin{remark}
{\rm(i)} More precisely, Jendrej in \cite{Jen1} proved that
  \begin{equation*}
    \big\|\vec u(t) - \big(\vec W + \vec W_{\frac{1}{\kappa}e^{-\kappa|t|}}\big)\big\|_{\dot{H}^1\times L^2}\leq C_1 \cdot e^{-\frac 12 \kappa|t|}
  \end{equation*}
  for some constant $C_1 > 0$.

{\rm(ii)}  The author construct here \emph{pure} two-bubbles, that is the solution approaches a superposition of two stationary states,
with no energy transformed into radiation.
By the conservation of energy and the decoupling of the two bubbles, we necessarily have $E(\vec u(t)) = 2E(\vec W)$.
Pure one-bubble cannot concentrate and is completely classified, see \cite{DM08}.

{\rm(iii)} It was proved in \cite{moi15p-3} (see Theorem \ref{thm:nondeux-bulles} below),
 in any dimension $d \geq 3$, that there exist no solutions $\vec u(t): [t_0, T_+) \to \dot{H}^1\times L^2$
    of (FNLW) such that $\|\vec u(t) - (\vec W_{\mu(t)} - \vec W_{\lambda(t)})\|_{\dot{H}^1\times L^2} \to 0$ with $\lambda(t) \ll \mu(t)$
  as $t \to T_+ \leq +\infty$.

{\rm (iv)} In any dimension $d > 6$ one can expect an analogous result with concentration rate $\lambda(t) \sim |t|^{-\frac{4}{d-6}}$.
\end{remark}

According to Theorem~\ref{thm:srreect123}, for $d = 3$ such a solution has to behave asymptotically as a decoupled superposition
  of stationary states. Such solutions are called (pure) \emph{multi-bubbles} (or $n$-bubbles, where $n$ is the number of bubbles).
  By conservation of energy, if $\vec u(t)$ is an $n$-bubbles, then
  \begin{equation}
    E(\vec u(t)) = n E(\vec W).
  \end{equation}
  The case $n = 1$ in dimension $d \in \{3, 4, 5\}$ was treated by Duyckaerts and Merle \cite{DM08},
  who obtained a complete classification of solutions of (FNLW) at energy level $E(\vec u(t)) = E(\vec W)$.

\begin{theorem}[Nonexistence of radial two-bubbles with opposite signs, \cite{moi15p-3}]
  \label{thm:nondeux-bulles}
  Let $d \geq 3$. There exists no radial solutions $\vec u: [t_0, T_+) \to {\dot{H}^1\times L^2}$ of {\rm (FNLW)} such that
  \begin{equation}
    \label{eq:deux-bulles}
    \lim_{t \to T_+}\|\vec u(t) - \vec W_{\lambda_1(t)} + \vec W_{\lambda_2(t)}\|_{\dot{H}^1\times L^2} = 0
  \end{equation}
  and
  \begin{itemize}
    \item The case $T_+ < +\infty$: $\lambda_1(t)\ll \lambda_2(t)\ll T_+-t\text{ as }t\to T_+$,
    \item The case $T_+ = +\infty$: $\lambda_1(t)\ll \lambda_2(t)\ll t\text{ as }t\to +\infty$.
  \end{itemize}
\end{theorem}

\subsection{Global dynamics away from the ground state}
Denote
$$ \vec u:=(|\nabla| u,\dot u) \in L^2_{\rm rad}(\R^d)^2\triangleq\mathcal{H},$$
and
$$\mathcal{S}\triangleq \{W_\lambda\}_{\lambda>0},\;\;~W_{\lambda} = T_\lambda W(x)=\lambda^{d/2-1}W(\lambda x),$$
and
$$J(\varphi) := \int_{\R^{d}} \Bigl[ \frac12 |\nabla \varphi^{2} - \frac{1}{2^*} |\varphi|^{{2^*}} \Bigr] dx.$$

\begin{theorem}[Global dynamics away from the ground state, \cite{KNW}]\label{thm: gdyaway}
There exist a small $\epsilon_*>0$, a neighborhood $\mathcal{B}$ of $\vec{\mathcal{S}}$
within $O(\epsilon_*)$ distance in $\mathcal{H}$, and a continuous functional
$$\mathcal{G}:\{\vec{\varphi}\in\mathcal{H}\setminus\mathcal{B} \mid E(\vec{\varphi})<J(W)+\epsilon_*^2\} \to \{\pm 1\},$$
 such that the following properties hold: For any solution $u$ with $E(\vec u)<J(W)+\epsilon_*^2$ on the maximal existence interval $I(u)$, let
\begin{align*}
  I_0(u):=&\{t\in I(u)\mid \vec u(t)\in \mathcal{B}\},\\
  I_\pm(u) :=& \{t\in I(u) \mid \vec u(t)\not\in\mathcal{B},\ \mathcal{G}(\vec u(t))=\pm 1\}.
\end{align*}
Then $I_0(u)$ is an interval, $I_+(u)$ consists of at most two infinite intervals, and $I_-(u)$ consists of at most two finite intervals. $u(t)$ scatters to $0$ as $t\to\pm\infty$ if and only if $\pm t\in I_+(u)$ for large $t>0$. Moreover, there is a uniform bound $M<\infty$ such that
$$
 \|u\|_{L^q_{t,x}(I_+(u)\times\R^d)}\le M, q:=\frac{2(d+1)}{d-2}.$$
For each $\sigma_1,\sigma_2\in\{\pm\}$, let $A_{\sigma_1,\sigma_2}$ be the collection of initial data $\vec u(0)\in\mathcal{H}$ such that $E(\vec u)<J(W)+\epsilon_*^2$, and for some $T_-<0<T_+$,
$$
 (-\infty,T_-)\cap I(u) \subset I_{\sigma_1}(u), (T_+,\infty)\cap I(u) \subset I_{\sigma_2}(u).$$
Then each of the four sets $A_{\pm,\pm}$ is open and non-empty, exhibiting all possible combinations of scattering to zero/finite time blowup as $t\to\pm\infty$, respectively.
\end{theorem}

\subsection{Unconditional uniqueness}

Consider energy-critical NLW
\begin{equation}\label{equ:nlwenergy}
\begin{cases}
u_{tt}-\Delta u=\mu|u|^\frac4{d-2}u,\\
(u,\pa_tu)(0)=(u_0,u_1)\in\dot{H}^1\times L^2.
\end{cases}
\end{equation}
It is well known that this equation is  locally well-posedness in the energy space, $u\in C_t(\dot{H}^1)$.
Furthermore, this  solution can be
extended to a global one in the defocusing case (i.e. $\mu=-1$).
However uniqueness holds only with some additional assumptions,
like $u\in L_{t,x}^\frac{2(d+1)}{d-2}$ (strong Strichartz solution).
 This restriction can be related to the choice of spaces involved in the fixed point argument. Therefore, it is  natural to ask whether such assumptions are necessary to provide uniqueness,
 or if uniqueness in the energy class holds. Uniqueness in the energy class for the defocusing case is known to hold
 under an additional a priori assumption of type $\pa_t u\in L_t^1L_x^4$ (or assuming smoothness on $u$, see \cite{struwe99}).
 By {\bf unconditional
uniqueness}, we mean that for given initial data $(u_0,u_1)$, there
exists at most one solution of \eqref{equ:nlwenergy} in the class $C_t \dot H_x^1(I\times \R^d)$,

\begin{center}{\bf Unconditional
uniqueness}\\
\begin{tabular}{|c|c|c|}
  \hline
 $d=3$ & $d=4,5$ & $d\geq6$\\\hline
  Open!  & Planchon \cite{P04} & Bulut-et al \cite{BCLPZ} \\\hline
\end{tabular}
\end{center}
We note that we do not need any assumption on $\pa_tu$.

\begin{remark}
{\rm (i)} We should also stress that the unconditional uniqueness in $d=3$ is
still {\bf open} due to the failure of the endpoint Strichartz estimates
except for the radial case. However, Masmoudi-Planchon in \cite{MP06} proved  an interesting result concerning uniqueness of
weak solutions ($u\in L_t^\infty\dot{H}^1\cap \dot{W}_t^{1,+\infty}(L_x^2)$) to defocusing NLW in $\mathbb R^3$ under a local energy
inequality assumption on the light cone as follows: for all $0\leq s \leq t \leq t_0$
\begin{eqnarray}
  \label{eq:lei1}
  \int_{B(x_0,t_0-t)} e(u(t,x)) dx
 \leq C    \int_{B(x_0,\alpha(t_0-s))} e(u(s,x)) dx,
\end{eqnarray}
and
\begin{equation}
  \label{eq:lei}
 \frac1 {\sqrt 2} \int_{t_1}^{t'} \int_{\partial
 B(x_1,\tau-t_1)}  \frac{|\partial_{K_1} u(\tau)|^2} 2+\frac{|u(\tau)|^6} 6  d\sigma d\tau
  \leq  C  \int_{B(x_1,\alpha(t'-t_1))} e(u(t',x)) dx,
\end{equation}
where  $C \geq 1$, $\alpha \geq 1$ and
$$e(u)=|\partial u|^2/2+|u|^6/6,\;\;  K_1 = \{ |x-x_1| = t -t_1\}.$$

In many cases we can only construct global weak solutions by using compactness arguments,
but their uniqueness is not known.
The weak-strong uniqueness investigation is an attempt to reconcile the weak
and strong viewpoints of solutions.
More precisely, this investigation is to show that any weak solution agrees with the strong solution sharing the same
 initial data if it exists. See, for instance, \cite{struwe99, Struwe-IMRN} for wave and Schr\"odinger equations,
  \cite{Chemin, Chen-Miao-Zhang-1, Germ3, Leray, May} for the Navier-Stokes system,
  \cite{Germ1} for nonhomogeneous Navier-Stokes system and \cite{Germ2} for the isentropic compressible Navier-Stokes system.

{\rm (ii)} In the context of {\bf $\dot H^s$ critical NLS}, the unconditional uniqueness
was first established
by Furioli and Terraneo \cite{FT03} using para-product analysis. A review of the unconditional uniqueness
for both the NLS and NLW can be found in the paper by Furioli, Planchon and Terraneo \cite{FPT03}.
\end{remark}

\subsection{Ill-posedness}

For the ill-posedness, we refer to
Christ, Colliander, Tao \cite{cct1}.

%%%%%%%%%%%%%%%%%%%%%%%%%%%%%%%%%%%%%%%%%%%%%%%%%%%%%%%%%%%%%%%%%%%%%%%%%%%%%%%%%%%%%%%%%%%%%%%%%%%%%%%%%%%%%%%%%%%%%%%%%%%%%%%%%%%%%%%%%%%%%%

%%%%%%%%%%%%%%%%%%%%%%%%%%%%%%%%%%%                                %%%%%%%%%%%%%%%%%%%%%%%%%%%%%%%%%%%%%%%%%%%%%%%%%%%%%%%%%%%%%%%%%%%%%%

%%%%%%%%%%%%%%%%%%%%%%%%%%%%%%%%%%%%%%%%%%%%%%%%%%%%%%%%%%%%%%%%%%%%%%%%%%%%%%%%%%%%%%%%%%%%%%%%%%%%%%%%%%%%%%%%%%%%%%%%%%%%%%%%%%%%%%%%%%%%%%

\section{Energy-supcritical NLW}

Let us recall the classical results for the
 nonlinear wave equation(NLW).
\begin{align} \label{equ:nlwesup}
{\rm (NLW)}\;\;\;\;\begin{cases}    \partial_{tt}u-\Delta u+\mu|u|^{p-1}u=0,\quad
(t,x)\in\R\times\R^d,
\\
(u,\pa_tu)(0,x)=(u_0,u_1)(x),\quad x\in\R^d,
\end{cases}
\end{align}
where $u:\R_t\times\R_x^d\to \R$.
 The equation \eqref{equ:nlwesup} can be classified  energy-subcritical, energy-critical
  and  energy-supercritical such as
\begin{equation*}\left\{\begin{aligned}
&p<1+\tfrac4{d-2},\;\; \text{\rm energy-subcritical};\\
&p=1+\tfrac4{d-2},\;\; \text{\rm energy-critical};\\
&p>1+\tfrac4{d-2},\;\; \text{\rm energy-supcritical}.\end{aligned}\right.\end{equation*}
The energy-critical equations have been the most widely studied
instances of NLW, since the rescaling
\begin{equation}
(u,u_t)(t,x)\mapsto~\big(\la^\frac2{p-1}u(\la t,\la
x),\la^{\frac2{p-1}+1}u_t(\la t,\la x)\big), \;\;p=1+\tfrac4{d-2}
\end{equation}
leaves invariant the energy of solutions, which is defined by
$$E(u,u_t)=\frac12\int_{\R^d}\big(|\pa_tu|^2+|\nabla
u|^2\big)\mathrm{d}
x+\frac{\mu}{p+1}\int_{\R^d}|u|^{p+1}\mathrm{d}x,$$ and is a
conserved quantity for equation in \eqref{equ:nlwesup}.
For the defocusing
energy-critical NLW \eqref{equ:nlwesup} with $\mu=1$, Grillakis \cite{Gri90} proved that the Cauchy problem of  \eqref{equ:nlwesup} with the initial  data $\dot
H^1(\mathbb R^3)\times L^2(\mathbb R^3)$ is  globally well-posed,  we  refer the readers to \cite{BG},  \cite{Kapi94},
\cite{ShaStr94} and  \cite{T07} for the scattering theory and the high dimensional case. In particular, Tao
derived a exponential type spacetime bound in \cite{T07}. In the
above papers, their methods rely heavily on the classical finite speed of
propagation (i.e. the monotonic local energy estimate on the light
cone)
\begin{align}\label{}
\int_{|x|\leq R-t}e(t,x)dx\leq \int_{|x|\leq R }e(0,x)dx,~~t>0
\end{align}where\begin{align}\label{}
  e(t,x):=\frac12  | u_t
 |^2  + \frac12  | \nabla u  |^2  + \tfrac{d-2}{2d} |
u|^{\frac{2d}{d-2}},
\end{align}
and Morawetz estimate
\begin{eqnarray}\label{morawetzes}
\int_{\R}\int_{\R^d}\frac{|u|^{2^*}}{|x|} dxdt\leq
C\big(E(u_0,u_1)\big), \;\;\; 2^*\triangleq\tfrac{2d}{d-2}.
\end{eqnarray}
However, the Morawetz estimate fails for the focusing energy-critical NLW.
 Kenig and Merle \cite{KM1} first employed
sophisticated {\bf ``concentrated compactness + rigidity method" } to
obtain the dichotomy-type result under the assumption that $E (u_0,
u_1) < E (W, 0)$, where $W$ denotes the ground state of the elliptic
equation
$$\Delta W+|W|^\frac4{d-2}W=0.$$ Thereafter,  Bulut et.al \cite{BCLPZ}
extended the above result in \cite{KM1} to higher dimensions.
This was proven by making use of minimal counterexamples derived from the
concentration-compactness approach to induction on energy.

\vskip0.2cm

Let $s_c=\frac{d}2-\frac2{p-1}$,  for the $\dot H^{s_c}$-critical NLW \eqref{equ:nlwesup} with
$p\neq1+\tfrac4{d-2}$, both  the Morawetz estimate and energy
conservation fail. It is hard to prove  global well posedness and
scattering of equation \eqref{equ:nlwesup}
 in  the space $\dot H^{s_c}\times \dot{H}^{s_c-1}$. Up to now, we do not know  how to
treat the large-data case  since there does not exist  any a priori control of a
critical norm. The first result in this
direction is due to Kenig-Merle \cite{KM2010}, where they studied
the $\dot H^\frac12$-critical Schr\"odinger equation in $\R^3$. For
the defocusing energy-supercritical NLW in odd dimensions, Kenig and
Merle \cite{KM2011,KM2011D} proved that if the radial solution $u$
is apriorily bounded in the critical Sobolev space, that is
$$(u,u_t)\in L_t^\infty(I; \dot{H}^{s_c}_x(\R^d)\times
\dot{H}^{s_c-1}_x(\R^d)), \quad s_c:=\tfrac{d}2-\tfrac2{p-1}>1,$$ then $u$
is global and scatters.  Later, Killip and Visan \cite{KV} showed
the result in $\R^3$ for the non-radial solutions by making use of
Huygens principle and so called ``localized double Duhamel trick".
We refer to \cite{Blut2012,KV2011,MWZ} for some high dimensional
cases.
%%Furthermore, they \cite{KV2011} proved the radial solution in all
%%dimensions in some ranges of $p$.Thereafter, Bulut \cite{Blut2012}
%%proved the results in dimensions $d\geq5$ for the cubic nonlinearity
%%(i.e. $p=2$). Lately, Miao, Wu and Zhang \cite{MWZ} extend the
%%result in \cite{KV2011} for dimension four to the non-radial
%%solution  in some ranges of $p$.
 Recently, Duyckaerts, Kenig and Merle \cite{DKM} obtain such result
 for the focusing energy-supercritical NLW with radial solution
in three dimension. Their proof relies on the compactness/rigidity
method, pointwise estimates on compact solutions obtained in
\cite{KM2011}, and the channels of energy method  developed in \cite{DKM11}.
Furthermore, by exploiting the double Duhamel trick, Dodson and
Lawrie\cite{DL14} extend the result in \cite{DKM} to dimension five.

\subsection{Critical norm conjecture}

\begin{conjecture}[{\bf Critical norm conjecture}]\label{con:con1.1}
Let $s_c\geq1$ and $\mu=-1$.  Suppose
$u:~I\times\R^d\to\R$ is a maximal-lifespan solution to
\eqref{equ:nlwesup} such that
\begin{equation}\label{assume1.1}
(u,\pa_tu)\in L_t^\infty(I; \dot{H}_x^{s_c}\times\dot{H}^{s_c-1}(\R^d)).
\end{equation}
 Then $u$ is global and scatters in the sense that  there exist unique $(u_0^{\pm},u_1^{\pm})\in\dot{H}^{s_c}_x(\R^d)\times\dot{H}^{s_c-1}(\R^d))$ such that
        $$\lim_{t\to\pm\infty}\|u(t)-S(t)(u_0^{\pm},u_1^{\pm})\|_{\dot{H}_x^{s_c}(\R^d)}=0.$$
\end{conjecture}

\noindent{Conjecture \ref{con:con1.1}}  implies that
\begin{equation}
 \label{Eqn:ResultKM123}
\limsup_{t\to T_+(u)}\|\vec{u}(t)\|_{\dot{H}^{s_c} (\mathbb{R}^{3}) \times \dot{H}^{s_c-1}(\mathbb{R}^{3})}<\infty
\Longrightarrow u\text{ scatters forward in time.}
\end{equation}

\vskip0.2cm

\noindent\underline{The well-known results on {\bf defocusing Conjecture}}:

\begin{center} \begin{tabular}{|c|c|}
  \hline
  $d=3$ & $p>4$ \\\hline
  \text{radial} & Kenig-Merle \cite{KM2011} \\\hline
  \text{non-radial} & Killip-Visan \cite{KV}, p-1\ \text{even} \\
  \hline
\end{tabular}
\vskip0.3cm

\begin{tabular}{|c|c|c|c|}
  \hline
  $d=4$&$3<p<5$&$p=5$&$p>5$ \\\hline
  \text{radial} & Killip-Visan\cite{KV2011} & & \text{Miao-Wu-Zhang} \\\hline
  \text{non-radial} &  \text{Miao-Wu-Zhang} & \text{Miao-Wu-Zhang}& \text{Open!}\\
  \hline
\end{tabular}

\vskip0.3cm

\begin{tabular}{|c|c|c|c|c|}
  \hline
  $d=5$ &$7/3<p<3$&$p=3$&$3<p<5$&$p\geq5$ \\\hline
  \text{radial} & Killip-Visan\cite{KV2011} & &Open!& Kenig-Merle \cite{KM2011D}\\\hline
  \text{non-radial} & Easy  &Bulut \cite{Blut2012}& \text{Open!}&\text{Open!}\\
  \hline
\end{tabular}

\vskip0.3cm

\begin{tabular}{|c|c|c|c|c|}
  \hline
  $d=6$ &$2<p<7/3$&$7/3<p<3$&$p=3$&$p-1$ even or $p>3$ \\\hline
  \text{radial} & Killip-Visan\cite{KV2011} & &&\\\hline
  \text{non-radial} & Easy  &Easy & Bulut\cite{Blut2012}&Easy \\
  \hline
\end{tabular}

\vskip0.3cm

\begin{tabular}{|c|c|c|c|}
  \hline
  $d\geq7$ &$\frac{4}{d-2}<p-1<\frac{d(d-1)-\sqrt{d^2(d-1)^2-16(d+1)^2}}{2(d+1)}$&$p=3$& other \\\hline
  \text{radial} & Killip-Visan\cite{KV2011} & &  \\\hline
  \text{non-radial} &  Easy & Bulut\cite{Blut2012}& Easy\\
  \hline
\end{tabular}
\end{center}
It is  OPEN: $d=4,5,~p>1+\tfrac{4}{d-3}$!
\vskip 0.2in
\noindent\underline{The well-known results on {\bf focusing Conjecture}}:
\begin{center}
\begin{tabular}{|c|c|c|c|}
  \hline
  & radial & nonradial \\\hline
  $d=3$ & Duyckaerts-Kenig-Merle \cite{DKM}  & open! \\\hline
  $d=5$ & Dodson -Lawrie\cite{DL14} & open! \\\hline
  $d$ \text{even} & open & open! \\
  \hline
 \end{tabular}

\end{center}

\subsection{Nonscattering radial solution} \; Duyckaerts and Roy in \cite{DuRoy} proved the existence of
nonscattering radial solution to supercritical wave equation.

\begin{theorem}[Nonscattering radial solution, Duyckaerts-Roy \cite{DuRoy}]\label{thm:nonscarai}
Assume $d=3$ and $p>5$.
Let $u$ be a solution of $u_{tt}-\Delta u\pm|u|^{p-1}u=0$ with radial data $(u_{0},u_{1}) \in \dot{H}^{s_{c}}(\mathbb{R}^{3}) \times
\dot{H}^{s_{c}-1}(\mathbb{R}^{3})$. Then
\begin{itemize}
\item either $T_{+}(u) <+ \infty$ and
\begin{align}
\lim_{t \rightarrow T_{+}(u)} \left\| ( u(t), \partial_{t} u(t) ) \right\|_{\dot{H}^{s_{c}}(\mathbb{R}^{3}) \times \dot{H}^{s_{c}-1}(\mathbb{R}^{3})}
=+\infty,\label{Eqn:Blowuptype1}
\end{align}
\item or $T_{+}(u) =+ \infty$ and $u$ scatters forward in time.
\end{itemize}
An analogous statement holds for negative times.
\end{theorem}

\vskip0.2cm

\noindent\underline{\bf Question: How about NLS?}

\vskip 0.1in

Theorem \ref{thm:nonscarai} is equivalent to \eqref{Eqn:ResultKM123}
 with the limit superior replaced by a limit inferior. We do not know any direct application of this qualitative improvement,
 however its analog in the case $p=5$, is crucial in the proof of the soliton resolution conjecture for the energy-critical wave equation
 in \cite{DKM13}.

\vskip0.15cm

Theorem \ref{thm:nonscarai} means exactly that solutions of 3d energy-supercritical NLW are of one of {\bf the following three types}:
scattering solutions, solutions blowing-up in finite time with a critical norm going to infinity at the maximal time of existence,
 and global solutions with a critical norm going to infinity for infinite times.

\vskip0.15cm

In the defocusing case, it is conjectured that all solutions  with initial data in the critical space $\dot{H}^{s_c}$ scatter.
The difficulty of this conjecture is of course the lack of conservation law at the level of crirtical regularity.
The only supercritical dispersive equations for which scattering was proved for all solutions are wave and Schr\"odinger equations
 with defocusing barely supercritical nonlinearities: see e.g \cite{RoyLog,RoyLogSchrod,Shih,TaoLog}
  \footnote{In \cite{RoyLog} the scattering with data in $\tilde{H}^{2}:=\dot{H}^{2}\cap \dot{H}^1 (\mathbb{R}^3)$
   is not explicitely mentionned. However, it can be easily derived from the finite bounds of the $L_{t}^{4}(\R, L_{x}^{12})$
   norm and the $L_{t}^{\infty}(\R, \tilde{H}^2)$ norm of the solution, by using a similar argument to that in \cite{RoyLogSchrod}
    to prove scattering.}.

\vskip0.15cm

In the focusing case, solutions blowing up in finite time are known. One type of finite time blow-up is given by
the ODE $y''=|y|^{p-1}y$, and is believed to be stable: see \cite{DonningerSchorkhuber,DonningerSchorkhuber2}
for stability inside the wave cone.

\subsection{Stability in $H^2\times H^1$ of the ODE blowup profile}

Consider the Cauchy problem for the semilinear wave
equation
\begin{equation}\label{eq:main}
\partial_t^2u - \Delta u=|u|^{p-1}u,\qquad p>3
 \end{equation}
for $x\in \R^3$.

First, recall that
the existence of finite-time blowup is most easily seen by ignoring the
Laplacian in \eqref{eq:main}. The remaining ODE in $t$ can be solved explicitly
which
 leads to the solution
\[ u_1(t,x)=c_p (1-t)^{-\frac{2}{p-1}},\qquad
c_p=\Big(~\tfrac{2(p+1)}{(p-1)^2}~\Big)^\frac{1}{p-1}. \]
Obviously, $u_1$ becomes singular at $t=1$.
One might object that this solution does not have data in $H^2\times H^1(\R^3)$.
However, this
defect is easily fixed by using suitable cut-offs and exploiting finite speed of propagation.
By using symmetries of the equation one can in fact produce a much larger family of blowup
solutions.
For instance, the time translation symmetry immediately yields the one-parameter family
\[ u_T(t,x):=c_p (T-t)^{-\frac{2}{p-1}} \]
and by applying the Lorentz transform $\Lambda_T(a)$,
we can even generate the 4-parameter family
\[ u_{T,a}(t,x)\triangleq u_T(\Lambda_T(a)(t,x)) \]
of explicit
blowup solutions,  Lorentz boosts is as follows: for $a=(a^1,a^2,a^3)\in \R^3$ and $T\in \R$,
we set
$$\Lambda_T(a)=\Lambda_T^3(a^3)\Lambda_T^2(a^2)\Lambda_T^1(a^1),$$
where $\Lambda^j: \R\times \R^3\to \R\times \R^3$, $j\in \{1,2,3\}$, is
defined by
\begin{align*}
\Lambda_T^j(a^j): \left\{ \begin{aligned}
&t \mapsto (t-T)\cosh a^j+x^j \sinh a^j+T, \\
&x^k \mapsto x^k+\delta^{jk}[(t-T)\sinh a^j+x^j \cosh a^j-x^j].
\end{aligned} \right.
\end{align*}
Merle-Zaag in [\cite{MZ03},\cite{MZ05}\cite{MZ2}] proved the stability of the above family of explicit blowup solutions
in the following sense.

\begin{theorem}[Stability in $H^2\times H^1$ of the ODE blowup profile, ]\label{thm:staode}
Fix $p>3$.
There exist constants $M,\delta>0$ such that the following holds.
Suppose $(f,g)\in H^2\times H^1(\R^3)$ satisfy
\[ \|(f,g)-u_{1,0}[0]\|_{H^2\times H^1({B}_{1+\delta}(x_0))}\leq \tfrac{\delta}{M} \]
for some $x_0\in \R^3$.
Then
$T:=T_{f,g}(x_0)\in [1-\delta,1+\delta]$ and there exists an
$a\in {B}_{Mp\delta}(x_0)$ such that
the solution $u: \Gamma(T_{f,g}(x_0),x_0)\to \R$ of ~\eqref{eq:main} with data
$u[0]=(f,g)$
satisfies
\begin{equation*}\left\{\begin{aligned}
&(T-t)^{-s_p+2}\|u[t]-u_{T,a}[t] \|_{\dot H^2\times \dot H^1(\mathbb{B}^3_{T-t}(x_0))}\lesssim
(T-t)^{\frac{1}{p-1}}, \\
&(T-t)^{-s_p+1}\|u[t]-u_{T,a}[t]
\|_{\dot H^1\times L^2(\mathbb{B}^3_{T-t}(x_0))}\lesssim
(T-t)^{\frac{1}{p-1}}, \\
&(T-t)^{-s_p}\|u(t,\cdot)-u_{T,a}(t,\cdot)
\|_{L^2(\mathbb{B}^3_{T-t}(x_0))}\lesssim
(T-t)^{\frac{1}{p-1}},
\end{aligned}\right.\end{equation*}
for all $t\in [0,T)$, where $s_p=\frac32-\frac{2}{p-1}$ is the critical Sobolev exponent.
\end{theorem}

\subsection{Type II blowup solutions}

 Let $d\geq 11$ and let the Joseph-Lundgren exponent be
\begin{equation}
\label{exponentpjl}
p_{JL}=1+\frac{4}{d-4-2\sqrt{d-1}}.
\end{equation}
Then for $p>p_{JL}$, the soliton profile admits an asymptotic expansion \begin{equation}
\label{expansionq}
  Q(r)=\frac{c_{\infty}}{r^{\frac{2}{p-1}}}+\frac{a_1}{r^{\gamma}}+o\left(\frac1{r^{\gamma}}\right), \; \; a_1\neq 0, \;\;\;\Delta Q+Q^p=0,
\end{equation}
  with
 \begin{equation}
 \label{defgamma}
 c_\infty=\left[\frac{2}{p-1}\left(d-2-\frac 2{p-1}\right)\right]^{\frac 1{p-1}}, \ \ \gamma=\frac12(d-2-\sqrt{{\triangle}})>0
 \end{equation}
 and where $${\triangle}=(d-2)^2-4pc_{\infty}^{p-1}\ \ (\triangle>0 \  \mbox{for}\ \ p>p_{JL}).$$ These numbers describing the asymptotic behavior of the soliton are essential in the analysis of type II blow up bubbles as follows:

\begin{theorem}[Type II blow up for the energy super critical wave equation, \cite{Collot}]\label{thm:typeiiesup}
Let $d\geq 11$, $p_{JL}$ be given by \eqref{exponentpjl} and a nonlinearity
\begin{equation}
\label{plarge}  p=2q+1,\ \ q\in \Bbb N^*,\ \ p> p_{JL}.
\end{equation}
Let $\gamma$ be given by \eqref{defgamma} and define:
\begin{equation}
\label{intro:eq:def alpha}
\alpha=\gamma-\frac{2}{p-1}.
\end{equation}
Assume moreover:
\begin{equation}
\label{conditiongamma}
\left(\frac d2-\gamma\right)\notin \Bbb N.
\end{equation}
Pick an integer
\begin{equation}
\label{codnionrionfeell}
 \ell\in \Bbb N \; \; \mbox{with}\; \; \ell>\alpha,
\end{equation}
 and a large enough regularity exponent
$$s_+\in \Bbb N, \; \; s_+\geq s(\ell) \;\; (s(\ell) \to +\infty\; \; \mbox{as}\; \; \ell\to +\infty).$$
 Then there exists a radially symmetric initial data $(u_0, u_1)\in H^{s_+}\times H^{s_+-1}(\Bbb R^d)$
  such that the corresponding solution to {\rm (NLW)} blows up in finite time $0<T<+\infty$ by concentrating the soliton profile:
\begin{equation}
\label{concnenergy}
u(t,r)=\frac{1}{\l(t)^{\frac 2{p-1}}}(Q+\varepsilon)\left(\frac{r}{\l(t)}\right)
\end{equation}
with:\\
\noindent{\em (i) Blow up speed}:
\begin{equation}
\label{Pexciitedlaw}
\l(t)=c(u_0)(1+o_{t\uparrow T}(1))(T-t)^{\frac{\ell}{\alpha}}, \ \ c(u_0)>0;
\end{equation}
\noindent{\em (iii) Asymptotic stability above scaling in renormalized variables}:
\begin{equation}
\label{intro:eq:convergence surcritique}
\lim_{t\uparrow T}\|\varepsilon(t,\cdot),\lambda(\partial_t u)_{\lambda}(t,\cdot)\|_{\dot{H}^s\times \dot{H}^{s-1}}=0  \ \ \mbox{for all}\ \ s_c<s\leq s_+;
\end{equation}
\noindent{\em (iv) Boundedness below scaling}:
\begin{equation}
\label{intro:eq:bornitude sous critique}
\limsup_{t\uparrow T}\|u(t),\partial_t u(t)\|_{\dot{H}^s\times \dot{H}^{s-1}}<+\infty  \ \ \mbox{for all}\ \ 1\leq s<s_c;
\end{equation}
\noindent{\em (v) Behavior of the critical norms}:
\begin{equation}
\label{intro:eq:comportement norme critique}
\|u(t)\|_{\dot{H}^{s_c}}=\left[c(d,p)\sqrt{\ell}+o_{t\uparrow T}(1)\right]\sqrt{|\log(T-t)|},
\end{equation}

\begin{equation}
\label{intro:eq:comportement norme critique 2}
\limsup_{t\uparrow T} \|\partial_t u(t)\|_{\dot{H}^{s_c-1}}<+\infty.
\end{equation}

\end{theorem}

The proof of Theorem \ref{thm:typeiiesup} relies on an explicit construction of blow up solutions. It allows us to find a whole set of initial data leading to such a blow up, and to investigate its topological properties:

\begin{theorem} \label{thm:thmmain2}
We keep the notations and assumptions of Theorem \ref{thm:typeiiesup}. Let a slightly supercritical regularity exponent $\sigma=\sigma(\ell)$ satisfying:
$$
0<\sigma-s_c\ll 1 \ \ \text{small enough}.
$$
There exists a locally Lipschitz manifold of codimension $\ell-1$ in the Banach space $\dot{H}^{\sigma}\cap \dot{H}^{s_+}\times \dot{H}^{\sigma-1}\cap \dot{H}^{s_+-1}$ of initial data leading to the blow up scenario described by Theorem \ref{thm:typeiiesup}. We point out that as $\alpha>2$, the codimension satisfies $\ell-1>2$.
\end{theorem}

%%%%%%%%%%%%%%%%%%%%%%%%%%%%%%%%%%%%%%%%%%%%%%%%%%%%%%%%%%%%%%%%%%%%%%%%%%%%%%%%%%%%%%%%%%%%%%%%%%%%%%%%%%%%%%%%%%%%%%%%%%%%%%%%%%%%%%%%%%%%%%

%%%%%%%%%%%%%%%%%%%%%%%%%%%%%%%%%%%                                %%%%%%%%%%%%%%%%%%%%%%%%%%%%%%%%%%%%%%%%%%%%%%%%%%%%%%%%%%%%%%%%%%%%%%

%%%%%%%%%%%%%%%%%%%%%%%%%%%%%%%%%%%%%%%%%%%%%%%%%%%%%%%%%%%%%%%%%%%%%%%%%%%%%%%%%%%%%%%%%%%%%%%%%%%%%%%%%%%%%%%%%%%%%%%%%%%%%%%%%%%%%%%%%%%%%%

\section{Energy-subcritical NLW}

Consider energy subcritical NLW equation
\begin{align} \label{equ:waveequ-5.1}
{\rm (NLW)}\quad \begin{cases}    \partial_{tt}u-\Delta u= \mu|u|^{p-1}u,\quad
(t,x)\in\R\times\R^d,
\\
(u,\pa_tu)(0,x)=(u_0,u_1)(x),
\end{cases}
\end{align}
where $u:\R_t\times\R_x^d\to \R$, $s_c=\frac{d}2-\frac2{p-1}<1.$
\vskip0.15cm

\underline{\bf History of defocusing energy-subcritical NLW: $\mu=-1,~s_c<1.$}
\begin{enumerate}
\item  For $d=3$ and $1<p<p_c=5$, J\"orgen in 1961  proved the global existence of smoothing solution of
\eqref{equ:waveequ-5.1}\cite{Jor}.
\item For higher dimensions with $4\le d\le 9$,  Brenner, Wahl and Pecher proved the global existence of smoothing solution of
\eqref{equ:waveequ-5.1}\cite{BW,Pech,Wa}.
\item  Ginibre-Velo proved the global well posedness of energy solution of
\eqref{equ:waveequ-5.1} with intial data $(u_0,u_1)\in H^1\times L^2$  \cite{GV1,GV2}.
\end{enumerate}

\subsection{Critical norm conjecture}

 The methods employed in the energy-supercritical NLW
 also lead to the study of the energy-subcritical
NLW. In fact, using  the channels of energy method pioneered in
\cite{DKM11,DKM13}, Shen \cite{Shen} proved the analog result of
\cite{KM2011} for $2<p<4$ with $d=3$ in both defocusing and focusing
case. However, the channels of energy  method does not work so
effectively for $p\leq2$. Recently, by virial based rigidity
argument and improving addition regularity for the minimal
counterexamples,  Dodson and Lawrie \cite{DL} extended the result of
\cite{Shen} to $\sqrt{2}<p\leq2$. How to establish  monotonicity
formulae at a different regularity matching with the critical
 conservation laws
is a difficulty intrinsic to the nonlinear wave equation. In order
to  utilize the monotonicity formulae, one needs to  improve the
regularity for the almost periodic solutions. In \cite{DL,Shen},
they used the double Duhamel trick to show that almost periodic
solutions belong to energy space $\dot H^1_x(\R^3)\times
L^2_x(\R^3)$. The main difficult is that the decay rate of the
linear solution is not enough to guarantee the double Duhamel
formulae converges. However, the weighted decay available from
radial Sobolev embedding can supply the additional decay to
guarantee the double Duhamel formulae converges. Thus, one need the
radial assumption in \cite{DL,Shen}.
 This is different
 from cubic Schr\"odinger equation \cite{KM2010}, where no radial assumption is made. This is due to the Lin-Strauss Morawetz inequality
$$\int_{I}\int_{\R^{3}}\frac{|u(t,x)|^4}{|x|}\;dx\;dt\lesssim\sup_{t\in I}\|u(t,x)\|_{\dot H^{\frac12}_x(\R^3)}^2,$$
which is scaling critical with the cubic Schr\"odinger equation.

\begin{conjecture}[{\bf Critical norm conjecture}]\label{con:con1.123}
Let $s_c>0$ and $\mu=-1$.  Suppose
$u:~I\times\R^d\to\R$ is a maximal-lifespan solution to
\eqref{equ:waveequ-5.1} such that
\begin{equation}\label{assume1.112}
(u,\pa_tu)\in L_t^\infty(I; \dot{H}_x^{s_c}\times\dot{H}^{s_c-1}).
\end{equation} Then $u$ is global and scatters in the sense that  there exist unique
$(u_0^{\pm},u_1^{\pm})\in\dot{H}^{s_c}_x(\R^d)\times \dot{H}^{s_c-1}_x(\R^d)$
such that
        $$\lim_{t\to\pm\infty}\|u(t)-S(t)(u_0^{\pm},u_1^{\pm})\|_{\dot{H}_x^{s_c}(\R^d)}=0.$$
\end{conjecture}

\vskip0.2cm

\begin{center}Defocusing/focusing\\
\begin{tabular}{|c|c|c|}
  \hline
 & $\sqrt{2}<p\leq2$ & $2<p<4$\\\hline
  $d=3$  & Dodson-Lawrie \cite{DL} & Shen \cite{Shen} \\\hline
  $d\geq4$ & & \\
  \hline
\end{tabular}
\end{center}

\vskip0.2cm

\noindent\underline{\bf The clue for the proof in \cite{DL})} \vskip0.2cm

The proof follows the concentration-compactness approach  to induction on energy. We
argue by contradiction. The failure of scattering result would
imply the existence of very special class of solutions. On the other
hand, these solutions have so many good properties that they do not
exist. Thus we get a contradiction. While we will make some further
reductions later, the main property of the special counterexamples
is almost periodicity modulo symmetries:

\vskip0.2cm
{\bf Step 1}\; Reduction to almost periodic solution

\begin{theorem}[Reduction to two enemies]
Suppose that Conjecture \ref{con:con1.123} for $\dot{H}^{\frac12}$-critical wave equation fails,
then there exists a minimal $\dot H^{\frac12}$-energy, maximal-lifespan solution $u: ~I\times\R^3\to\R$ to
\eqref{equ:waveequ-5.1}, which is almost periodic modulo symmetries,
$\|u\|_{L_{t,x}^{4}(I\times\R^3)}=+\infty$.  The almost periodic modulo symmetries means: for $\forall \eta>0$,
there exist $C(\eta)>0$ and $N(t)$ defined in $I$ such that
\begin{align*}&\left\{\begin{aligned}
&\int_{|x|\geq\frac{C(\eta)}{N(t)}}\big||\nabla|^{\frac12}u\big|^2dx+\int_{|\xi|\geq
C(\eta)N(t)}|\xi|\cdot\big|\hat{u}(t,\xi)\big|^2d\xi<\eta,\\
&\int_{|x|\geq\frac{C(\eta)}{N(t)}}\big||\nabla|^{-\frac12}u_t\big|^2dx+\int_{|\xi|\geq
C(\eta)N(t)}|\xi|^{-1}\cdot\big|\hat{u}_t(t,\xi)\big|^2d\xi<\eta.\end{aligned}\right.
\end{align*}
Moreover,
\begin{enumerate}
\item[{\rm (i)}] $N(t)\equiv1,\quad \forall~t\geq0$.\;\quad\qquad \qquad{\rm (soliton-like
solution)}
\item[{\rm (ii)}] $N(t)\leq1$ and $\varlimsup\limits_{t\to+\infty}N(t)=0$. \quad
{\rm (Energy-cascade)}.
\end{enumerate}
\end{theorem}

\underline{The properties of almost periodic solution}:

\begin{lemma}[No-waste Duhamel formula] Let $u:~(T_-,T_+)\times\R^3\to\R$
be an almost  periodic solution, then
\begin{equation}\label{nwd3dradw}
\vec{u}(t)\triangleq\big(u,u_t\big)(t)=\lim_{T\to
T_\pm}\int_t^TS(t-s)(0,\pm|u|^2u)(s)ds,
\end{equation}
where the limit is to be understood in the weak $L^2$ topology.
\end{lemma}

\begin{proposition}[Negative energy implies that the finite time blow-up]\label{propnegebf}
Let $u:~(T_-,T_+)\times\R^3\to\R$ be the   maximal-lifespan solution for  3d focuing cubic wave equation
with  $E(u_0,u_1)\leq0$.  Then either $u\equiv0$, or
$u$ blows up both forward and backward in finite time ( i.e. $T_->-\infty$ and $T_+<\infty$).

\end{proposition}

\vskip0.2cm

{\bf Step 2}. Additional regularity and preclude energy-critical cascade.
\begin{theorem}[Additional regularity]\label{ewzzx3drw} Let $T_-<0$, and
$u:~(T_-,+\infty)\times\R^3\to\R$ be an almost periodic solution with $N(t)\leq1$, then
 $\vec{u}(t)\in \dot H^1\times L^2$ with
\begin{equation}\label{ewzzxgj3dw}
\|\vec u(t)\|_{\dot H^1\times L^2}\lesssim N(t)^\frac12.
\end{equation}
\end{theorem}

\begin{remark}\; Exponent $1/2$ is fully determined by scaling invariant of (NLW).  In fact,
$$\frac{u}{(\Delta t)^2}\simeq u^3\Longrightarrow~u\simeq (\Delta
t)^{-1}\simeq N(t),$$
hence
$$\int_{\R^3}|\nabla u|^2dx\simeq \frac{|u|^2}{(\Delta x)^2}(\Delta
x)^3\simeq |u|^2(\Delta x)\simeq N(t).$$
\end{remark}
\vskip0.2cm

 \underline{\bf To preclude  energy-cascade solution:} \;  By making use of Theorem \ref{ewzzx3drw}
and Sobolev embedding theorem, we have
\begin{equation*}
\begin{cases}
\varlimsup\limits_{t\to\infty}\|\vec u(t)\|_{\dot H^1\times
L^2}\lesssim
\varlimsup\limits_{t\to\infty} N(t)^\frac12=0,\\
\|u(t)\|_{L^4}^2\lesssim\|u\|_{\dot H^\frac34}^2\lesssim \|u\|_{\dot
H^\frac12}\|u\|_{\dot H^1}\to0,
\end{cases}\Longrightarrow~E(u_0,u_1)=0.
\end{equation*}
This  fact together with Proposition \ref{propnegebf} precludes the energy-cascade solution.
Therefore, to remove energy-cascade solution can be reduced to Theorem \ref{ewzzx3drw}.

\vskip0.2cm

{\bf Step 3}.   Additional compactness  and to preclude quasi-soliton solution.
\vskip.15cm

Due to $N(t)\equiv1$, this almost periodic solution enjoys compactness modulo symmetries.

\begin{proposition}[Compactness and more additional regularity]\label{gdzzx3drw}
Let  $u:~\R\times\R^3\to\R$ be a almost periodic solution with $N(t)\equiv1$, then
$$~K\triangleq\big\{\vec{u}(t):~t\in\R\big\}$$
is a  precompact in $(\dot
H^\frac12\times\dot H^{-\frac12})\cap(\dot H^1\times L^2)$
This fact can be reduced to  prove $\vec u\in\dot
H^{1+}\times\dot H^{0+}$.
\end{proposition}

\begin{lemma}\label{vir3drw}
Let $\vec u\in\dot H^1\times L^2$ be the solution for 3d cubic wave equation,
then for any $R>0$,
\begin{align}
&\frac{d}{dt}\big\langle u_t,
\chi_R(u+ru_r)\big\rangle\nonumber\\
=&-\frac1{|S^2|}E(\vec
u)(t)+\int_0^\infty
(1-\chi_R)\Big(\frac12u_t^2+\frac12u_r^2\mp\frac14u^4\Big)r^2dr\nonumber\\
&-\int_0^\infty\Big(\frac12u_t^2+\frac12u_r^2\pm\frac14u^4\Big)r\chi_R'
r^2dr-\int_0^\infty uu_rr\chi_R'rdr,
\label{equlem1.23dw}
\end{align}
where
$$\langle f,g\rangle\triangleq \int_0^\infty
f(r)g(r)r^2dr,~\;\;~\chi_R(r)=\chi(r/R),~\;\;~\chi\in C_c^\infty(\R).$$
\end{lemma}
\begin{proof}[\rm Proof]
\begin{align*}
\frac{d}{dt}\big\langle u_t, \chi_R(u+ru_r)\big\rangle
=&\int u_{tt}\chi_R(u+ru_r)r^2dr+\int
u_t\chi_R(u_t+r\partial_ru_t)r^2dr\\
\triangleq&I_1+I_2.
\end{align*}

\noindent\underline{\bf To calculate  $I_2$:}
\begin{align*}
I_2=&\int |u_t|^2\chi_Rr^2dr+\int u_t\partial_ru_t\chi_Rr^3dr\\
=&\int |u_t|^2\chi_Rr^2dr-\frac12\int
|u_t|^2\partial_r\big(\chi_Rr^3\big)dr\\
=&-\frac12\int |u_t|^2\chi_Rr^2dr-\frac12\int |u_t|^2\chi_R'r^3dr.
\end{align*}

\noindent\underline{\bf  To calculate  $I_1$:}  Making use of
$$u_{tt}=\Delta u\pm u^3=\partial_{r}^2u+2r^{-1}\partial_ru\pm
u^3,$$ and taking it into  $I_1$, one easily verifies
that
\begin{align*}
I_1=&\int \partial_{r}^2u\chi_R(u+ru_r)r^2dr+2\int
\partial_ru\chi_R(u+ru_r)rdr\pm\int u^3\chi_R(u+ru_r)r^2dr\\
\triangleq&I_{11}+I_{12}\pm I_{13}.
\end{align*}

\noindent\underline{\bf To calculate   $I_{13}$:}
\begin{align*}
I_{13}=&\int \chi_R|u|^4r^2dr+\int u^3u_r\chi_Rr^3dr\\
=&\int \chi_R|u|^4r^2dr-\frac14\int
|u|^4\partial_r\big(\chi_Rr^3\big)dr\\
=&\frac14\int \chi_R|u|^4r^2dr-\frac14\int |u|^4\chi_R'r^3dr.
\end{align*}

\noindent\underline{\bf  To calculate  $I_{11}+I_{12}$:}
\begin{align*}
I_{11}=&\int \partial_{r}^2u u\chi_Rr^2dr+\int
\partial_{r}^2uu_r\chi_Rr^3dr\\
=&-\int|u_r|^2\chi_Rr^2dr-\int
u_ru\partial_r\big(\chi_Rr^2\big)dr-\frac12\int|u_r|^2\big(\chi_Rr^3\big)_rdr\\
=&-\frac52\int|u_r|^2\chi_Rr^2dr+\frac12\int|u|^2\partial_r^2\big(\chi_Rr^2\big)dr-\frac12\int|u_r|^2\chi_R'r^3dr,\\
I_{12}=&2\int
\partial_ruu\chi_Rrdr+2\int
|u_r|^2\chi_Rr^2dr\\
=&-\int |u|^2\big(\chi_Rr\big)_rdr+2\int
|u_r|^2\chi_Rr^2dr,\\
I_{11}+I_{12}=&-\frac12\int|u_r|^2\chi_Rr^2dr+\frac12\int|u|^2\partial_r\big(\chi_R'r^2\big)dr-\frac12\int|u_r|^2\chi_R'r^3dr\\
=&-\frac12\int|u_r|^2\chi_Rr^2dr-\int
uu_r\chi_R'r^2dr-\frac12\int|u_r|^2\chi_R'r^3dr.
\end{align*}
Hence,
\begin{align*}
&\frac{d}{dt}\big\langle u_t, \chi_R(u+ru_r)\big\rangle\\
=&-\int\Big[\frac12|u_r|^2+\frac12|u_t|^2\mp\frac14|u|^4\Big]\chi_Rr^2dr\\
&-\int\Big[\frac12|u_r|^2+\frac12|u_t|^2\pm\frac14|u|^4\Big]\chi_R'r^3dr-\int
uu_r\partial_r\chi_R'r^2dr\\
=&-\frac1{|{S}^2|}E(\vec u)+\int(1-\chi_R)\Big[\frac12|u_r|^2+\frac12|u_t|^2\mp\frac14|u|^4\Big]r^2dr\\
&-\int\Big[\frac12|u_r|^2+\frac12|u_t|^2\pm\frac14|u|^4\Big]\chi_R'r^3dr-\int
uu_r\chi_R'r^2dr.
\end{align*}
\end{proof}

\noindent{\bf \underline{To preclude quasi-soliton solution:}}
$${\bf  Claim}: \;\;\forall~\eta>0,\;~E(\vec{u})\leq C\eta.$$
This together with Propsition \ref{propnegebf} can kill  the  quasi-soliton solution.

\vskip0.2cm

{\bf Step 4}.  We shall use Propositin \ref{gdzzx3drw} and Lemma \ref{vir3drw}
to prove {\bf the above  Claim}.
\vskip0.15cm

First, compactness implies that for $\forall~
\eta>0$, there exists $R_0=R_0(\eta)$ such that when $R\geq R_0(\eta)$,
\begin{equation}\left\{
\begin{aligned}
&\int_R^\infty\big[u_t^2+u_r^2\big](t)r^2dr<\eta,\\
&\dot H^\frac12\cap\dot H^1\hookrightarrow\dot
H^\frac34\hookrightarrow L^4.
\end{aligned}\right. \Longrightarrow~\int_R^\infty
u^4(t)r^2dr<\eta, \;\; \forall \; t\in\mathbb R.\label{radial-embeding-0}
\end{equation}
Note that
\begin{equation}
\int_R^\infty f(r)^2dr+Rf(R)^2=-\int_R^\infty
f_r(r)f(r)rdr,~f\in\dot H^1
\end{equation}
it follows that
$$\int_R^\infty f(r)^2dr\leq\int_R^\infty
f_r(r)^2r^2dr\buildrel{\eqref{radial-embeding-0}}\over{=\!\!\Longrightarrow}~\int_R^\infty u^2(r,t)dr<\eta.$$
In fact,   it suffice to prove
$$\int_0^\infty u^2dr\lesssim
\int_0^\infty\Big|\frac{u}{r}\Big|^2r^2dr\lesssim\|u\|_{\dot
H^1}^2<\infty,$$
by  the Hardy inequality.   This further implies that
\begin{equation}\label{radial-embeding-1}
\left\{\begin{aligned} &\Big|\int_0^\infty
(1-\chi_R)\Big(\frac12u_t^2+\frac12u_r^2\pm\frac14u^4\Big)r^2dr\Big|\leq
C\eta,\\
&\big|\int_0^\infty\Big(\frac12u_t^2+\frac12u_r^2\pm\frac14u^4\Big)r\chi_R'
r^2dr \Big|\lesssim C\eta,\\
&\Big| \int_0^\infty uu_rr\chi_R'rdr\Big|\lesssim C\eta.
\end{aligned}\right.
\end{equation}
Therefore, we integrate \eqref{equlem1.23dw} with respect to  $r$ from  $0$ to  $T$ to see
\begin{align*}
E(\vec u)\leq& C\eta+\frac1T\Big|\big\langle
u_t(t),~\chi_R(u+ru_r)(t)\big\rangle\Big|_0^T\\
\leq&C\eta+\frac{C}T\int_0^T|u_t(T)|\cdot|u(T)|r^2dr+\frac{C}{T}\int_0^T|u_t(0)|\cdot|u(0)|r^2dr\\
&+\frac{C}T\int_0^T|u_t(T)|\cdot|u_r(T)|r^3dr+\frac{C}T\int_0^T|u_t(0)|\cdot|u_r(0)|r^3dr,
\end{align*}
where we choose $R=T$ with $T\gg R_0(\eta)$.

on the other hand,
\begin{align*}
\frac{1}T\int_0^T|u_t|\cdot|u|r^2dr\leq&\frac1T\Big(\int_0^T|u_t|^2r^2dr\Big)^\frac12
\Big(\int_0^T|u|^3r^2dr\Big)^\frac13\Big(\int_0^Tr^2dr\Big)^\frac16\\
\leq&\frac{C}{T^\frac12}\|u_t\|_2\|u\|_{\dot H^\frac12},\\
\frac{1}T\int_0^T|u_t|\cdot|u_r|r^3dr\leq&\frac{R(\eta)}T\int_0^{R(\eta)}|u_t|\cdot|u_r|r^2dr+\frac{1}T\int_{R(\eta)}^T|u_t|\cdot|u_r|r^3dr\\
\leq&\frac{R(\eta)}T\|u_t\|_2\|u\|_{\dot
H^1}+\Big(\int_{R(\eta)}^\infty
u_t^2r^2dr\Big)^\frac12\Big(\int_{R(\eta)}^\infty
u_r^2r^2dr\Big)^\frac12\\
\leq&\frac{R(\eta)}T+\eta.
\end{align*}
Letting $T\to+\infty$, we prove the {\bf Claim}, this completes the Theorem.

\vskip 0.2cm

\noindent{\bf \underline{\bf In conclusion}} From the above induction, we only need to prove  Theorem \ref{ewzzx3drw}
and Proposition \ref{gdzzx3drw}.

\vskip 0.45cm

\subsection{Low regularity}

\begin{conjecture}[{\bf Low regularity conjecture}]\label{con:conlrg}
Let $s_c\geq0$, $\mu=-1$ and $(u_0,u_1)\in \dot{H}^{s_c}\times \dot{H}^{s_c-1}$. Then $u$ is global and scatters in the sense that  there exist unique $(u_0^{\pm},u_1^{\pm})\in\dot{H}^{s_c}_x(\R^d)\times \dot{H}^{s_c-1}$ such that
        $$\lim_{t\to\pm\infty}\|u(t)-S(t)(u_0^{\pm},u_1^{\pm})\|_{\dot{H}_x^{s_c}(\R^d)}=0.$$
\end{conjecture}

\vskip 0.2cm

\noindent \underline{\bf Low regularity problem}, where we ask what is the minimal $s$  to ensure that
(NLW) has either  a local solution  or a  global solution for which the scattering hold?

{\rm (i)}   Miao-Zhang \cite{Miao-Zhang-Wave} studied  the  Cauchy
problem for the semilinear wave equation
\begin{equation}\label{miao-zhang-1}\left\{\begin{aligned}
&u_{tt}- \triangle u=-|u|^{\rho-1}u,\;\;\;(t, x)\in\R\times\R^d,\;\;d\ge3,\\
&u(x,0)=\phi(x)\in (\dot H^s\cap L^{\rho+1})(\R^d),\;\; u_t(x, 0)=\psi(x)\in \dot H^{s-1}(\R^d),
\end{aligned}\right.\end{equation}
and proved the global well-posedness of low regularity solution for \eqref{miao-zhang-1} in $\dot H^s\times \dot H^{s-1}$ with $s\ge \alpha(\rho)$,
where
\begin{equation*}\alpha(\rho)\triangleq
\left\{\begin{aligned}
& \frac{2(\rho-1)^2-[d+2-\rho(d-2)][d+1-\rho(d-1)]}
{2(\rho-1)[d+1-\rho(d-3)]}, \\
& k_0(d)\le\rho<\frac{d-1}{d-3},\;  k_0(d)=\frac{(d+1)^2}{(d-1)^2+4}, \;\;d\ge3;
\end{aligned}\right.\end{equation*}

\begin{equation*}\alpha(\rho)\triangleq
\left\{\begin{aligned}
&\frac{4(\rho-1)+[d+2-\rho(d-2)](d\rho-d-4)}
{2(\rho-1)[d+4-\rho(d-2)]},\\
& \frac{d-1}{d-3}\le\rho<\frac{d+2}{d-2},\;\; d=5;
\end{aligned}\right.\qquad\qquad\end{equation*}
\begin{equation*}\alpha(\rho)\triangleq
\left\{\begin{aligned}
&\frac{2\rho(\rho-1)+[d+2-\rho(d-2)](d\rho-d-\rho-2)}
{2(\rho-1)[d+2-\rho(d-3)]},\\
&\frac{d-1}{d-3}\le\rho<\min(\frac{d+2}{d-2},\frac d{d-3}), \;\;\;d\ge6.
\end{aligned}\right.\quad \end{equation*}
\vskip 0.15cm
\vskip 0.15cm
{\rm (ii)} In \cite{Shen1}, Shen proved that for $d=3$, and
$$p\in\big(\tfrac{11}3,5\big),\; ~s_0(p)=\tfrac{2+(5-p)s_c}{7-p}\in(s_c,1),$$
 then for $(u_0,u_1)\in (\dot{H}^{s_c}\cap\dot{H}^s)\times
(\dot{H}^{s_c-1}\cap\dot{H}^{s-1})$, the solution $u$ to $u_{tt}-\Delta u+|u|^{p-1}u=0$ with initial data $(u_0,u_1)$ is global.

\vskip 0.15cm

{\rm (iii)} The most interesting problem is 3d cubic wave equation
\begin{equation}\label{1.1}
\begin{cases}
u_{tt}-\Delta u=-u^3,\quad (t,x)\in\R\times\R^3,\\
(u,\pa_tu)(0)=(u_0,u_1),\quad\; x\in\R^3.
\end{cases}
\end{equation}

\begin{center} The well-known results on the cubic wave equation in $\mathbb R^3$ \\
\begin{tabular}{|c|c|c|}
  \hline
  % after \\: \hline or \cline{col1-col2} \cline{col3-col4} ...
 $s>3/4$  & $s=3/4$ & $s>\frac7{10} \;{\rm or} \;\frac{13}{18}$ \\\hline
 KPV\cite{kenponcevega}, Gallagher-Planchon
\cite{gallagplanch} & Bahouri-Chemin \cite{bahchemin} &Roy \cite{triroy}, \cite{triroy1}  \\\hline
\end{tabular}
\end{center}
\vskip0.15cm

\begin{theorem}[Dodson \cite{Dod1} ]\label{thm:t1.5}
Suppose there exists a positive constant $\varepsilon > 0$ such that
\begin{equation}\label{1.10}
\| u_{0} \|_{\dot{H}^{1/2 + \varepsilon}(\mathbf{R}^{3})} + \| |x|^{2 \varepsilon} u_{0} \|_{\dot{H}^{1/2 + \varepsilon}(\mathbf{R}^{3})} \leq A < \infty,
\end{equation}
and
\begin{equation}\label{1.11}
\| u_{1} \|_{\dot{H}^{-1/2 + \varepsilon}(\mathbf{R}^{3})} + \| |x|^{2 \varepsilon} u_{1}
\|_{\dot{H}^{-1/2 + \varepsilon}(\mathbf{R}^{3})} \leq A < \infty.
\end{equation}
Then $(\ref{1.1})$ has a global solution and there exists some $C(A, \varepsilon) < \infty$ such that

\begin{equation}\label{1.12}
\int_{\mathbf{R}} \int (u(t,x))^{4} dx dt \leq C(A, \varepsilon),
\end{equation}
which proves that $u$ scatters both forward and backward in time.
\end{theorem}

\begin{theorem}[Dodson \cite{Dod2} ]\label{t1.6}
The Cauchy problem $(\ref{1.1})$ is globally well-posed and scattering
 for radial initial data  $(u_{0}, u_1) \in \dot B_{1,1}^{2}(\mathbf{R}^{3})\times \in \dot B_{1,1}^{1}(\mathbf{R}^{3}))$.
 Moreover,
\begin{equation}\label{1.15}
\| u \|_{L_{t,x}^{4}(\mathbf{R} \times \mathbf{R}^{3})} \leq C(\| u_{0} \|_{\dot B_{1,1}^{2}}, \| u_{1} \|_{\dot B_{1,1}^{1}}).
\end{equation}
\end{theorem}

This is the first result in which large data scattering was proved for initial data in a scale - invariant space for which the norm was not controlled by a conserved quantity.

\subsection{Description of local singularity for type I blow up bubble}

 From finite speed of propagation, using the constant in space solution
$$
u(t,x)=\frac{C}{(T-t)^{\frac{2}{p-1}}}
$$
one can construct solutions blowing up like the ODE $u_{tt}=u^p$. They are called type I blow up and correspond to a {\it complete} blow up
$$\lim_{t\to T}\| u(t),\partial_tu(t)\|_{\dot{H}^1\times L^2}\to +\infty.$$
In the subconformal case ($s_c\leq \frac{1}{2}$) the recent works by Merle and Zaag \cite{MZ03,MZ05,MZ2} give in particular a complete description of the local singularity, being always a type I blow up bubble, and we refer to this monumental series of works for complete references on the history of the problem. Recently also, general upper bounds on the blow up rates have been obtained \cite{HZ}, \cite{KV} in the superconformal, energy subcritical case ($\frac{1}{2}<s_c<1$).

%%%%%%%%%%%%%%%%%%%%%%%%%%%%%%%%%%%%%%%%%%%%%%%%%%%%%%%%%%%%%%%%%%%%%%%%%%%%%%%%%%%%%%%%%%%%%%%%%%%%%%%%%%%%%%%%%%%%%%%%%%%%%%%%%%%%%

%%%%%%%%%%%%%%%%%%%%%%%%%%%%%%%%%%%%%%%%%%%%%%%%%%%%%%%%%%%%%%%%%%%%%%%%%%%%%%%%%%%%%%%%%%%%%%%%%%%%%%%%%%%%%%%%%%%%%%%%%%%%%%%%%%%%%

\section{Nonlinear Klein-Gordon equation}

Let {$\square=\partial_{tt}-\Delta$}, we
consider
\begin{align*}
\square u+u+f(u)= 0, \;\;\text{{(NLKG)}}.
\end{align*}

{Assumptions: }
\begin{enumerate}
\item {Structure condition}: $F$ satisfies
$ \frac{\partial F(z)}{\partial \overline{z}}=f(z)$, $F(0)=0. $

\item {Smooth} and {growth conditions}: For $\frac{4}{d}<p_1<p_2<\frac{4}{d-2}$
\begin{equation*}
|f(u)-f(v)|\le C(|u|^{p_1}+|v|^{p_1}+|u|^{p_2}+|v|^{p_2})|u-v|,\;\;
f(0)=0.
\end{equation*}

\item {Repulsive condition}. If $
V(u)=\frac{F(u)}{|u|^2}:\; \C\rightarrow\R $ then
\begin{align}
\partial_{|z|}V(z)=2{\rm Re}\Big[\partial_zV(z)\frac{z}{|z|}\Big]\ge0,
\label{4.1}\end{align} which means that $ V(u)$ is
{non-decreasing} function on $|u|$.
\\
\end{enumerate}

\underline{\bf The classical Morawetz estimate}
\begin{align}
\int\int_{\R^{d+1}}\frac{G(u)}{|x|}dxdt\le CE(u),
\label{4.2}\end{align}
where
\begin{align}
G(z)={\rm Re}\big(\partial_zV(z)|z|^2z\big)={\rm Re}\big(\overline{z}f(z)\big)-F(z).
\label{4.3}\end{align}

\begin{remark}\small
{\rm (i)} As we know, the  scattering theory for defocuing (NLKG) is based on the classical
Morawetz estimate \eqref{4.2}, from which we can
deduce the {\bf decay estimate} of the solution.
This ensures the potential energy has no concentration.

{\rm (ii)} The introduction of $V(z)$ makes us to compare (NLKG) with
 the following linear equation
 $$\square u+u+V(x)u=0. $$
{\bf Berestyeki-Lions} have shown that the above equations
have the {\bf traveling wave solution} under the
condition $V(z)<V(0)$, which shows that the wave
operator is not surjection. The current study is
constrained under
\begin{align}
\partial_{|z|}V(z)\ge C \min(|z|^{-1},|z|^{\gamma}),\quad \gamma>0
\label{4.3}\end{align} i.e.,  $V(u)$ is not {\bf flat}
at $u=0$, and diverse at $u=\infty$. As for the general case
$$V(z)\ge V(0),\quad \forall z\in\C$$
it is still {\bf an open problem} whether one can establish
the scattering theory of {\rm (NLKG)}.
\end{remark}

\vskip0.12cm
\underline{\bf Generalized  Morawetz inequality}
\vskip0.2cm

Under the {\bf repulsive condition},
Nakanishi derived the {\bf modified
Morawetz estimate} independent of the form of nonlinearity by
time-space  multiplier argument, which is
suitable  for different dimensions(including $d=1,2$)!
The arguments are based on

\begin{enumerate}
\item {\bf Generalized multiplier $(h\cdot \mathcal{D}u+qu)$}.
\vskip0.2cm

\item  The Hardy inequality  and {\bf generalized Gagliardo-Nirenberg inequality}
\begin{align*}
\int_{\R^d}\kappa^2|u|^pdx\le\|u\|_q^{p-2}\int_{\R^d}(\kappa^2|\nabla
u+i\lambda u|^2+|u\nabla\kappa|^2)dx,
\end{align*}
where $p>2,  q=\frac{d(p-2)}{2}$.
\end{enumerate}

 Assume that $f(u)$ satisfies the {\bf structure, growth}
and {\bf repulsive conditions}, and
$$2+\frac{4}{d}<p<\infty,\;\; d\le 2;\quad \text{or}\quad 2+\frac{4}{d}<p <2+\frac{4}{d-2},\;\; d\ge3, $$
then the energy solution $u$ of
(NLKG)  satisfies that
 \begin{align}
\iint_{|x|\le t}\frac{|u|^p}{|t|}dxdt\le CE(u)^{\frac{p}{2}},   \label{4.5}
\end{align}
where $C=C(d,p)$ is {\bf independent of
the growth exponents $p_1$ or $p_2$} in nonlinearity.

 The {modified Morawetz estimate} breaks through the {\bf restriction on
dimension}. For  the critical case $p=2^*$ with  $d\ge3$, we have to use the Gagliardo-Nirenberg inequality
 in hyperbolic space to establish the following  Morawetz  estimate:
\begin{align}
\iint_{\R^{d+1}}\frac{|u|^{2^*}}{t+|x|}dxdt\le C(E),\quad d\ge 3
\label{4.11}\end{align}

\vskip0.2cm

\noindent\underline {\bf Scattering theory and its induction}.\;  In order to
simply the statement, we introduce the notation
\begin{equation}
\overrightarrow{u}= (u,(I-\Delta)^{-\frac{1}{2}}\dot{u}) \label{4.12}\end{equation}
Then $\overrightarrow{u}(t)\in\dot H^1(\R^d)$ has the following
conservation
\begin{align*}
E(u,t)\triangleq\int_{\R^d}\big(
|\nabla\overrightarrow{u}(t)|^2+|\overrightarrow{u}(t)|^2+F(u(t))
\big)dx =E(u,0).\label{4.13}\end{align*}
Obviously, the above energy is consistent with the original form
\begin{align*}
E(u,t)=&
\int_{\R^d}\big( |\nabla u|^2+|\nabla\sqrt{1-\Delta}^{-1}\dot{u}|^2+|u|^2\\
&\qquad \qquad \;\;\;+|\sqrt{1-\Delta}^{-1}\dot{u}|^2 + F(u)\big)dx\\
=&\int_{\R^d}\left(|\nabla u(t)|^2+|\dot{u}(t)|^2+|u(t)|^2+F(u(t))
\right)dx.
\end{align*}
\vskip0.2cm

\begin{theorem}[Brenner-Ginibre-Velo-Nakanishi]\label{NLKG-scattering}
Let $d\geq 1,\; f:\C\rightarrow\C$ satisfy the
{\bf structure, growth} and {\bf repulsive
conditions}. Then {\bf (NLKG)} has the unique global energy solution ${u}\in C(\R;H^1(\R^d))\cap C^1(\R;L^2(\R^d))$  with
\begin{equation}
\left\{ \aligned
&\|u\|_{{L^\rho(\R;B^{\frac{1}{2}}_{\rho,2}(\R^d))}}\le  C(E_0),\quad \qquad \qquad\qquad d\le 2, \\
&\|u\|_{{L^\rho(\R;B^{\frac{1}{2}}_{\rho,2}(\R^d))\cap
 L^\zeta(\R;B^{\frac{1}{2}}_{\zeta,2}(\R^d))}}\le C(E_0),\quad \; d\geq 3,\endaligned
\right. \label{4.14}
\end{equation}
where
$$\rho=2+\frac{4}{d}, \;\;\;\zeta=2+\frac{4}{d-1}$$
and $E_0 = E(u,0)$ is the energy of the solution $u$ at $t=0$.

Moreover, by \eqref{4.14}, the solution $u$
scatters, i.e.,   there exists a unique solution
$v_\pm$ of the following free Klein-Gordon equation
\begin{align*}
\square v +v= 0
\end{align*}
such that
\begin{align*}
\|\overrightarrow{u}(t)-\overrightarrow{v}_\pm(t)\|_{H^1(\R^d)\times L^2(\R^d)}\rightarrow
0,\quad\text{as}\;\; t\rightarrow \pm\infty
\end{align*}
and the corresponding wave operator
\begin{align*}
\Omega_{\pm}:\; \overrightarrow{v}^{\pm}(0)\triangleq v^\pm(x)\rightarrow\overrightarrow{u}(0),\quad
H^1(\R^d)\times L^2(\R^d)\rightarrow H^1(\R^d)\times L^2(\R^d)
\end{align*}
is a {homeomorphism} from $H^1(\R^d)\times L^2(\R^d)$ to itself.
\end{theorem}

For the defocusing  energy-critical NLKG, refer to Miao's  Book \cite{MC}.

\subsection{Focusing case}

First, we consider the
{\bf Energy-subcritical} case:

\begin{equation}\label{equ:nlkgsub}
\begin{cases}
u_{tt}-\Delta u+u-|u|^{p-1}u=0,\\
(u,\pa_tu)(0)=(u_0,u_1)\in H^1(\R^d)\times L^2(\R^d).
\end{cases}
\end{equation}

\noindent\underline{{\bf Motivation:}}
\vskip0.2cm

(1) It is obvious that
$u(x,t)=Q(x)$ is the solution of \eqref{equ:nlkgsub} with $(u,\pa_tu)(0)=(Q(x),0)$,
where  $Q$ is the ground state
of elliptic equation
\begin{align*}
-\Delta Q+Q=|Q|^{p-1}Q,\quad p<1+\tfrac{4}{d-2}.
\end{align*}
From the viewpoint of ``{\bf phase transition}",  we
conjecture that $u(t,x)=Q$ belongs to
 the boundary of scattering domain.

\vskip0.12cm

(2) Payne and Sattinger  proved  {\bf  the global existence/finite time blowup dichotomy} for
solutions of energies of (NLKG) on bounded smooth domain in $\mathbb R^d$ below the ground-state energy,  by the sign of the functional
\begin{equation}
K_{1,0}(u)\triangleq\int\big(|\nabla u|^2+|u|^2-|u|^{p+1}\big)dx.
\end{equation}
It is easy to observe that their argument applies to the whole space
$\R^d$ as soon as one has the local well-posedness for (NLKG) in the energy
space.

\vskip0.12cm

(3) A simple computation shows
\begin{align*}
\frac{d}{dt}\int_{\mathbb{R}^d}\big(x\cdot\nabla+\nabla\cdot x)u
u_tdx=-\Big(\int_{\R^d}\big(2|\nabla
u|^2-\tfrac{d(p-1)}{p+1}|u|^{p+1}\big)dx\Big)=-K_{d,-2}(u),
\end{align*}
which is useful for the {\bf scattering theory}.

\vskip0.2cm

\noindent\underline{{\bf Variational setting:}}
\vskip0.2cm

To state the main results, we need to introduce some notation
and assumptions for the variational setting.

\vskip0.12cm
Let the static energy be
\begin{equation}\label{static}
J(\phi)=\frac12\int_{\mathbb{R}^d}\big(|\nabla
\phi|^2+|\phi|^2\big)dx-\frac{1}{p+1}\int_{\mathbb{R}^d}|\phi|^{p+1}dx.
\end{equation}
The scaling derivative of the {\bf static energy} is denoted by
\begin{align*}
& K_{\alpha,\beta}(\phi)=\frac{d}{d\lambda}\big |_{\lambda=0}J\left(e^{\lambda\alpha}\phi(e^{-\beta\lambda}x)\right)\\
=&\int_{\R^d}\left(\frac{2\alpha+(d-2)\beta}{2}|\nabla\phi|^2+\frac{2\alpha+d\beta
}{2} |\phi|^2-\frac{(p+1)\alpha+pd}{p+1} |\phi|^{p+1}\right)\; dx\\
=&\int_{\mathbb{R}^d}\big(\alpha\varphi-\beta
x\cdot\nabla\varphi\big)\big(-\Delta
\varphi+\varphi-|\varphi|^{p-1}\varphi\big)dx.
\end{align*}

\vskip0.2cm

Let
\begin{align*}
\Omega\triangleq\left\{ (\alpha,\beta)\; \Big|\; \alpha\ge0, \;
(\alpha,\beta)\neq(0,0), { 2\alpha+(d-2)\beta\ge0, \atop
2\alpha+d\beta
\geq0, }  \right\}
\end{align*}
we consider the {\bf constrained minimization} problem
$$m_{\alpha,\beta}=\inf\big\{J(\varphi)~|~\varphi\in
H^1(\R^d)\backslash \{0\},~K_{\alpha,\beta}=0\big\}.$$

\vskip0.12cm
\begin{prop}[Parameter independent]\label{prop-Variational setting-1}
If $(\alpha,\beta)\in\Omega$, then $m_{\alpha,\beta}$ is
independent of $(\alpha,\beta)$ and it is attained
$$m_{\alpha,\beta}=J(Q),$$
where $Q$ is the {\bf ground state} of the elliptic equation
$$-\Delta Q+Q=|Q|^{p-1}Q.$$
\end{prop}

\vskip0.12cm

The solutions start from the following subsets of the energy space:
\begin{align*}
 K_{\alpha, \beta}^{\pm}=\left\{ (u_0,u_1)\;
\Big|\;  E(u_0,u_1)<E(Q,0)=J(Q),\;\; K_{\alpha, \beta}(u){\geq \atop
< }0 \right\}.
\end{align*}

\begin{prop}[Parameter independence of the splitting]\label{prop-Variational setting-2}
For $(\alpha,\beta)\in\Omega,$ $K_{\alpha, \beta}^{\pm}$ is
{\bf independent} of $(\alpha,\beta)$.
\end{prop}

\begin{lemma}[Uniform bounds on $K$] \label{Lem-Variational setting-3} Assume that  $(\alpha,\beta)\in\Omega$ and
$(d,\alpha)\neq(2,0)$, then $\exists
\delta=\delta(\alpha,\beta,d,p)>0$, s.t. $\forall \varphi\in
H^1,~J(\varphi)<m,$ we have
$$K_{\alpha,\beta}\geq\min\{\bar{\mu}(m-J(\varphi)),\delta
K_{\alpha,\beta}^Q(\varphi)\},~\text{or}~
K_{\alpha,\beta}\leq-\bar{\mu}(m-J(\varphi)),$$
where ${\bar{\mu}=\max\{2\alpha+(d-2)\beta,2\alpha+d\beta\}}$
and  the quadratic part
$$K_{\alpha,\beta}^Q(\varphi)=\frac{2\alpha+(d-2)\beta}{2}\|\nabla\varphi\|_2^2+\frac{2\alpha+d\beta}{2}\|\varphi\|_2^2.$$
\end{lemma}

\begin{theorem}[Focusing energy-subcritical, Ibrahim, Masmoudi and Nakanishi \cite{IMN}]
\label{theorem-Variational setting-4}
 For $(\al,\beta)\in\Omega.$
\begin{enumerate}
\item If $(u(0),\dot{u}(0))\in {K_{\al,\beta}^-,}$ then the solution $u$
 extends {\bf neither} for $t \to \infty$ {\bf nor} for $t \to-\infty$ as the unique strong
solution in $H^1\times L^2$.
\item If $(u(0),\dot{u}(0))\in {K_{\al,\beta}^+},$ then the solution
$u$ {\bf scatters} both in $t\to\pm\infty$ in the energy space. In other
words, $u$ is a {\bf global} solution and there are ${v_\pm}$ satisfying
\begin{align}
v_{tt}-\Delta v+v=&0,\\
\|(u,\dot{u})-(v_\pm,\dot{v}_\pm)\|_{H^1\times L^2}\to&
0,~\text{as}~t\to\pm\infty.
\end{align}
\end{enumerate}
\end{theorem}

\vskip0.2cm

\noindent\underline{\bf Energy-critical NLKG}
\vskip0.2cm

Let $ d\geq 3$, we consider
\begin{equation}\label{cri-NLKG}
\left\{ \aligned
&u_{tt}-\Delta u+u=|u|^{\frac{4}{d-2}}u, \\
&u(0)=u_0(x)\in H^1,\\
& u_t(0)=u_1(x)\in L^2,
\endaligned
\right.
\end{equation}
where $ u:\R^{1+d}\mapsto\R. $

\vskip0.2cm
It is obvious that $u(t,x) = e^{it}Q$ is the
solution of energy-critical NLKG, where $Q$ is the
{\bf ground state} of elliptic equation
\begin{align*}
-\Delta Q=|Q|^{\frac4{d-2}}Q.
\end{align*}
From the viewpoint of phase transition,  we conjecture that $u(t,x)
= e^{it}Q$  corresponds to the boundary of {\bf scattering domain}.

Let the {\bf static energy} be
\begin{equation}\label{static}
J(\phi)=\frac12\int_{\mathbb{R}^d}\big(|\nabla
\phi|^2+|\phi|^2\big)dx-\frac{1}{p+1}\int_{\mathbb{R}^d}|\phi|^{p+1}dx.
\end{equation}
In the critical case, we also need
the  {\bf modified static energy}
\begin{equation}\label{modestatic}
{J^{(0)}(\phi)}:=\frac12\int_{\mathbb{R}^d}|\nabla
\phi|^2dx-\frac{1}{p+1}\int_{\mathbb{R}^d}|\phi|^{p+1}dx.
\end{equation}
The {\bf scaling derivative} of the static energy is denoted by
\begin{align*}
& {K_{\alpha,\beta}(\phi)}=\frac{d}{d\lambda}\big |_{\lambda=0}J\left(e^{\lambda\alpha}\phi(e^{-\beta\lambda}x)\right)\\
=&\int_{\R^d}\left(\frac{2\alpha+(d-2)\beta}{2}|\nabla\phi|^2+\frac{2\alpha+d\beta
}{2} |\phi|^2-\frac{(p+1)\alpha+pd}{p+1} |\phi|^{p+1}\right)\; dx\\
=&\int_{\mathbb{R}^d}\big(\alpha\varphi-\beta
x\cdot\nabla\varphi\big)\big(-\Delta
\varphi+\varphi-|\varphi|^{p-1}\varphi\big)dx.
\end{align*}

\vskip0.15cm

\noindent\underline{\bf Variational Characterization of
Ground State $Q$}

\vskip0.2cm
The ground state $Q$ has the following
characterization
\begin{align*}
m_{\al,\beta}\triangleq& \inf\left\{ J(\phi) \; \big| \; \phi\in H^1
\setminus\{0\},\; K_{\alpha, \beta}(\phi)=0\;
\right\}\\
=&\inf\left\{ J^{(0)}(\phi) \; \big| \; \phi\in \dot{H}^1
\setminus\{0\},\; -\Delta \phi=|\phi|^{\frac{4}{d-2}}\phi\;
\right\}\\
=&J^{(0)}(Q).
\end{align*}
where {\bf $(\alpha, \beta)\in \Omega=\left\{ (\alpha,\beta)\; \Big|\; \alpha\ge0, \;
(\alpha,\beta)\neq(0,0), { 2\alpha+(d-2)\beta\ge0, \atop
2\alpha+d\beta
\geq0, }  \right\}$}

\vskip0.15cm
The solutions start from the following subsets of the energy space:
\begin{align*}
 K_{\alpha, \beta}^{\pm}=\left\{ (u_0,u_1)\;
\Big|\;  E(u_0,u_1)<J^{(0)}(Q),\;\; K_{\alpha, \beta}(u){\geq \atop
< }0 \right\}.
\end{align*}

\begin{prop}[Parameter independence of the splitting]\label{prop-Variational setting-5}
For $(\al,\beta)\in\Lambda,$ $K_{\alpha, \beta}^{\pm}$ is
{\bf independent} of $(\al,\beta)$.
\end{prop}

\begin{lemma}[Uniform bounds on $K$] \label{Lem-Variational setting-6}Assume that
$(\al,\beta)\in\Omega,$ then $\exists
\delta=\delta(\alpha,\beta,d,p)>0$, s.t. $\forall \varphi\in
H^1,~J(\varphi)<m$, we have
$${ K_{\alpha,\beta}\geq\min\{\bar{\mu}(m-J(\varphi)),\delta
K_{\alpha,\beta}^Q(\varphi)\},~\text{or}~
K_{\alpha,\beta}\leq-\bar{\mu}(m-J(\varphi)),}$$
where ${\bar{\mu}=\max\{2\alpha+(d-2)\beta,2\alpha+d\beta\}}$ and
 the quadratic part
 $$K_{\alpha,\beta}^Q(\varphi)=\frac{2\alpha+(d-2)\beta}{2}\|\nabla\varphi\|_2^2+\frac{2\alpha+d\beta}{2}\|\varphi\|_2^2.$$
\end{lemma}

\begin{theorem}[Focusing energy-critical NLKG, Ibrahim, Masmoudi and Nakanishi \cite{IMN}]
\label{Theorem-Variational setting-7}
 For $(\al,\beta)\in\Omega.$
\begin{enumerate}
\item If $(u(0),\dot{u}(0))\in K_{\al,\beta}^-$,
 then the solution $u$ extends {\bf neither} for $t \to \infty$ nor for $t \to-\infty$ as the unique strong
solution in $H^1\times L^2$.
\item If $(u(0),\dot{u}(0))\in K_{\al,\beta}^+$, then the solution
$u$ {\bf scatters} both in $t\to\pm\infty$ in the energy space. In other
words, $u$ is a {\bf global} solution and there are $v_\pm$ satisfying
\begin{align}
v_{tt}-\Delta v+v=&0,\\
\|(u,\dot{u})-(v_\pm,\dot{v}_\pm)\|_{H^1\times L^2}\to&
0,~\text{as}~t\to\pm\infty.
\end{align}
\end{enumerate}
\end{theorem}

\begin{remark}
One can refer this part to the book by Nakanishi and Schlag \cite{KS10} or see Miao's lecture on nonlinear Klein-Gordon equation.
\end{remark}

%%%%%%%%%%%%%%%%%%%%%%%%%%%%%%%%%%%%%%%%%%%%%%%%%%%%%%%%%%%%%%%%%%%%%%%%%%%%%%%%%%%%%%%%%%%%%%%%%%%%%%%%%%%%%%%%%%%%%%%%%%%%%%%%%%%%%

%%%%%%%%%%%%%%%%%%%%%%%%%%%%%%%%%%%%%%%%%%%%%%%%%%%%%%%%%%%%%%%%%%%%%%%%%%%%%%%%%%%%%%%%%%%%%%%%%%%%%%%%%%%%%%%%%%%%%%%%%%%%%%%%%%%%%

\section{Proof of defocusing energy-critical NLW}\label{sec:defenecri}

\begin{theorem}\label{T:gopher} Let  $d\geq3$. Given $(u_0,u_1)\in\dot H^1(\R^d)\times L^2(\R^d)$.
Then, there is a unique global strong solution $u$ to
\begin{equation}\label{equ:nlwdef}
{\rm (NLW)}\;\;\;\;\begin{cases}
\pa_{tt}u-\Delta u+|u|^\frac4{d-2}u=0,\\
(u,\pa_tu)(0)=(u_0,u_1).
\end{cases}
\end{equation}
Moreover, the solution $u$  obeys the estimate
\begin{equation}\label{Gophergoal123}
\int_{\R}\int_{\R^d}|u(t,x)|^{\frac{2(d+1)}{d-2}}\,dx\,dt\leq C( \|(u_0,u_1)\|_{\dot H^1_x\times L^2_x} ).
\end{equation}
And so the solution scatters.

\end{theorem}

\vskip 0.1in

{\bf Notation}: Let $z=(x,t)$ denote the space-time
coordinate
\begin{enumerate}
\item Cone $K(z_0)$ and cut-off of Cone $K^T_S$:
\begin{align*}
{K(z_0)} = &  \{z=(x,t) \;\big|\; |x-x_0|\leq t_0-t, 0\leq t\leq t_0\}\\
{K^T_S} = & \{z=(x,t)\;\big|\; (x,t)\in K(z_0), \;
S\leq t\leq T\}
\end{align*}

\item Lateral Surface $M^T_S(z_0)$ and $\partial K^T_S$:
\begin{align*}{M^T_S(z_0)} = &  \{z=(x,t) \;\big|\; |x-x_0| = t_0-t, S\leq t\leq T\}\\
{\partial K^T_S} = &  D_S(z_0) \cup D_T(z_0) \cup
M^T_S
\end{align*}

\item Section: ${D_t(z_0)} =  \{z=(x,t) \;\big|\; (x,t)\in K(z_0),\; t\; \text{fixed}\}$

\vskip 0.1in
\item ST Slab: ${Q^T_S} = \{z=(x,t) \;\big|\; (x,t)\in Q, S\leq t\leq
T\}$ for any $Q \subset \R^d\times\R$ and $S\leq T$
\end{enumerate}

\vskip0.15cm

For any given function $u$ on $K(z_0)$, we introduce ($2^\ast=\frac{2d}{d-2}$)

\vskip0.15cm
\begin{enumerate}
\item  {\bf Energy-Momentum density}
$$e(u)=  \left( \frac12 |u_t|^2 + \frac12 |\nabla u|^2 + \frac{1}{2^*} |u|^{2^*}, -(u_t \nabla u)\right)$$

\item {\bf Local energy}
$${E\left(u, D_T(z_0)\right)} = \int_{ D_T(z_0)} \left(\frac12 |u_t|^2 + \frac12 |\nabla u|^2 + \frac{1}{2^*} |u|^{2^*}\right) \; dx$$

\item {\bf Energy flux density}
$$dz_0(u) =\frac12 \left|\frac{y}{|y|}u_t -\nabla u\right|^2 +   \frac{1}{2^*} |u|^{2^*}, \;\; y=x-x_0$$

\item {\bf Emission energy from side $M^T_S$}
$${\rm Flux}(u, M^T_S(z_0))=   \int_{M^T_S(z_0)} dz_0(u)\; d\sigma$$

\end{enumerate}

\vskip0.1cm
\noindent\underline{\bf Energy conservative and its local form-Basic Fact I}:

\vskip0.2cm

\noindent Multiplying (NLW) with
$\partial_t u$ ({\bf generation element
of time translation invariance}),
$$\text{div}_{t,x}e(u) = \left(\frac12 |u_t|^2 + \frac12 |\nabla u|^2 + \frac{1}{2^*} |u|^{2^*}\right)_t - \text{div}(\nabla u u_t)=0,$$
then integrating on $K(z_0)$ and
$\R^d$, we can obtain
$$E(u, D_T(z_0)) + \frac{1}{\sqrt{2}} \text{Flux}(u, M^T_S(z_0)) = E(u, D_S(z_0))),$$
$$E(u,u_t) = \int_{\R^d} \left(\frac12 |u_t|^2 + \frac12 |\nabla u|^2 + \frac{1}{2^*} |u|^{2^*}\right)\; dx = E(u_0, u_1).$$

\vskip0.2cm
\noindent\underline{\bf The local form of conformal
identity- Basic Fact II}:
\vskip0.2cm

\noindent Multiplying  \eqref{equ:nlwdef} with
$(t\partial_t + x\cdot\nabla + \frac{d-1}{2})u$,
which is the generation element of the {\bf scaling transformation} $u_{\varepsilon}(t,x)\rightarrow
\varepsilon^{\frac{d-1}{2}}u(\varepsilon t, \varepsilon x)$, we can obtain
\begin{align}
 & \text{div}_{t,x} \left(tQ_0 + \frac{d-1}{2} u_t u, -tP_0\right) +
R_0 =0 \label{add3-1}\\
&\text{or}\nonumber\\
 & \partial_t (tQ_0 +\frac{d-1}{2}u_t u) -
\text{div}(tP_0) + R_0 =0\label{add3-2}
\end{align}
where
 \begin{align*}
Q_0 = & \frac12 |u_t|^2 + \frac12 |\nabla u|^2 + \frac{1}{2^*} |u|^{2^*} + u_t (\frac{x}{t} \cdot \nabla u),\quad R_0 =\frac{|u|^{2^*}}{d}\\
P_0 = & \frac{x}{t}\left(\frac{|u_t|^2 - |\nabla u|^2}{2} -
\frac{|u|^{2^*}}{2^*}\right) + \nabla u \left(u_t + \frac{x}{t}\cdot
\nabla u + \frac{d-1}{2} \frac{u}{t}\right).
\end{align*}

\vskip0.15cm
\noindent\underline{\bf Non concentration of potential energy-Basic Fact III}:
\vskip0.15cm

 If {$u(t,x)$} is
the regularized solution on {$K(z_0)\backslash
\{z_0\}$} ($u\in C^2((0, t_0)\times \R^d) $), then we have
\begin{align*}
\lim_{s\rightarrow t_0} \int_{D_S(z_0)} |u|^{2^*}\; dx =0.
\end{align*}

\noindent\underline{{\bf The  sketch of proof on existence of global smooth solution}}:
\vskip 0.2cm

{\bf Step 1}: By the {\bf finite propagation
speed},  we can  induce the above problem  to the problem
with {\bf compact supported data} $(u_0(x), u_1(x))$.
\vskip 0.15cm

{\bf Step 2}: {\bf Nonlinear estimate}
\begin{align*}
\big\||u|^{2^*-2}u\big\|_{L^{q'}\left([s, \tau];
B^{1/2}_{q'}(D_t(z_0))\right)} \leq \sup_{t\in [s,\tau]}
\big\|u\big\|^{\frac{2^*-2}{d-1}}_{L^{2^*}\left(D_t(z_0)\right)}
\big\|u\big\|^{1+\frac{d-2}{d-1}(2^*-2)}_{L^q\left([s,\tau];
B^{1/2}_{q}(D_t(z_0))\right)}.
\end{align*}
This estimate together with the Strichartz estimates implies that
\begin{align*}
\big\|u\big\|_{L^q\left([s,\tau]; \dot B^{1/2}_{q}(D_t(z_0))\right)}
\leq C(E_0) + \varepsilon \big\|u\big\|^{\gamma}_{L^q\left([s,\tau];
\dot B^{1/2}_{q}(D_t(z_0))\right)},
\end{align*}
where
$$q=\frac{2(d+1)}{d-1}, \quad \gamma>1.$$

\vskip0.15cm
 {\bf Step 3: Continuous argument} (Compactness $\Longrightarrow$ GWP)
\vskip 0.2cm

 By the {compactness} and the {finite propagation speed}, it suffices to
 show that $u(t,x)$ is smooth on the
 neighborhood region of $(t_0, x_0)$.

\vskip0.2cm

\underline{\bf The proof of scattering result}
\vskip0.2cm

{\bf Step 1: Induction}. Given $(u,u_t)\in C(\R;\dot H^1)\times C(\R;L^2(\R^d))$, it suffices to show that
 \begin{align*}
(u,u_t)\in L^q(\R;\dot B^{\frac{1}{2}}_q(\R^d))\times L^q(\R;\dot
B^{-\frac{1}{2}}_q(\R^d)).
\end{align*}
By  LWP thery, we know that
\begin{align*}
(u,u_t)\in L^q_{loc}(\R;\dot B^{\frac{1}{2}}_q(\R^d))\times
L^q_{loc}(\R;\dot B^{-\frac{1}{2}}_q(\R^d)),
\end{align*}
it suffices to show for sufficiently large $T_0>0$
\begin{align}
(u,u_t)\in L^q( {[T_0,\infty)};\dot
B^{\frac{1}{2}}_q(\R^d))\times L^q(
{[T_0,\infty)};\dot B^{-\frac{1}{2}}_q(\R^d)).
\label{star1}\end{align}

\vskip0.2cm

{\bf Step 2}: We can further reduce
\eqref{star1} in essence to prove
\begin{align}
\lim_{t\rightarrow +\infty} g(t)\triangleq \frac{1}{2^*}\int_{\R^d}
|u|^{2^*}dx=0. \label{star2}\end{align}

In fact, by {Strichartz estimate} and
{nonlinear estimates}, we have
\begin{align*}
\|u\|_{L^q([T_0,T];\dot B^{\frac{1}{2}}_q) }
& \lesssim E(u_0,u_1)^\frac{1}{2}+  {\sup_{T_0\le t\le T}\|u\|_{2^*}^{(1-\alpha)(2^*-2)}}\|u\|^{1+\alpha(2^*-2)}_{L^q([T_0,T];\dot B^{\frac{1}{2}}_q)} \\
& \lesssim
 E(u_0,u_1)^\frac{1}{2}+ {\varepsilon_0^{(1-\alpha)(2^*-2)}}\|u\|^{1+\alpha(2^*-2)}_{L^q([T_0,T];\dot
B^{\frac{1}{2}}_q)},
\end{align*}
where $\alpha=\frac{d-2}{d-1}$. By the  {continuous
argument}, we can show that
\begin{align*}
\|u\|_{L^q([T_0,T];\dot B^{\frac{1}{2}}_q(\R^d))}\le
2CE(u_0,u_1)^\frac{1}{2} \Rightarrow \|u\|_{L^q([T_0,\infty);\dot
B^{\frac{1}{2}}_q(\R^d))} < \infty
\end{align*}

\vskip0.2cm

 {\bf Step 3}: Proof of  \eqref{star2} ( {FPS} and  {conformal scaling
transform.}). \vskip 0.15cm

 First of all, for sufficiently large $R>R_0$, we have
\begin{align*}
\int_{|x|\ge R}e(u_0,u_1)dx \le \frac{\varepsilon_0}{8}, \quad
e(u,u_t)=\frac{1}{2}|\nabla u|^2+\frac{1}{2}|
u_t|^2+\frac{1}{2^*}|u|^{2^*}.
\end{align*}
Now we integrate the following identity
\begin{align*}
\partial_t(e(u,u_t))-\nabla\cdot(\nabla uu_t)=0,
\end{align*}
over the outside of the forward light cone, and obtain
\begin{align*}
\int_{{|x|\ge
R+t}}e(u,u_t)dx+\frac{1}{\sqrt{2}}Flux(u,M_0^t) =\int_{
{|x|\ge R}}e(u_0,u_1)dx \le \frac{\varepsilon_0}{8},
\end{align*}
where
\begin{align*}
Flux(u,
M_0^t)=\int_{M_0^t}\left(\frac{1}{2}\big|\frac{x}{|x|}u_t+\nabla
u\big|^2+\frac{1}{2^*}|u|^{2^*}\right)d\sigma\ge0,
\end{align*}
 which implies that
\begin{align*}
\frac{1}{2^*}\int_{ { |x|\ge
R+t}}|u(t,x)|^{2^*}dx\le\frac{\varepsilon_0}{4}.
\end{align*}
Now it suffices to show that for any $t\geq T_0$
\begin{align*}
\frac{1}{2^*}\int_{ { |x|\le
R+t}}|u(t,x)|^{2^*}dx\le\frac{\varepsilon_0}{4}.
\end{align*}
By the  {\bf time translation invariance} ($t\rightarrow
t+R$), we only need to show that for any $t\ge T_0$
\begin{align}
\frac{1}{2^*}\int_{ { |x|\le
t}}|u(t,x)|^{2^*}dx\le\frac{\varepsilon_0}{2}.
\label{star3}\end{align}
 Integrating  the local form of conformal
identity \eqref{add3-2} over the cone
$M^T_{\varepsilon T}$, we can obtain
\eqref{star3}.

%%%%%%%%%%%%%%%%%%%%%%%%%%%%%%%%%%%%%%%%%%%%%%%%%%%%%%%%%%%%%%%%%%%%%%%%%%%%%%%%%%%%%%%%%%%%%%%%%%%%%%%%%%%%%%%%%%%%%%%%%%%%%%%%%%%%%%%%%%%%%%

%%%%%%%%%%%%%%%%%%%%%%%%%%%%%%%%%%%

\section{Proof of sacttering/blowup dichotomy (Theorem \ref{thm:kmw08})}\label{sec:focener}

For this section, we refer to Kenig's lecture notes \cite{Kenig01,Kenig02} and Kenig-Merle \cite{KM1}.
To describe the proof of Theorem \ref{thm:kmw08}, we start out by a brief  review of the local Cauchy problem.
To do this, we introduce some definitions and background materials.
\begin{definition}[Strong solution]\label{def: def1.1}  Let $I$ be a nonempty time interval including $t_0$.
 A function $u:~I\times\R^d\to\R$ is a strong solution to problem \eqref{equ:nlwecrifoc},
if
$$(u,u_t)\in C_t^0(J;\dot H^1(\R^d)\times\dot L^2(\R^d), \;\text{and}\; u\in
L_{t,x}^\frac{2(d+1)}{d-2}(J\times\R^d)$$
for any compact $J\subset
I$ and for each $t\in I$ such that
\begin{equation}\label{duhamel}
{u(t)\choose \dot{u}(t)} = V_0(t-t_0){u_0(x) \choose u_1(x)}
-\int^{t}_{t_0}V_0(t-s){0 \choose f(u(s))} \mathrm{d}s,
\end{equation}
where
\begin{equation}\label{v0tdefin}
V_0(t) = {\dot{K}(t)\quad K(t) \choose \ddot{K}(t)\quad \dot{K}(t)},
\quad K(t)=\frac{\sin(t\omega)}{\omega},\quad
\omega=\big(-\Delta\big)^{1/2},
\end{equation} and the dot denotes the time derivative.
 We refer to the interval $I$ as
the lifespan of $u$. We say that $u$ is a maximal-lifespan solution
if the solution cannot be extended to any strictly larger interval.
We say that $u$ is a global solution if $I=\R.$
\end{definition}

The solution lies in the space $L_{t,x}^\frac{2(d+1)}{d-2}(I\times\R^d)$ locally in time is a natural fact
since by Strichartz estimate,
the linear flow always lies in this space.  We define the scattering
size of a solution $u$ to problem \eqref{equ:nlwecrifoc} on a time interval
$I$ by
\begin{equation}\label{scattersize1.1}
 \|u\|_{S(I)}:=\|u\|_{
L_{t,x}^\frac{2(d+1)}{d-2}(I\times\R^d)}.
\end{equation}
Standard arguments show that if a solution $u$ of problem \eqref{equ:nlwecrifoc} is
global, with $S_{\R}(u)<+\infty$, then it scatters, i.e.  there
exist unique
$(v_0^\pm,v_1^\pm)\in\dot H^1(\R^d)\times L^2(\R^d)$ such
that
\begin{equation}\label{1.2.1}
\bigg\|{u(t)\choose \dot{u}(t)}-V_0(t){v_0^\pm\choose
v_1^\pm}\bigg\|_{\dot H^1(\R^d)\times L^2(\R^d)}
\longrightarrow 0,\quad \text{as}\quad t\longrightarrow \pm\infty.
\end{equation}

\vskip0.1cm

 The notion closely associated with
scattering is the definition of blowup:
\begin{definition}[Blowup]\label{def1.2.1} We call that a  maximal-lifespan solution
$u:I\times\R^d\to\mathbb{R}$ of problem \eqref{equ:nlwecrifoc} blows up forward in time if there
exists a time $\tilde{t}_0\in I$ such that $\|u\|_{S([\tilde{t}_0,\sup
I))}=+\infty$. Similarly, $u(t,x)$ blows up backward in time if there
exists a time $\tilde t_0\in I$ such that
$\|u\|_{S(\inf I,\tilde{t}_0]}=+\infty$.
\end{definition}

\subsection{Local Cauchy problem}

 We first consider
the associated linear problem,
\begin{equation} \label{linearwave}
\left\{ \aligned
    &w_{tt}-\Delta w =h,  \\
    &(w(0),\pa_tw(0))=(w_0,w_1)\in\dot H^1(\R^d)\times L^2(\R^d).
\endaligned
\right.
\end{equation}

 As is well known, the solution is given by
 \begin{align*}
w(x,t)=&\cos(t\sqrt{-\Delta})w_0+\frac{\sin(t\sqrt{-\Delta})}{\sqrt{-\Delta}}w_1
 +\int_0^t\frac{\sin((t-s)\sqrt{-\Delta})}{\sqrt{-\Delta}}h(s)ds\\
\triangleq& S(t)(w_0,w_1)+\int_0^t\frac{\sin((t-s)\sqrt{-\Delta})}{\sqrt{-\Delta}}h(s)ds.
 \end{align*}

\underline{\bf The normal Strichartz norms}
\vskip0.12cm
\begin{equation}\label{equ:siwi}
\|f\|_{S(I)}:=\|f\|_{L_{t,x}^\frac{2(d+1)}{d-2}(I\times\R^d)},\;\;
~\|f\|_{W(I)}:=\|f\|_{L_{t,x}^\frac{2(d+1)}{d-1}(I\times\R^d)}.
\end{equation}

\begin{theorem}[Strichartz estimates, \cite{GiV95}] We have
\begin{align}
&\sup_{t}\|(w(t),\pa_tw(t))\|_{\dot H^1\times L^2}+\|D^\frac12w\|_{W(\R)}+\|\pa_tD^{-\frac12}w\|_{W(\R)}\nonumber\\
&+\|w\|_{S(\R)}+\|w\|_{L_t^\frac{d+2}{d-2}L_x^\frac{2(d+2)}{d-2}}\nonumber\\
\leq&C\|(w_0,w_1)\|_{\dot H^1\times L^2}+C\|D^\frac12h\|_{L_{t,x}^\frac{2(d+1)}{d+3}}\label{Strichartz-one}
\end{align}
\end{theorem}
Because of the appearance of $D^\frac12$ in these estimates, we also need to use the
following version of the chain rule for fractional derivatives (see \cite{KPV93}).

\begin{lemma}[Chain rule for fractional derivatives]\label{lem:frderi}
Assume $F\in C^2,~F(0)=F'(0)=0,$ and that for all $a,b$, we have
\begin{align*}
|F'(a+b)|\leq C(|F'(a)|+|F'(b)|),~|F''(a+b)|\leq C(|F''(a)|+|F''(b)|).
\end{align*}
Then, for
$$\al\in(0,1),~\frac1p=\frac1{p_1}+\frac1{p_2}=\frac1{r_1}+\frac1{r_2}+\frac1{r_3},$$
we have
\begin{align}\|D^\al F(u)\|_{L^p}\leq C\|F'(u)\|_{L^{p_1}}\|D^\al u\|_{L^{p_2}},\label{Strichartz-two}
\end{align}
\begin{align}
&\|D^\al(F(u)-F(v))\|_{L^p}\leq C\big(\|F'(u)\|_{L^{p_1}}+\|F'(v)\|_{L^{p_1}}\big)\|D^\al(u-v)\|_{L^{p_2}}\nonumber\\
&+C\big(\|F''(u)\|_{L^{r_1}}+\|F''(v)\|_{L^{r_1}}\big)\cdot(\|D^\al u\|_{L^{r_2}}+\|D^\al v\|_{L^{r_2}})\|u-v\|_{L^{r_3}}.\label{Strichartz-three}
\end{align}
\end{lemma}

\noindent \underline{\bf Local theory}\;  Using these estimates, we obtain
\begin{theorem}[Local theory, \cite{GiV95,Kapi94,KM1}]\label{thm:local}
 Assume $$(u_0,u_1)\in\dot{H}^1\times L^2,~\|(u_0,u_1)\|_{\dot H^1\times L^2}\leq A.$$
  Then, there exists $\delta=\delta(A)$ such that if
$$\|S(t)(u_0,u_1)\|_{S(I)}\leq \delta,$$
there exists a unique solution to \eqref{equ:nlwecrifoc} in $I\times\R^d$ s.t. $(u,\pa_t)\in C(I;\dot H^1\times L^2)$, and
\begin{equation}
\|D^\frac12u(t)\|_{W(I)}+\|\pa_tD^{-\frac12} u\|_{W(I)}\leq C(A),~\text{and}~\|u\|_{S(I)}\leq 2\delta.
\end{equation}
Moreover, $(u_0,u_1)\mapsto~(u,\pa_tu)\in C(I;\dot H^1\times L^2)$ is Lipschitz.
\end{theorem}
\vskip0.15cm

\begin{remark}\label{remark-local-theory-one}
{\rm (i)} ({\bf Small data theory}) From Strichartz estimate and Sobolev embedding, $\exists~\tilde\delta$, s.t.
 if $\|(u_0,u_1)\|_{\dot H^1\times L^2}\leq \tilde\delta$, the hypothesis of Theorem \ref{thm:local} holds for $I=\R$.
 Moreover, $u$ scatters in the following sense: there exists $(u_0^{\pm},u_1^\pm)\in \dot H^1\times L^2$ such that
$$\lim_{t\to\pm\infty}\|(u(t),\pa_tu(t))-S(t)(u_0^\pm,u_1^\pm)\|_{\dot H^1\times L^2}=0.$$

\vskip0.15cm
{\rm (ii)} ({\bf Maximal life interval}) It is easy to see that given $(u_0,u_1)\in\dot{H}^1\times L^2$, there exists time interval $I$ such that
$\|S(t)(u_0,u_1)\|_{S(I)}\leq\delta,$ so the Theorem \ref{thm:local} applies on $I$.
Moreover,  here exists a maximal interval $I^*=(-T_-(u_0),T_+(u_0))$ where
$$u\in C(I';\dot H^1)\cap\big\{\nabla u\in S(I')\big\}~\text{of all}~I'\subset\subset I^*~\text{is defined}.$$
We call $I^*$ the maximal interval of existence.

\vskip0.15cm

{\rm (iii)} ({\bf Blowup criterion}) If $T_+(u_0)<+\infty$, we must have
\begin{equation}\label{nlsblowup}
\|u\|_{S[0,T_+(u_0,u_1))}=+\infty.
\end{equation}

\vskip0.15cm

{\rm (iv)}({\bf Energy conservation}) For $t\in I^*$, we have
$$E(u_0,u_1)\triangleq\frac12\int\big(|\nabla u_0|^2+u_1^2\big)dx-\frac1{2^\ast}\int|u_0|^{2^\ast}dx=E(u(t),\pa_tu(t)).$$

\vskip0.15cm

{\rm (v)} ({\bf Momentum conservation}) For $t\in I^*$, we have
$$\int\nabla u(t)\cdot \pa_tu(t)dx=\int\nabla u_0\cdot u_1dx,$$
which is  a very important conserved quantity in
the energy space, and is crucial to be able to treat
non-radial data.

\end{remark}

\begin{proof}[{\bf Proof of blowup criterion}]  If \eqref{nlsblowup} is not true, then
$$M:=\|u\|_{S[0,T_+(u_0,u_1))}<+\infty.$$
For $\varepsilon>0$ to be chosen, partition $[0,T_+(u_0,u_1))=\bigcup\limits_{j=1}^{\gamma(\varepsilon,M)}I_j$, so that
$$\|u\|_{S(I_j)}\leq\varepsilon.$$
If we write $I_j=[t_j,t_{j+1})$, using Strichartz estimate, we have
$$\sup_{t\in I_j}\|(u,\pa_tu)(t)\|_{\dot{H}^1\times L^2}+\|D^\frac12u\|_{W(I_j)}\leq C\big\|(u_0,u_1)\big\|_{\dot{H}^1\times L^2}+C
\|u\|_{S(I_j)}^\frac4{d-2}\|D^\frac12u\|_{W(I_j)}.$$
If $C\varepsilon^\frac4{d-2}\leq\frac12$, we can show by continuous argument
$$\sup_{t\in[0,T_+(u_0,u_1))}\|(u,\pa_tu)(t)\|_{\dot{H}^1\times L^2}+\|D^\frac12u\|_{W([0,T_+(u_0,u_1)))}\leq C(M).$$
Choose now $\{t_n\}\nearrow T_+(u_0,u_1)$ and show again using Strichartz estimate
$$\Big\|S(t-t_n)\big(u(t_n),\pa_tu(t_n)\big)\Big\|_{S(t_n,T_+(u_0,u_1))}\leq \frac{\delta}2.$$
But then by continuous, we have for some $\varepsilon_0>0$
$$\Big\|S(t-t_n)\big(u(t_n),\pa_tu(t_n)\big)\Big\|_{S(t_n,T_+(u_0,u_1)+\varepsilon_0)}\leq\delta,$$
which, by the Theorem \ref{thm:local} contradicts the definition of $T_+(u_0,u_1)$.
\end{proof}

\noindent \underline{\bf Perturbation Theorem}\; We next turn to another fundamental result in the ``local Cauchy theory",
the so called ``Perturbation Theorem".  More
precisely, given an approximate equation
\begin{equation*}
\pa_{tt}\tilde u-\Delta\tilde u=|\tilde u|^\frac4{d-2}\tilde u+e
\end{equation*}
to (FNLW), with $e$ small in sense of suitable time-space norm  and
$(\tilde{u}_0,\tilde{u}_1)$ is close to $(u_0,u_1)$ in energy space,
is it possible to show that solution $u$ to (FNLW) stays very close
to $\tilde{u}$? Note that the question of continuous dependence on
the data corresponds to the case $e=0$.

\begin{theorem}[Stability Lemma]\label{thm:pertur}
 Assume $\tilde u$ is a near
solution on $I\times\R^d$
\begin{equation}\label{eque1.1}
\pa_t^2\tilde u-\Delta \tilde u=F(\tilde u)+e
\end{equation}
such that

\vskip0.12cm

{\rm (i)}  Finite energy  Strichartz solution:
$$\big\|(\tilde u,\pa_t\tilde u)\big\|_{L_t^\infty(I;\dot{H}^1\times
L^2)}+\big\||\nabla|^\frac12\tilde{u}\big\|_{L_{t,x}^{\frac{2(d+1)}{d-1}}}\leq
E.$$

\vskip0.12cm
{\rm (ii)}  The finite energy of difference of initial data:

$$\big\|(\tilde u_0-u_0,\tilde u_1-u_1)\big\|_{\dot{H}^1\times
L^2}\leq E'.$$

\vskip0.12cm

{\rm (iii)} Smallness perturbation:
\begin{align}\label{eque1.4}
\big\|S(t)(\tilde u_0-u_0,\tilde u_1-u_1)\big\|_{S(I)}\leq&
\varepsilon\\\label{eque1.5}
\big\||\nabla|^\frac12e\big\|_{L_{t,x}^{\frac{2(d+1)}{d+3}}}\leq&\varepsilon.
\end{align}
\vskip0.12cm

\noindent Then there exists $\varepsilon_0=\varepsilon_0(d,E,E')$ such that if
$\varepsilon\in(0,\varepsilon_0),$ for {\rm(FNLW)} with initial
data
$(u_0,u_1)$, there exists a unique solution $u$ on $I\times\R^d$ with the properties
\begin{align}\label{eque1.6}
\|\tilde u-u\|_{S(I)}\leq
C(d,E,E')\cdot\varepsilon^c,\\\label{eque1.7}\|\tilde
u-u\|_{\dot{S}^1(I)}\leq C(d,E,E')\cdot E',
\end{align}
where $c$ satisfies: when $d\in\{3,4,5\}$, $c=1;$ when $d\geq6$,
$c=c(d)\in(0,1)$; and
$$\|u\|_{\dot S^1(I)}\triangleq\sup_{(q,r)\in\Lambda}\big\{\|u\|_{L_t^q\dot
B_{r,2}^{1-\beta(r)}(I\times\R^d)}+\|\pa_tu\|_{L_t^q\dot
B_{r,2}^{-\beta(r)}(I\times\R^d)}\big\},$$ where
$\beta(r)=\frac{d+1}2\Big(\frac12-\frac1r\Big),$
$$\Lambda:=\Big\{(q,r):~\frac2q=(d-1)\Big(\frac12-\frac1r\Big),~q,r\geq2,~(q,r,d)\neq(2,\infty,3)\Big\}.$$
\end{theorem}

\subsection{Variational estimates}

\begin{lemma}[Coercivity Lemma]\label{lem:coe}
Assume that $E(v_1,v_2)\leq(1-\delta_0)E(W)$,~$\delta_0\in(0,1)$. Then
\begin{enumerate}
\item If $\|\nabla v_1\|_{L^2}<\|\nabla W\|_{L^2}$, then there exists $\delta_1=\delta_1(\delta_0)>0$ such that
\begin{align*}
&{\rm (i)} \quad\|\nabla v_1\|_{L^2}\leq(1-\delta_1)\|\nabla W\|_{L^2};\\
&{\rm (ii)} \quad\int_{\mathbb R^d}\big(\|\nabla v_1|^2-|v_1|^\frac{2d}{d-2}\big)\;dx\geq\delta_1\|\nabla v_1\|_{L^2}^2;\\
&{\rm (iii)} \quad c\|(v_1,v_2)\|_{\dot{H}^1\times L^2}^2\leq E(v_1,v_2)\leq\frac12 \|(v_1,v_2)\|_{\dot{H}^1\times L^2}^2.
\end{align*}

\item If $\|\nabla v_1\|_{L^2}>\|\nabla W\|_{L^2}$,  then there exists $\delta_2=\delta_2(\delta_0)>0$ such that
$$\|\nabla v_1\|_{L^2}\geq(1+\delta_2)\|\nabla W\|_{L^2}.$$
\end{enumerate}
\end{lemma}

\begin{proof}
 Let
\begin{equation}
f(y)=\frac12y-\frac{C_d^{2^\ast}}{2^\ast}y^{2^\ast/2},\quad ~\bar y\triangleq\|\nabla v_1\|_2^2,
\quad ~2^\ast\triangleq\frac{2d}{d-2},
\end{equation}
where $C_d$ is the sharp constant of Sobolev inequality
$$\|g(x)\|_{2^*}\le C_d\|\nabla g\|_2, \;\quad \; \big(~ \|W(x)\|_{2^*}=C_d\|\nabla W(x)\|_2~\big).$$
\begin{center}
\includegraphics[width=3in,height=2in]{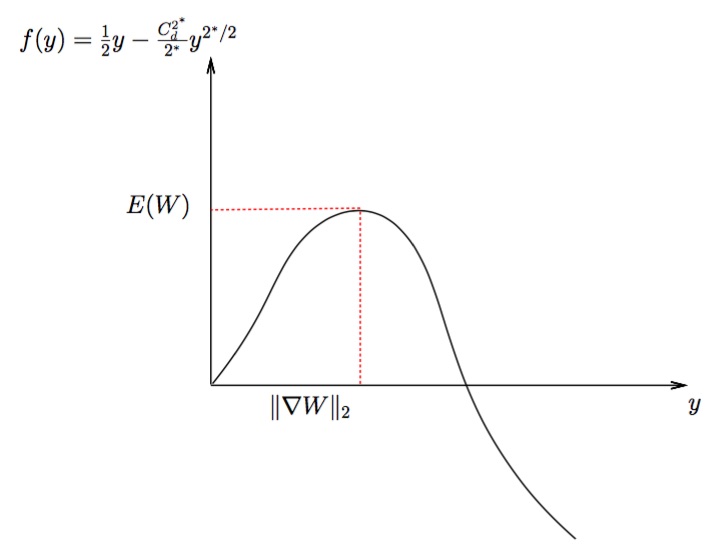}
\end{center}
Note that $f(0)=0,~f(y)>0$ for $y$ near $0$, $y>0$, and that
$$f'(y)=\frac12-\frac{C_d^{2^\ast}}{2}y^{\frac{2^\ast}2-1},$$
so that
$$f'(y)=0\;\;\Longleftrightarrow y=y_c=\frac1{C_d^d}=\|\nabla W\|_2^2.$$
Also
\begin{equation*}
\left\{\begin{aligned}
&f(\bar{y})=f(\|\nabla v_1\|^2)\le E(v_1,v_2),\quad \text{\rm (sharp Sobolev inequality)}\\
&f(\|\nabla W\|_2^2)=f(y_c)=\frac1{dC_d^d}=E(W,0).\end{aligned}\right.\end{equation*}

Since $f\ge 0$ is  strictly
increasing in $[0,y_c]=[0,\|\nabla W\|^2_2]$,  $f''(y_c)<0$, so $f$ arrive the maximal at
$y=y_c=\|\nabla W\|_2^2$. Note that
$$E(v_1,v_2)=(1-\delta_0)E(W,0)<E(W,0),$$
 there exists $\delta_1=\delta(\delta_0)$ such that
$$f^{-1}(E(v_1,v_2))=f^{-1}((1-\delta_0)E(W,0))=(1-\delta_1)\|\nabla W\|_2^2.$$
Hence
$$\|\nabla v_1\|_2^2=\bar y\le f^{-1}(E(v_1,v_2))=f^{-1}((1-\delta_0)E(W,0))\le
(1-\delta_1)\|\nabla W\|_2^2.$$
 This shows (1)(i). By the same way, we conclude (2).

 For (1)(ii) note that
\begin{align*}
\int_{\mathbb R^d}\big(|\nabla v_1|^2-|v_1|^{2^\ast}\big)dx\geq&\int_{\mathbb R^d}|\nabla v_1|^2dx
-C_d^{2^\ast}\Big(\int_{\mathbb R^d}|\nabla v_1|^2dx\Big)^{2^{\ast}/2}\\
=&\Big[1-C_d^{2^\ast}\Big(\int_{\mathbb R^d}|\nabla v_1|^2dx\Big)^{2/d-2}\Big]
\int_{\mathbb R^d} |\nabla v_1|^2dx\\
\geq&\Big[1-C_d^{2^\ast}(1-\delta_1)^{2/d-2}\Big(\int_{\mathbb R^d}|\nabla W|^2dx\Big)^{2/d-2}\Big]\int_{\mathbb R^d} |\nabla v_1|^2dx\\
=&\Big[1-(1-\delta_1)^{2/d-2}\Big]\int_{\mathbb R^d} |\nabla v_1|^2dx,
\end{align*}
which gives (1)(ii).

For (1)(iii), the upper bound is obviously, and the lower bound follows from (1)(ii).

\end{proof}

From this static Lemma, we obtain dynamic consequence.

\begin{lemma}[Energy trapping]\label{lem:enetra} Let $u$ be a solution of {\rm (FNLW)} with maximal interval $I$,
$$E(u_0,u_1)<E(W,0).$$
Choose $\delta_0>0$ such that $E(u_0,u_1)\leq(1-\delta_0)E(W,0).$ Then, for $t\in I,$ we have
\begin{enumerate}
\item if $\|\nabla u_0\|_2<\|\nabla W\|_2$,  then there exists $\delta_1=\delta_1(\delta_0)>0$ such that

\begin{align*}
&{\rm (i)} \;\;
\|\nabla u(t)\|_2^2\leq(1-\delta_1)\|\nabla W\|_2^2,~E(u_0,u_1)\geq0;\\
&{\rm (ii)}\;\; \int_{\mathbb R^d}\big(|\nabla u(t)|^2-|u(t)|^{2^\ast}\big)\geq\delta_1\int|\nabla u|^2, \text{\rm (``coercivity")};\\
&{\rm (iii)} \;\; E(u,\pa_tu)\simeq_{\delta_0}\|\nabla (u(t),\pa_tu(t))\|_{\dot H^1\times L^2}^2
\simeq_{\delta_0}\|\nabla (u_0,u_1)\|_{\dot H^1\times L^2}^2.
\end{align*}

\item If $\|\nabla u_0\|_{L^2}>\|\nabla W\|_{L^2}$,  then there exists $\delta_2=\delta_2(\delta_0)>0$ such that
$$\|\nabla u(t)\|_{L^2}\geq(1+\delta_2)\|\nabla W\|_{L^2}.$$
\end{enumerate}
\end{lemma}

\subsection{Proof of finite-time blowup}

In this subsection, we prove (ii) in Theorem \ref{thm:kmw08}.

\begin{theorem}[Finite time blowup]\label{thm:fiblwo}
Assume $(u_0,u_1)\in\dot{H}^1\times L^2$ with
$$E(u_0,u_1)<E(W,0)~\text{and}~\|\nabla u_0\|_{L^2}>\|\nabla W\|_{L^2}.$$
Then, the solution $u$ to \eqref{equ:nlwecrifoc} blows up in finite time in both directions.
\end{theorem}

\begin{proof} Without loss of generality,  it suffices to consider  $t>0$.
By contradiction, we assume that the solution $u$ exists for all
$t>0$.

\vskip0.20cm

\noindent\underline{\bf Case 1: $u_0\in L^2(\R^d).$}\; By Lemma \ref{lem:enetra}, we have
\begin{align*}
\int_{\mathbb R^d}|\nabla u(t)|^2\geq (1+\delta_1)\int_{\mathbb R^d}|\nabla W|^2,\quad ~t\in I,
\end{align*}
and
$E(W,0)\geq E(u(t),\pa_tu)+\tilde\delta_0$  implies that
$$\int_{\mathbb R^d}|u(t)|^{2^\ast}\geq\frac{d}{d-2}\int_{\mathbb R^d}|\pa_tu(t)|^2+\frac{d}{d-2}
\int_{\mathbb R^d}|\nabla u(t)|^2-2^\ast E(W,0)+2^\ast\tilde\delta_0.$$
Let $y(t)=\int_{\R^d} |u(t)|^2dx$, so that
$$y'(t)=2\int_{\R^d} u(t)\pa_tu(t)dx,\quad ~y''(t)=2\int_{\R^d}\big(|\pa_tu|^2-|\nabla u|^2+|u(t)|^{2^\ast}\big)dx.$$
Thus,
\begin{align*}
y''(t)\geq&2\int_{\mathbb R^d}|\pa_tu|^2dx+\frac{2d}{d-2}\int_{\mathbb R^d}|\pa_tu|^2dx-2\cdot2^\ast E(W,0)\\
&+\tilde\delta_0+\frac{2d}{d-2}\int_{\mathbb R^d}|\nabla u(t)|^2dx-2\int_{\mathbb R^d}|\nabla u(t)|^2dx\\
=&\frac{4(d-1)}{d-2}\int_{\mathbb R^d}|\pa_tu|^2dx+\frac4{d-2}\Big(\int_{\mathbb R^d}|\nabla u(t)|^2dx
-\int_{\mathbb R^d}|\nabla W|^2dx\Big)+\tilde \delta_0\\
\geq&\frac{4(d-1)}{d-2}\int_{\mathbb R^d}|\pa_tu|^2dx+\tilde \delta_0.
\end{align*}
If $I\cap[0,+\infty)=[0,+\infty)$, there exists $t_0>0$ so that
$$y'(t_0)>0,\;\;\;~y'(t)>0,\qquad~t>t_0.$$
For $t>t_0,$ we have
\begin{equation}
y(t)y''(t)\geq\frac{4(d-1)}{d-2}\int|\pa_tu|^2\cdot\int|u|^2\geq\frac{d-1}{d-2}y'(t)^2,
\end{equation}
so that
$$\frac{y''(t)}{y'(t)}\geq\frac{d-1}{d-2}\frac{y'(t)}{y(t)},$$
or
$$y'(t)\geq C_0y(t)^\frac{d-1}{d-2},~\forall~t>t_0.$$
This implies
$$y(t)^{-\frac1{d-2}}\leq y(t_0)^{-\frac1{d-2}}-C_0(t-t_0).$$
Hence, for large $t$, $y(t)<0$ which contradicts with $y(t)\geq0$.

\vskip 0.1in

\underline{\bf Case 2: $u_0\in\dot{H}^1(\R^d).$} \; The proof follows by using, in addition, localization and finite speed of propagation.

\vskip0.2cm

Let $A=\|(u_0,u_1)\|_{\dot{H}^1\times L^2}>0$. From finite speed of propagation and local theory (see Lemma \ref{lem:fisma} below), there exists $\eps_0>0$ such that, for $0<\eps<\eps_0$, there
exists $M_0=M_0(\eps)$, with
\[
\int_{|x|\geq M_0+t} \Big(|\nabla u(x,t)|^2 + |\pa_t u(x,t)|^2
+ |u(x,t)|^{2^*} + \frac{|u(x,t)|^2}{|x|^2}\Big) dx \leq\eps,
\]
for $t\in[0,T_+(u_0,u_1))$. Assume that $T_+(u_0,u_1)=+\infty$ to
reach a contradiction.

Let $f(\tau)$ be a solution to the differential
inequality ($f\geq 0$)
\begin{equation}\label{eq:7.5}
f'(\tau) \geq B f(\tau)^{\frac{d-1}{d-2}}, \;\;\; f(0) = 1.
\end{equation}
Then, the time of blow-up for $f$ is $\tau_*$, with $\tau_*\leq K_d
B^{-1}$.

Let $\phi\in C_0^\infty(B_2)$ such that $\phi\equiv 1$ on $|x|<1$, $0\leq\phi\leq 1$.
Consider now, for $R$ large,
$$y_R(t)=\int
u^2(x,t)\phi(x/R)dx.$$
Then, $y'_R(t) = 2\int_{\R^d} u\, \pa_t u\,
\phi(x/R)dx$,  and
\[
y''_R(t) = 2 \left[ \int_{\R^d} \Big(|\pa_t u|^2 - |\nabla_x u|^2
+ |u|^{2^*} \Big)dx \right] + O(r(R)),
\]
with
$$r(R)=\int_{|x|\geq R}\Big(\frac{|u|^2}{|x|^2}+|u|^{2^\ast}+|\nabla u|^2+|\pa_tu|^2\Big)\;dx.$$
Arguing as in the proof of case 1, we find that
\[
y''_R(t) \geq \frac{4(d-1)}{d-2}\int_{\mathbb R^d} |\pa_t u)|^2 \phi(x/R)dx + \widetilde{\widetilde{\delta}}_0 + O(r(R)).
\]
Choose now $\eps_1$ so small, and $M_0=M_0(\eps_1)$, as above, so that,
for $R>M_0/2$, $O(r(R))\leq\eps_1$, $\eps_1\leq\widetilde{\widetilde{\delta}}_0/2$.
We then have, for $0<t<R/2$,
 \begin{equation}\label{eq:7.6}\left\{\begin{aligned}
 y''_R(t) &\geq \widetilde{\widetilde{\delta}}_0/2, \\
 y''_R(t) &\geq \frac{4(d-1)}{d-2} \int_{\mathbb R^d} |\pa_t u|^2 \phi(x/R)dx.
\end{aligned}\right.
 \end{equation}
It is easy to verify that
 \begin{equation}\label{eq:7.7}
 y_R(0) \leq C M_0^2 A^2 + \eps_1 R^2, \; \;
|y_R'(0)| \leq C M_0 A^2 + \eps_1 R,
 \end{equation}
where $A=\|(u_0,u_1)\|_{\dot H^1\times L^2}$. Let
$$
T = \frac{4CM_0 A^2+2\eps_1 R+2R\sqrt{\eps_1}}{\widetilde{\widetilde{\delta}}_0},
$$
then, (if $T<R/2$)
 \begin{align*}
 y'_R(T) \geq& T \frac{\widetilde{\widetilde{\delta}}_0}{2} + y'_R(0)
 \geq 2CM_0 A^2 +\eps_1 R + R \sqrt{\eps_1} - C M_0 A^2 -\eps_1 R  \\
=& CM_0 A^2 + R \sqrt{\eps_1}.
 \end{align*}
Thus, there exists $0<t_1<T$ such that $y_R'(t_1)=CM_0
A^2+R\sqrt{\eps_1}$, and for $0<t<t_1$, we have $y'_R(t)<CM_0 A^2 +
R\sqrt{\eps_1}$. Note that, in light of \eqref{eq:7.6},
$y'_R(t)>y'_R(t_1)>0$, $t>t_1$ ($t<R/2$) and also
% \bbea
 $$
y_R(t_1) \leq y_R(0) + \int_0^{t_1} y'_R(\tau) d\tau\leq y_R(0) + t_1
(CM_0 A^2+R\sqrt{\eps_1}) = y_R(0) + t_1 y'_R(t_1).
$$

For sufficient small $\eps_1$, we claim
$$T=\frac{4CM_0A^2}{\widetilde{\widetilde{\delta}}_0} +
\frac{2\eps_1 R}{\widetilde{\widetilde{\delta}}_0} +
\frac{2\sqrt{\eps_1}R}{\widetilde{\widetilde{\delta}}_0}\le \frac{1}{8K_d} R. $$
In fact, we first
choose $\eps_1$ so small that
$\frac{2\eps_1}{\widetilde{\widetilde{\delta}}_0}
+\frac{2\sqrt{\eps_1}}{\widetilde{\widetilde{\delta}}_0}\leq
\frac{1}{32K_d}$, where $K_d$ is the constant defined at the
beginning of the proof, and $R$ so large that
$\frac{4CM_0A^2}{\widetilde{\widetilde{\delta}}_0}\leq\frac{1}{16
K_d}R$. We then prove the claim, and also ensure
$T\leq\frac{R}{8}$. Thus,
\[
y_R(t_1) \leq CM_0^2A^2 +\eps_1R^2 + \frac{R}{8K_d}y_R'(t_1).
\]
If we now use the argument in the proof of Case 1,
for the function
$$\widetilde{y}_R(\tau)=y_R(t_1+\tau),\;\;\;\;0\leq\tau\leq R/4,$$
 in light of \eqref{eq:7.6}, we see that, for
$0<\tau<R/4$, we have
$$\Big(\log\widetilde{y}'_R(\tau)\Big)'
\geq \Big(\frac{(d-1)}{(d-2)} \log\widetilde{y}_R(\tau)\Big)'.$$
So that, by integration,
\[
\frac{\widetilde{y}'_R(\tau)}{\widetilde{y}'_R(0)}
\geq \left[ \frac{\widetilde{y}_R(\tau)}{\widetilde{y}_R(0)}
\right] ^{\frac{(d-1)}{(d-2)}} \quad \text{for } 0\leq\tau\leq R/4.
\]
Thus, if $f(\tau)=\frac{\widetilde{y}_R(\tau)}{\widetilde{y}_R(0)}$ and
$B=\frac{\widetilde{y}'_R(0)}{\widetilde{y}_R(0)}=
\frac{y'_R(t_1)}{y_R(t_1)}$, we have that $f$ is a solution of \eqref{eq:7.5} for $0\leq\tau\leq R/4.$ Thus, we must have
\[
\frac{R}{4} \leq \frac{y_R(t_1)}{y'_R(t_1)} K_d
\leq \frac{K_d(CM_0^2 A^2 + \eps_1 R^2)}{y_R'(t_1)} +\frac{R}{8},
\]
or
\begin{align*}
\frac{1}{8} \leq& \frac{K_d(CM_0^2 A^2 + \eps_1 R^2)}{CM_0A^2 R+\sqrt{\eps_1}R^2}
 = \frac{K_d[CM_0^2 A^2/R^2 + \eps_1]}{[CM_0A^2/R+\sqrt{\eps_1}]}  \\
\leq& K_d M_0/R + K_d \sqrt{\eps_1}.
\end{align*}
By taking $K_d\sqrt{\eps_1}<\frac{1}{32}$, and $\frac{K_d M_0}{R}<\frac{1}{32}$ we
reach a contradiction, which gives the proof of   Theorem \ref{thm:kmw08} (ii).

\end{proof}

\subsection{Concentration-compactness procedure}

We proceed ``concentration-compactness"  in dispersive framework developed by Kenig-Merle with profile decomposition
of Bahouri-G\'{e}rard \cite{BG} for $d=3$ and Bulut\cite{Bu2010} for $d\geq4$.
 Thus, arguing by contradiction, we find a number $E_c$, with $0<\eta_0\leq
E_c<E(W, 0)$ with the property that if
$$E(u_0, u_1)<E_c,\;\; \|\nabla u_0\|_2<\|\nabla W\|_2,$$
then $\|u\|_{S(I)}<\infty$. And $E_c$ is optimal with this property. We will see that this leads to
a contradiction. We have:

\begin{prop}[Critical element]\label{prop:crit}
There exists
$$(u_{0,c},u_{1,c})\in\dot H^1\times L^2,\;\;~\|\nabla u_{0,c}\|_2<\|\nabla W\|_2,~ E(u_{0,c},u_{1,c})=E_c,$$
such that for the corresponding solution $u_c$, we have $$\|u_c\|_{S(I)}=+\infty.$$
\end{prop}

\begin{prop}[Compactness] For any $u_c$ as in the above Proposition, with (say)
$$\|u_c\|_{S(I_+)}=+\infty,~\;\;~ I+=
I\cap[0,+\infty),$$
there exist $x(t)\in \mathbb R^d,~\la(t)\in\R^+,~t\in I_+$, such that
\begin{equation*}
K=\Big\{\vec{v}(x,t)=\Big(\frac1{\la(t)^\frac{d-2}2}u\Big(\frac{x-x(t)}{\la(t)},t\Big),
\frac1{\la(t)^\frac{d}2}\pa_tu\Big(\frac{x-x(t)}{\la(t)},t\Big)\Big),~t\in I_+\Big\}
\end{equation*}
has compact closure in $\dot H^1\times L^2.$
\end{prop}

\begin{remark}
\begin{itemize}
\item[(1)] Because of the continuity of $u(t)\in\dot H^1$ with respect to $t\in I$,  we can
construct $\la(t),~x(t)$ continuous in $[0, T_+(u_{0,c},u_{1,c}))$ with $\la(t)>0.$

\item[(2)] Because of scaling and the compactness of $\bar K$ above, if $T_+(u_{0,c},u_{1,c})<\infty$,
one always has that
$$\la(t)\geq\frac{C_0(K)}{T_+(u_{0, c},u_{1,c})-t}.$$

\item[(3)] If $T_+(u_{0,c},u_{1,c})=+\infty$, we can always find another (possibly different) critical
element $v_c$ with a corresponding $\tilde \la(t)$ such  that $\tilde\la(t)\geq A_0>0$ for $t\in[0, T_+(v_{0,c},v_{1,c})).$
(Again by compactness of $\bar K$.)

\item[(4)] One can use the ``profile decomposition" to also show that there exists
a nondecreasing function $g:~(0,E_c]\to~[0,+\infty)$ so that if $$\|\nabla u_0\|_2<\|\nabla W\|_2~\text{and}~
E(u_{0,c},,u_{1,c})\leq E_c-\eta,$$ then $\|u\|_{S(\R)}\leq g(\eta)$.
\end{itemize}

\end{remark}

Up to here, we have used, in this step, only variational argument and ``general arguments".
To proceed further we need to use specific features of (NLW) to establish fine
properties of critical elements.

The first one is a consequence of the finite speed of propagation and the
compactness of $\bar K$.

\begin{lemma}[Compact support set]\label{compset} Let $u_c(t)$ be a critical element, with $T_+(u_{0,c},u_{1,c})<+\infty$. (We can assume, by scaling, that $T_+(u_{0,c},u_{1,c})=1$.)
Then there exists $\bar x\in\R^d$ such that
$${\rm supp}(u_c(\cdot,t),\pa_tu_c(\cdot,t))\subset B(\bar x,1-t),~0<t<1.$$
After translation, we can take $\bar{x}=0.$
\end{lemma}

In order to prove Lemma \ref{compset}, we will need the following consequence of the
finite speed of propagation:
\begin{lemma}\label{lem:fisma}
Let $\|(u_0,u_1)\|_{\dot H^1\times L^2}\leq A.$ If for some $M>0$ and $0<\varepsilon<\varepsilon_0(A)$, we have
$$\int_{|x|\geq M}\Big(|\nabla u_0|^2+|u_1|^2+\frac{|u_0|^2}{|x|^2}\Big)dx\leq\varepsilon,\;\; t\in(0,T_+(u_{0,c},u_{1,c})).$$
Then we have
$$\int_{|x|\geq \frac32M+t}\Big(|\nabla u(t)|^2+|\pa_tu(t)|^2+|u(t)|^{2^\ast}+\frac{|u(t)|^2}{|x|^2}\Big)dx\leq C\varepsilon.$$
\end{lemma}

\begin{proof}
  Indeed, we choose $\psi_M(x)\in C^\infty_c(\mathbb R^d)$ such that
\begin{equation*}
\psi_M=\begin{cases}
1\quad |x|\geq\frac3
2M,\\
0\quad |x|\leq M,
\end{cases}\big(u_{0,M},u_{1,M}\big)\triangleq\big(\psi_Mu_0,\psi_Mu_1\big).
\end{equation*}
From our assumptions, we have
$$\|(u_{0,M},u_{1,M})\|_{\dot H^1\times L^2}\leq C\varepsilon.$$
If $C\varepsilon_0<\tilde\delta$, where $\tilde\delta$ is as in the ``local Cauchy theory",
the corresponding solution $u_M$ of (FNLW) has maximal interval $\R$ and
$$\sup_{t\in\R}\|(u_M(t),\pa_tu_M(t))\|_{\dot H^1\times L^2}\leq 2C\varepsilon.$$  But, by finite speed of propagation,
$$u_M\equiv u,~|x|\geq\frac32M+t,~ t\in[0, T_+(u_0, u_1)),$$
 this  proves Lemma \ref{lem:fisma}.
\end{proof}

\begin{proof}[\underline{\bf Proof of Lemma \ref{compset}}]

{\bf Step 1:}   Since  $\la(t)\geq C_0(K)/(1-t)$, we claim that, for any $R_0>0$,
\begin{equation}
\lim_{t\nearrow1}\int_{|x+\frac{x(t)}{\la(t)}|\geq R_0}\Big(|\nabla u_c(x,t)|^2+|\pa_tu_c(x,t)|^2+\frac{|u_c(x,t)|^2}{|x|^2}\Big)dx=0.
\end{equation}
Indeed, if
$$\vec{v}(x,t)=\Big(\frac1{\la(t)^\frac{d-2}2}u_c\Big(\frac{x-x(t)}{\la(t)},t\Big),
\frac1{\la(t)^\frac{d}2}\pa_tu_c\Big(\frac{x-x(t)}{\la(t)},t\Big)\Big),$$
then, by the compactness of $\bar K$ and the fact that $\la(t)\to+\infty$ (as $t\to1$),
$$\int_{|x+\frac{x(t)}{\la(t)}|\geq R_0}\big(|\nabla u_c(x,t)|^2+|\pa_tu_c(x,t)|^2\big)dx=\int_{|y|\geq\la(t)R_0}|\vec{v}(y,t)|^2dy\to0,~t\to1.$$
Because
of this fact, using  Lemma \ref{lem:fisma} with respect to
 backward in time, we have, for each $s\in[0,1),~R_0>0$,
\begin{equation}\label{compset1.2}
\lim_{t\nearrow1}\int_{|x+\frac{x(t)}{\la(t)}|\geq \frac32R_0+(t-s)}\big(|\nabla u_c(x,s)|^2+|\pa_tu_c(x,s)|^2\big)dx=0.
\end{equation}
\vskip0.15cm

{\bf Step 2:} We next show that
\begin{equation}\label{compset1.3}
\frac{|x(t)|}{\la(t)}\leq M,~t\in[0,1).
\end{equation}
If not, we can find $t_n\uparrow1$, so that $|x(t_n)|/\la(t_n)\to+\infty.$
 Then, for all $R>0$,
$$\big\{ |x|\leq R \big\}\subset\Big\{\frac{|x+x(t_n)|}{\la(t_n)}\geq\frac32R_0+t_n\Big\}~\text{for}~n~\text{large},$$
so that, passing to the limit in $n$, for $s=0$, we obtain
$$\int_{|x|\leq R}\Big(|\nabla u_{0,c}|^2+|u_{1,c}|^2\Big)dx=0,$$
a contradiction.

\vskip0.15cm

{\bf Step 3:} Finally, pick $t_n\uparrow1$ so that $x(t_n)/\la(t_n)\to-\bar x.$ Observe that, for every $\eta_0>0,$
for $n$ large enough, for all $s\in[0,1)$,
$$\{|x-\bar x|\geq1+\eta_0-s\}\subset\Big\{\big|x+\tfrac{x(t_n)}{\la(t_n)}\big|\geq\frac32R_0+(t_n-s)\Big\}$$
for some $R_0=R_0(\eta_0)>0$. From this we conclude that
$$\int_{|x-\bar x|\geq 1+\eta_0-s}\Big(|\nabla u(x,s)|^2+|\pa_su(x,s)|^2\Big)dx=0,$$
which gives the Lemma.

\end{proof}

\subsection{Zero momentum of critical element}
We next turn to a
result which is fundamental for us to be able to treat non-radial data.

\begin{theorem}[Zero momentum:~orthogonality for critical elements]\label{zeromome}
Let $(u_{0,c},u_{1,c})$ be as in Proposition \ref{prop:crit}, with $\la(t),~x(t)$ continuous. Assume that
$$\text{\rm either}\;\;~T_+(u_{0,c},u_{1,c})<\infty~\;\;\;\;\text{\rm or}\;\;\;~T_+(u_{0,c},u_{1,c})=\infty,~\la(t)\geq A_0>0.$$
Then
\begin{equation}\label{zerom}
\int_{\mathbb R^d}\nabla u_{0,c}\cdot u_{1,c}dx=0.
\end{equation}
\end{theorem}

\begin{remark}
Note that in the radial case this is automatic by Plancherel Theorem.
For non-radial case, we need a further linear estimate.
\end{remark}

\begin{lemma}[Trace Lemma]\label{lem:trace} Let $w$ solve the linear wave equation
\begin{equation} \label{linearwave-1}
\left\{ \aligned
    &w_{tt}-\Delta w =h\in L_t^1L_x^2,  \\
    &(w(0),\pa_tw(0))=(w_0,w_1)\in\dot H^1(\R^d)\times L^2(\R^d).
\endaligned
\right.
\end{equation}
Then, for $|a|\leq\frac14$, we have
$$\sup_t\Big\|\nabla_{t,x}w\big(\frac{x_1-at}{\sqrt{1-a^2}},x',\frac{t-ax_1}{\sqrt{1-a^2}}\big)\Big\|_{L^2(dx_1dx')}\leq C\Big[\|(w_0,w_1)\|_{\dot H^1\times L^2}+\|h\|_{L_t^1L_x^2}\Big].$$
\end{lemma}

\begin{proof}
 Let $v(x,t)=U(t)f$ be given by
$\hat{v}(\xi,t)=e^{it|\xi|}\hat{f}(\xi)$, with $f\in L^2$. We
will show that
$$
\sup_t\Big\|v\left(\frac{x_1-at}{\sqrt{1-a^2}},x',\frac{t
-ax_1}{\sqrt{1-a^2}} \right)\Big\|_{L^2(dx_1dx')}\leq C\|f\|_{L^2},
$$
which easily implies the desired estimate. But
 \begin{align*}
 v(x,t)=&\int_{\R^d} e^{ix.\xi}e^{it|\xi|}\hat{f}(\xi)d\xi=
\int_{\R^d} e^{ix_1\xi_1}e^{it\abs{\xi}}e^{ix'.\xi'}\hat{f}(\xi)d\xi_1d\xi'\\
=& \int_{\R^d}e^{ix_1\xi_1}e^{it\sqrt{\xi_1^2+\abs{\xi'}^2}}e^{ix'.\xi'}\hat{f}(\xi_1,\xi')d\xi_1d\xi',
 \end{align*}
 so that
 \begin{align*}
 & v\left(\frac{x_1-at}{\sqrt{1-a^2}},x',\frac{t -ax_1}{\sqrt{1-a^2}} \right)\\
=&\int_{\R^d}
e^{i(x_1-at)\xi_1/\sqrt{1-a^2}}e^{i(t-ax_1)\sqrt{\xi_1^2+\abs{\xi'}^2}/\sqrt{1-a^2}}e^{ix'.\xi'}\hat{f}(\xi)d\xi_1d\xi'\\
=&\int_{\R^d}
e^{ix_1(\xi_1-a\abs{\xi})/\sqrt{1-a^2}}e^{ix'.\xi'}e^{-iat\xi_1/\sqrt{1-a^2}}e^{it\abs{\xi}/\sqrt{1-a^2}}\hat{f}(\xi)d\xi_1
d\xi'\\
=&\int_{\R^d}
e^{ix_1(\xi_1-a\abs{\xi})/\sqrt{1-a^2}}e^{ix'.\xi'}\hat{g}_t(\xi)d\xi_1d\xi',
 \end{align*}
where
$\hat{g}_t(\xi)=e^{-iat\xi_1/\sqrt{1-a^2}}e^{it\abs{\xi}/\sqrt{1-a^2}}\hat{f}(\xi)$.
We now define
$$\eta_1=\frac{\xi_1-a\abs{\xi}}{\sqrt{1-a^2}},\;\;\eta'=\xi'$$
and compute
 \begin{align*}
 \left\lvert \frac{d\eta}{d\xi} \right\rvert =& \det \left(
\begin{array}{ccccc}
\frac{1-a\xi_1/\abs{\xi}}{\sqrt{1-a^2}}
\frac{-a\xi_2/\abs{\xi}}{\sqrt{1-a^2}} & \ldots & \ldots &
\frac{-a\xi_d/\abs{\xi}}{\sqrt{1-a^2}} \\
0 & 1 & 0 & \ldots & 0 \\
0 & 0 & 1 & \ldots & 0 \\
\vdots & \vdots & \vdots & \ddots  & \vdots \\
0 & 0 & 0 & \ldots & 1
\end{array}
\right) \\
=& \left(\frac{1-a\xi_1/\abs{\xi}}{\sqrt{1-a^2}}\right)\approx
1\,\text{ for }\,\abs{a}\leq 1/4.
 \end{align*}
The result now follows from Plancherel's Theorem.

\end{proof}

\begin{remark}\label{r:2.3}
A density argument in fact shows that
$$
t\mapsto w\left(\frac{x_1-at}{\sqrt{1-a^2}},x',\frac{t-ax_1}{\sqrt{1-a^2}} \right)\in
C\left(\R;\dot{H}^1\left(\R^d,dx_1,d\bar{x}'\right) \right),
$$
and similarly for $\pa_tw$.
\end{remark}

 Note that if $u$ is a solution
of (FNLW), with maximal interval $I$ and $I'\subset\subset I, u\in L_{I'}^\frac{d+2}{d-2}L_x^\frac{2(d+2)}{d-2}$,
 and
so $|u|^\frac4{d-2}u\in L_{I'}^1L_x^2$. Thus, the conclusion of the Lemma \ref{lem:trace}
applies, provided the integration is restricted to
$$\Big(\frac{t-ax_1}{\sqrt{1-a^2}},\frac{x_1-at}{\sqrt{1-a^2}},x'\Big)\in I'\times\R^d.$$

\vskip0.15cm

\underline{\bf The proof of Theorem \ref{zeromome}}.

\vskip0.25cm

{\bf Case 1: $T_+(u_{0,c},u_{1,c})<+\infty$}.\; By scaling, we may assume $T_+(u_{0,c},u_{1,c})=1$.
By Lemma \ref{compset}, we have in this situation,
$${\rm supp}(u_c(\cdot,t),\pa_tu_c(\cdot,t))\subset B(0,1-t),~0<t<1,$$
 Assume,
to argue by contradiction, that
$$\int_{\mathbb R^d}\pa_{x_1}(u_{0,c})u_{1,c}dx=\gamma>0.$$
This equality will leads to  a contradiction.
For convenience, set
$$u(x,t)=u_c(x,1+t),~-1<t<0,$$
which satisfies the following  properties
\begin{equation*}
\left\{\begin{aligned}
&{\rm supp}~u(t,\cdot)\subset B(0,|t|), \;\;\;\; \gamma=\int_{\mathbb R^d}\pa_{x_1}u(t)u_t(t)dx,\quad\;t\in [-1,0);\\
&E(u(t),\partial_tu(t))=E_c,\;\; \int_{\mathbb R^d}|\nabla u(t)|^2dx\le (1-\bar\delta)\|\nabla W\|_2^2,\;\;\;t\in [-1,0),
\end{aligned}\right.\end{equation*}
by energy trapping Lemma \ref{lem:enetra}.
Note that
$${\rm supp}~u(t,\cdot)\subset B(0,-t), \;\; {\rm supp}~u_t(t,\cdot)\subset B(0,-t), \quad\; -1\le t<0,$$
For $a\in(0,1/4)$,  we
consider the {\bf Lorentz transformation}
\begin{equation}
z_a(x_1,x', t)=u\Big(\frac{x_1-at}{\sqrt{1-a^2}},x', \frac{t-ax_1}{\sqrt{1-a^2}}\Big),
\end{equation}
where $x=(x_1,x')\in \mathbb R^d, t\in [-1, 0)$ and $s=(t-ax_1)/\sqrt{1-a^2}$ is such that $-1\le s<0$.

 Note that, for this range of $s$ and $y=(y_1,y')$ such that$(y,s)\in {\rm supp} u$, we have$|y|\ge |s|$.
 Thus, if
 $$y_1=\frac{x_1-at}{\sqrt{1-a^2}}, \; y'=x'\;\;\Longrightarrow\;\Big\{(x_1,x')| x_1^2+|x'|^2\le t\Big\}\subset{\rm supp}(z_a,\dot{z}_a).$$
Fix  now {\bf $t\in[-1/2,0)$} and $x_1^2+|x'|^2\le t^2$, we have
$$ 0>\frac{(1-a)t}{\sqrt{1-a^2}}\ge \frac{t-ax_1}{\sqrt{1-a^2}} \ge \frac{(1+a)t}{\sqrt{1-a^2}}\ge -\frac12\frac{1+a^2}{\sqrt{1-a^2}} \ge-1$$
Hence, for such $(x,t)$, $z_a$ is defined and satisfies
$$z_a(x,t)=0,\;\;\nabla z_a(x,t)=0,\;\;\partial_tz_a(x,t)=0,\;\;\;-\frac12\le t<0, \;x_1^2+|x'|^2= t^2.$$
{\bf We extend $z_a(\cdot,t)$ to be zero for $|x|\ge t, -\frac12\le t<0$}.
An elementary calculation shows that if $u$ is  a
regular solution to (NLW) in $\mathbb R^d\times[-1,0)$,
then the resulting $z_d$  is  in the energy space and solves our equation (NLW) in $\mathbb R^d\times[-1/2,0)$.

The above easy calculation shows that
$${{\rm supp} z_a(t,\cdot)\subset B(0, |t|),\quad z_a\not\equiv0}$$
so that $T_+=0$ is the {\bf final time} of existence for $z_a$.
 A lengthy
calculation shows that
{\begin{equation*}
\lim_{a\downarrow0}\frac{E(z_a,\pa_tz_a)(-\frac12,\cdot)-E(u_{0,c},u_{1,c})}{a}=-\int_{\mathbb R^d}\pa_{x_1}(u_{0,c})u_{1,c}dx\triangleq-\gamma<0.\end{equation*}}
and  for small {{\bf a}}  such that
{\begin{equation*}
\left\{\begin{aligned}
& E(z_a,\pa_tz_a)(t_0,\cdot)<E(u_{0,c}, u_{1,c})=E_c,\\
&\int_{\mathbb R^d}|\nabla z_a(t_0)|^2dx<\int_{\mathbb R^d}|\nabla W|^2dx, \end{aligned}\right.
\;\quad{\text{for some}~t_0\in\big[-\frac12,-\frac14\big]}.\end{equation*}}
But, for $a$ small the above facts impliy that  $z_a$ should be global solution!
This {\bf contradicts}  with the fact  the {\bf final time of existence} of $z_a$ is {finite}!
and so {\bf contradicts}  with the {\bf definition of $E_c$} by taking $a>0$ small!

\vskip 0.1in

{\bf Case 2: $T_+(u_{0,c},u_{1,c})=+\infty$.} \;\;
In this case we have ${\la(t)\geq A_0>0}$.  The {\bf finiteness of the
energy} of $z_a$ is {\bf unclear}, because of the lack of the {\bf support property}. We instead
do a {\bf renormalization}. For large $R$, we first rescale $u_c$  as $u_R(x, t)=
R^\frac{d-2}2u_c(Rx,Rt)$, and for small {{$a$}}, define
$${z_{a,R}(t,x_1,x')\triangleq u_R\Big(\frac{t-ax_1}{\sqrt{1-a^2}},\frac{x_1-at}{\sqrt{1-a^2}},x'\Big).}$$
We assume that
$${\int_{\mathbb R^d}\pa_{x_1}(u_{0,c})u_{1,c}dx=\gamma>0}$$
and hope {\bf to obtain a  contradiction}.
We can prove (by integration in $t_0\in(1,2)$) that if
$${h(t_0)=\theta(x)z_{a,R}(t_0,x_1,x'),}$$
with {$\theta$ a fixed cut-off} function, for some $a_1$ small and $R$ large, we have, for some
${t_0\in(1, 2)}$, that
$${E(h(t_0),\pa_th(t_0))<E_c-\frac{a_1\gamma}2\quad~\text{and}\quad ~\int_{\mathbb R^d}|\nabla h(t_0)|^2dx<\int_{\mathbb R^d}|\nabla W|^2dx.}$$

We then let $v$ be the solution of (NLW) with data $h(t_0,\cdot)$, i.e.
\begin{equation*} \left\{ \aligned
    &{v_{tt}-\Delta v =-|v|^{\frac4{d-2}} v },    \;\qquad\; (t,x)\in \R \times \R^d , \\
   &  v(t_0)  = \theta(x)z_{a, R}(x, t_0), \quad\;     v_t(t_0)=  \theta(x)\dot{z}_{a, R}(x,t_0).
\endaligned
\right.
\end{equation*}
By the properties of critical element, we
know that
\begin{equation}{\|v\|_{S(\R)}\leq g\big(\frac{a_1\gamma}2\big)<\infty,}\quad \text{for $R$ large}. \label{fact-1}\end{equation}
 On the other hand, since ${\|u_c\|_{S(0,+\infty)}=+\infty}$,
we have that
\begin{equation}{\|u_R\|_{L_{t,x}^\frac{2(d+1)}{d-2}([0,1]\times\{|x|<1\})}\to+\infty,\quad ~\text{as}\quad ~R\to\infty.}\label{fact-2}\end{equation}
But, by {finite speed of propagation}, we have that $v=z_{a,R}$ on a large set and, after
a change of variables to undo the { Lorentz transformation}, we reach a {contradiction}
from  {\bf \eqref{fact-1} and \eqref{fact-2}}.

\subsection{Rigidity theorem}

To finish Theorem \ref{thm:kmw08}, we are reduced to:

\begin{theorem}[Rigidity Theorem]\label{thm:rigthm}
Assume that
$E(u_0,u_1)<E(W,0)$, and $\|\nabla u_0\|_2<\|\nabla W\|_2$.
Let $u$ be the corresponding solution of (FNLW), and
let $I_+=[0,T_+(u_0,u_1))$. Suppose that:

\vskip0.15cm

{\rm(i)}\;  $\int_{\mathbb R^d} \nabla u_0u_1dx=0.$

\vskip0.15cm
{\rm(ii)}\;  There exist $x(t),~\la(t)$ such that \begin{equation*}
K=\Big\{\vec{v}(x,t)=\Big(\la(t)^{-\frac{d-2}2}u_c\Big(\frac{x-x(t)}{\la(t)},t\Big),
\la(t)^{-\frac{d}2}\pa_tu_c\Big(\frac{x-x(t)}{\la(t)},t\Big)\Big),~t\in I_+\Big\}
\end{equation*}
has {\bf compact closure} in $\dot H^1\times L^2.$
\vskip0.15cm

{\rm(iii)} \; $x(t),~\la(t)$ are continuous, $\la(t)>0.$
\begin{equation*}\left\{\begin{aligned}
&\text{If}\; {T_+(u_0,u_1)<\infty},\;\text{we have}
\; {\la(t)\geq\frac{C}{T_+-t},~{\rm supp}(u,\pa_tu)\subset B(0,T_+-t);}\\
&\text{if}\;  {T_+(u_0,u_1)=\infty},\; \text{we have}
\;{x(0)=0,~\la(0)=1,~\la(t)\geq A_0>0.}\end{aligned}\right.\end{equation*}
 Then  ${T_+<\infty}$ is impossible,  and if ${T_+=\infty},~\Longrightarrow\;u\equiv0$.
\end{theorem}

\noindent\underline{\bf Virial identity}
This Rigidity Theorem provides the contradiction for the proof of Theorem \ref{thm:kmw08}.
For the proof we need some known identities (see
\cite{KM1,SS98}).

\begin{lemma}[Algebraic identity]\label{virial-identity-1} Let
$${r(R)=r(t,R)=\int_{|x|\geq R}\Big(|\nabla u|^2+|\pa_tu|^2+|u|^{2^\ast}+\frac{|u|^2}{|x|^2}\Big)dx.}$$
Let $u$ be a solution of {\rm (FNLW)}, {$t\in I,~\phi_R(x)=\phi(x/R),$ $\psi_R(x)=x\phi(x/R)$}, where {$\phi\in C_0^\infty(B_2),~\phi\equiv1$ on $|x|\leq1.$} Then:
\vskip 0.1in
{\rm(i)} {$\pa_t\big(\int_{\mathbb R^d}\psi_R\nabla u\pa_tudx\big)
=-\frac{d}2\int_{\mathbb R^d}|\pa_tu|^2dx+\frac{d-2}2\int_{\mathbb R^d}\big(|\nabla u|^2-|u|^{2^\ast}\big)dx+O(r(R)).$}
\vskip0.1in
{\rm(ii)}{$\pa_t\big(\int_{\mathbb R^d}\phi_R u\pa_tudx\big)=\int_{\mathbb R^d}|\pa_tu|^2dx
-\int_{\mathbb R^d}|\nabla u|^2dx+\int_{\mathbb R^d}|u|^{2^\ast}dx+O(r(R)).$}
\vskip0.1in
{\rm(iii)}{$\pa_t\big(\int_{\mathbb R^d}\psi_R\big[\frac12|\nabla u|^2
+\frac12|\pa_tu|^2-\frac1{2^\ast}|u|^{2^\ast}\big]dx\big)=-\int_{\mathbb R^d}\nabla u\pa_tudx+O(r(R)).$}
\end{lemma}

\vskip0.12cm
\noindent\underline{\bf The Proof of Rigidity theorem \ref{thm:rigthm}}

\vskip0.2cm
\noindent{\bf Case 1: $T_+(u_0,u_1)=+\infty$}

\vskip0.2cm
We start out the proof of case 1, without loss of generality, we may assume $x(t)$, $\lambda(t)$ continuous
such that
$$\lambda(t)\ge A_0>0, \;\; x(0)=0,\;\; \lambda(0)=1.$$
Assume that $(u_0,u_1)\neq(0,0)$ such that
\begin{equation}\label{nlw-rig1.1} ~E=E(u_0,u_1)\le (1-\delta_0)E(W,0), \quad\;\|\nabla u_0\|_2<\|\nabla W\|_2,\end{equation}
then, from {variational estimates},
\begin{equation}\label{nlw-rig1.2}
{\sup_{t>0}\|(\nabla u(t),\partial_tu(t))\|_{\dot H^1\times L^2}\leq CE, \quad E>0.}
\end{equation}
We also have
\begin{equation}\label{nlw-rig1.3}\left\{\begin{aligned}
&{\int_{\mathbb R^d}\big(|\nabla u(t)|^2-|u(t)|^{2^\ast}\big)dx\geq C_{\bar\delta}\int_{\mathbb R^d}|\nabla u(t)|^2dx,\quad \bar\delta=\bar\delta(\delta_0),\;~t>0}\\
&{\frac12\int_{\mathbb R^d}|\pa_tu(t)|^2dx+\frac12\int_{\mathbb R^d}\big(|\nabla u(t)|^2-|u(t)|^{2^\ast}\big)dx\geq C_{\bar\delta}E,\;\;\quad \;~t>0.}\end{aligned}\right.
\end{equation}
The {compactness of $\bar K$} and the fact that ${\la(t)\geq A_0>0}$ show that, given $\epsilon>0,$
we can find $R_0(\epsilon)>0$ so that, for all $t>0$, we have
\begin{equation}\label{nlw-rig1.4}
{\int_{\big|x+\frac{x(t)}{\la(t)}\big|\geq R_0(\epsilon)}\Big(|\nabla u|^2
+|\pa_tu|^2+|u|^{2^\ast}+\frac{|u|^2}{|x|^2}\Big)dx\leq \epsilon E.}
\end{equation}

The proof of this case is accomplished through two lemmas.

\begin{lemma}[Control-lemma]\label{spatcon} There exist $\epsilon_1,~C>0$ such that,
if $\epsilon\in(0,\epsilon_1),\exists R_0(\varepsilon)>0$, for any $~R>2R_0(\epsilon)$,there exists {$t_0=t_0(R,\epsilon)$}
with $0<t_0\leq CR,$ such that for $0<t<t_0$, we have
$${\frac{|x(t)|}{\la(t)}<R-R_0(\epsilon)~~\text{\rm and}~~\frac{|x(t_0)|}{\la(t_0)}=R-R_0(\epsilon)> R_0(\epsilon).}$$
\end{lemma}
Note that in the radial case, since we can take {$x(t)\equiv0$}, a {contradiction}
follows directly from the above Lemma. This will be the analog of the {local virial identity}
proof for the corresponding case of (NLS).

\qquad For the {non-radial case} we also need:
\begin{lemma}[$t_0$ control]\label{timecont}
There exists {$\epsilon_2\in(0,\epsilon_1),~R_1(\epsilon)>0,~C_0>0$} so that if ${R>R_1(\epsilon)}$, for ${0<\epsilon<\epsilon_2}$, we have
$${t_0(R,\epsilon)\geq C_0R/\epsilon.}$$
\end{lemma}

From the above two Lemmas, we have, for
$${\epsilon\in(0,\epsilon_1),\quad\Longrightarrow\; ~R>2R_0(\epsilon),\quad ~t_0(R,\epsilon)\leq CR,}$$
while for
$${\epsilon\in(0,\epsilon_2),\quad\Longrightarrow\; ~R>R_1(\epsilon),\quad ~t_0(R,\epsilon)\geq\frac{C_0R}{\epsilon}.}$$
Thus, for ${\epsilon\in(0,\epsilon_2),~R>\max\{R_1(\epsilon),2R_0(\epsilon)\}}$, we have
$${\frac{C_0R}{\epsilon}\leq t_0(R,\epsilon)\leq CR.}$$
This is
clearly a {contradiction} for {$\epsilon$ small}. Therefore,{$(u_0,u_1)\equiv0$} in the case {$T_+(u_0,u_1)=\infty.$}
\vskip 0.1in
\qquad Thus, it suffices to show the above two Lemmas.

\begin{proof}[{\bf Proof of Lemma \ref{spatcon}:}]
If not, since $x(0)=0,~\la(0)=1$,  and  both $\la(t),~x(t)$ continuous, we have
$${\forall~t\in(0,CR)\quad ~\text{with}~C~\text{large},\quad ~\frac{|x(t)|}{\la(t)}<R-R_0(\epsilon).}$$
Let
$${z_R(t)=\int_{\mathbb R^d}\Big[\psi_R\nabla u\pa_tu+\frac{d-1}2\phi_Ru\pa_tu\Big]dx
=\int_{\mathbb R^d}\Big[\phi_R|x|\Big(\frac{x}{|x|}\cdot \nabla+\frac{d-1}{2|x|}\Big)u\pa_tu\Big]dx.}$$
Then
$${z_R'(t)=-\frac12\int_{\mathbb R^d}|\pa_tu|^2dx-\frac12\int_{\mathbb R^d}\big(|\nabla u|^2-|u|^{2^\ast}\big)dx+O(r(R)).}$$
But, for $|x|>R,~t\in(0,CR)$, we have by \eqref{nlw-rig1.4} and  Lemma \ref{lem:enetra}
$${\Big|x+\frac{x(t)}{\la(t)}\Big|\geq R_0(\epsilon)}, \;\Longrightarrow\; {|r(R)|\leq\epsilon E.}
\;\;\Longrightarrow {z_R'(t)\leq-\tilde{C}E/2}, \;\;\text{\bf for $\epsilon$ small}.$$
 {On the other hand}, we  have
${|z_R(t)|\leq C_1RE.}$
Integrating in $t$ we obtain $${\frac12CR\tilde{C}
E\leq z_R(0)-z_R(CR)\leq2C_1RE,\Longrightarrow~}\text{a {contradiction} for $C$ large.}$$
\end{proof}

\begin{proof}[{\bf Proof of Lemma \ref{timecont}:}]

For $t\in[0,t_0],$ set
$${y_R(t)=\int_{\mathbb R^d}\psi_R\Big(\frac12|\pa_tu|^2+\frac12|\nabla u|^2-\frac1{2^\ast}|u|^{2^\ast}\Big)dx,
\quad \psi_R(x)=x\phi\big(\frac{x}R\big).}$$
By Lemma \ref{spatcon}, we have
$${\Big|x+\frac{x(t)}{\la(t)}\Big|\ge |x|-(R-R_0(\epsilon))\ge R_0(\epsilon),\quad ~\forall~|x|>R.}$$
Conservation of momentum and (iii) in Lemma \ref{virial-identity-1} imply that
$${\int_{\mathbb R^d} \nabla u_0u_1dx=0=\int_{\mathbb R^d}\nabla u(t)\pa_tu(t)dx\;\; ~\text{and}~y_R'(t)=O(r(R)),}$$
and hence
\begin{equation}\label{nlw-rig1.5}
 r(R)\le E\epsilon, \;\Longrightarrow\;{|y_R(t_0)-y_R(0)|\leq\tilde C\epsilon Et_0.}
\end{equation}
However, noting that $x(0)=0$ and compactness \eqref{nlw-rig1.4}, we have
\begin{equation}\label{nlw-rig1.6}
{|y_R(0)|\leq\tilde CR_0(\epsilon)E+O(Rr(R_0(\epsilon))\leq\tilde CE(R_0(\epsilon)+\epsilon R).}
\end{equation}
Also
\begin{align*}
{|y_R(t_0)|}\geq&{\Big|\int_{\big|x+\frac{x(t_0)}{\la(t_0)}\big|\leq R_0(\epsilon)}\psi_R\big(\frac12|\pa_tu|^2
+\frac12|\nabla u|^2-\frac1{2^\ast}|u|^{2^\ast}\big)dx\Big|}\\
&{-\Big|\int_{\big|x+\frac{x(t_0)}{\la(t_0)}\big|> R_0(\epsilon)}\psi_R\big(\frac12|\pa_tu|^2+\frac12|\nabla u|^2-\frac1{2^\ast}|u|^{2^\ast}\big)dx\Big|.}
\end{align*}
In the  {first} integral, {$|x|\leq R$} implies that {$\psi_R(x)=x$}.
The {second integral} is bounded
by {$2R\epsilon E$}. Thus,
$$|y_R(t_0)|\geq{\Big|\int_{|x+\frac{x(t_0)}{\la(t_0)}|\leq R_0(\epsilon)}x
\Big(\frac12|\pa_tu(t_0)|^2+\frac12|\nabla u(t_0)|^2-\frac1{2^\ast}|u(t_0)|^{2^\ast}\Big)dx\Big|}-
{2R\epsilon E}.$$
The integral on the right equals
\begin{align*}
&-{\frac{x(t_0)}{\la(t_0)}\int_{|x+\frac{x(t_0)}{\la(t_0)}|\leq R_0(\epsilon)}
\Big(\frac12|\pa_tu(t_0)|^2+\frac12|\nabla u|^2-\frac1{2^\ast}|u(t_0)|^{2^\ast}\Big)dx}\\
+&{\int_{|x+\frac{x(t_0)}{\la(t_0)}|\leq R_0(\epsilon)}\Big(x+\frac{x(t_0)}{\la(t_0)}\Big)
\cdot\Big(\frac12|\pa_tu(t_0)|^2+\frac12|\nabla u(t_0)|^2-\frac1{2^\ast}|u(t_0)|^{2^\ast}\Big)dx},
\end{align*}
so that its absolute value is greater than or equal to
$${(R-R_0(\epsilon))E-\tilde C(R-R_0(\epsilon))\epsilon E}-{\tilde CR_0(\epsilon)E.}$$
Thus,
$${{|y_R(t_0)|}\geq {E(R-R_0(\epsilon))(1-\tilde C\epsilon)}-}{\tilde{C}R_0(\epsilon)E-2R\epsilon E}\geq {ER/4},$$
for $R$ large, $\epsilon$ small. But then, this together with \eqref{nlw-rig1.5} and \eqref{nlw-rig1.6} shows
$${ER/4-\tilde CE[R_0(\epsilon)+\epsilon R]\leq|y_R(t_0)|-|y_R(0)|\leq|y_R(t_0)-y_R(0)|\leq\tilde C\epsilon Et_0,}$$
 which yields the
Lemma \ref{timecont} for $\epsilon$ small, $R$ large.
\end{proof}

\noindent\underline{\bf Case 2: $T_+(u_0,u_1)<\infty$}

\vskip0.15cm

We may assume $T_+(u_0,u_1)=1$. Then,
$${\rm supp}(u,\pa_tu)(t,\cdot)\subset B(0,1-t)~\text{and}~\la(t)\geq\frac{C}{1-t}.$$
For (FNLW) we cannot use the conservation of the $L^2$-norm as in the
(NLS) case and a new approach is needed. The first step is:
\begin{lemma}[Upper bound of $\la(t)$] Let $u$ be as in the Rigidity Theorem \ref{thm:rigthm}, with $T_+(u_0, u_1)=1$. Then there
exists $C>0$ so that
\begin{equation}
\la(t)\leq\frac{C}{1-t}.
\end{equation}
\end{lemma}
{\bf {Remark:}} Thus, in this case
$${\la(t)\simeq\frac{1}{1-t}.}$$

\begin{proof}
If not, we can find $t_n\nearrow1$ so that $\la(t_n)(1-t_n)\to+\infty$.

{\bf Step 1:} We claim that
\begin{equation}\label{nlw-rig1.7}
z(t)\geq CE\cdot (1-t),\;\;~t\in[0,1),
\end{equation}
where
$$z(t)=\int x\cdot\nabla u\pa_tu+\frac{d-1}2\int u\pa_tu=
\int |x|\Big(\frac{x}{|x|}\cdot\nabla+\frac{d-1}{2|x|}\Big)u\pa_tu.$$
We recall that $z(t)$ is well-defined since
$${\rm supp}(u,\pa_tu)(t,\cdot)\subset B(0,1-t).$$
Then, for $t\in(0,1),$  by variational estimate\eqref{nlw-rig1.3} and $E(u_0,u_1)=E>0$, we have
\begin{equation}\label{nlw-rig1.8}\left\{\begin{aligned}
&{z'(t)=-\frac12\int_{\mathbb R^d}|\pa_tu|^2dx-\frac12\int_{\mathbb R^d}\big(|\nabla u|^2-|u|^{2^\ast}\big)dx\leq -CE,}\\
&{\sup_{t\in(0,1)}\|(u(t),\pa_tu)\|_{\dot H^1\times L^2}\leq CE.}\end{aligned}\right.\end{equation}
From the support properties of $u$, it is easy to
see that $${|z(t)|\leq(1-t)\|(u(t),\pa_tu(t))\|_{\dot H^1\times L^2}\longrightarrow0,\quad ~\text{as}~t\longrightarrow1,}$$
so that, this together with \eqref{nlw-rig1.8} and integrating in $t$ gives
$$z(t)\geq CE(1-t),~t\in[0,1).$$

{\bf Step 2:} We will next show that
$${\frac{z(t_n)}{1-t_n}\to0\;\;~\text{as}\;\;~n\to\infty},$$
this yields a contradiction with the first step.

Since
$$\int_{\mathbb R^d} \nabla u(t)\pa_tu(t)dx=0, $$
we have
$${\frac{z(t_n)}{1-t_n}=\int_{\mathbb R^d}\Big(x+\frac{x(t_n)}{\la(t_n)}\Big)\cdot\frac{\nabla u(t_n)\pa_tu(t_n)}{1-t_n}dx
+\frac{d-1}2\int_{\mathbb R^d}\frac{u(t_n)\pa_tu(t_n)}{1-t_n}dx.}$$
Note that, for $\epsilon>0$ given, we have by compactness theorem
\begin{align*}
&\int_{|x+\frac{x(t_n)}{\la(t_n)}|\leq\epsilon(1-t_n)}\Big[\Big|x+\frac{x(t_n)}{\la(t_n)}\Big|\cdot|\nabla u(t_n)|\cdot|\pa_tu(t_n)|
+|u(t_n)|\cdot|\pa_tu(t_n)|\Big]dx\\
\leq& C\epsilon E(1-t_n).\end{align*}

Now we consider the outsider region. First we will show that
\begin{equation}\label{nlw-rig1.9}
{\frac{|x(t_n)|}{\la(t_n)}\leq2(1-t_n).}
\end{equation}
If not, ${B\big(-\frac{x(t_n)}{\la(t_n)},(1-t_n)\big)\cap B(0,(1-t_n))=\emptyset}$~$\Longrightarrow$
\begin{equation}\label{nlw-rig1.9-add}
~{\int_{B(-x(t_n)/\la(t_n),(1-t_n))}\big(|\nabla u(t_n)|^2+|\pa_tu(t_n)|^2\big)dx\equiv0,}\end{equation}
while from the definition of compactness
\begin{align*}
&{\int_{|x+\frac{x(t_n)}{\la(t_n)}|\geq(1-t_n)}\big(|\nabla u(t_n)|^2+|\pa_tu(t_n)|^2\big)dx}\\
=&\int_{|y|\geq\la(t_n)(1-t_n)}\Big|\nabla_{t,x} u\Big(\frac{y-x(t_n)}{\la(t_n)},t_n\Big)\Big|^2\frac{dy}{\la(t_n)^d}
{\longrightarrow0},~\text{as}~n\longrightarrow\infty.
\end{align*}
This together with \eqref{nlw-rig1.9-add} {contradicts $E>0$.} Then, \eqref{nlw-rig1.9} and the Hardy inequality imply that
\begin{align*}
&\frac1{1-t_n}\int_{|x+\frac{x(t_n)}{\la(t_n)}|\geq\epsilon(1-t_n)}\Big|x+\frac{x(t_n)}{\la(t_n)}
\Big|\cdot\Big[|\nabla u(t_n)|\cdot|\pa_tu(t_n)|\\
&\qquad\qquad\qquad\qquad\qquad\quad+\Big|x+\frac{x(t_n)}{\la(t_n)}\Big|^{-1}|u(t_n)|\cdot|\pa_tu(t_n)|\Big|\Big]dx\\
\leq&3\int_{|x+\frac{x(t_n)}{\la(t_n)}|\geq\epsilon(1-t_n)}\Big||\nabla u(t_n)|\cdot|\pa_tu(t_n)|+\Big|x+\frac{x(t_n)}{\la(t_n)}\Big|^{-1}|u(t_n)|\cdot|\pa_tu(t_n)|\Big|dx\\
=&{(3+C)\int_{|y|\geq\epsilon(1-t_n)\la(t_n)}\big(|\nabla u|^2+|\pa_tu|^2\big)\Big(\frac{y-x(t_n)}{\la(t_n)},t_n\Big)\frac{dy}{\la(t_n)^d}}
\to0
\end{align*}
by the compactness of $\bar K$ and the fact that $\la(t_n)(1-t_n)\to\infty$.
And using {Hardy's inequality} (centered at $-x(t_n)/\la(t_n)$), the proof
is {concluded.}
\end{proof}

\begin{prop}[Self-similar solution]  Let $u$ be as in the Rigidity Theorem \ref{thm:rigthm}, with $T_+(u_0, u_1)=1$,
$${\rm supp}(u,\pa_tu)(t,\cdot)\subset B(0,1-t).$$
Then
$$K=\big\{(1-t)^\frac{d-2}2u((1-t)x,t),(1-t)^\frac{d}2\pa_tu((1-t)x,t\in[0,1))\big\}$$
is precompact in $\dot H^1\times L^2.$
\end{prop}

\begin{proof}

Note that the fact that if $\overline{K}_1$ is compact in $L^2(\R^d)^{d+1}$, then
$$K_2:=\{\la^\frac{d}2\vec{v}(\la x):~\vec{v}\in K_1,~c_0\leq\la\leq c_1\}$$
also has $\overline{K}_2$ compact. This together with the condition (iii) of  Theorem  \ref{thm:rigthm}, combining this fact with
$c_0\leq (1-t)\la(t)\leq c_1$,
we get that the set
\begin{equation}\label{nlw-rig1.10}
\big\{\vec{v}(x,t)=(1-t)^\frac{d}2(\nabla u,\pa_tu)((1-t)(x-x(t)),t),~t\in[0,1)\big\}
\end{equation}
has compact closure in  $L^2(\R^d)^{d+1}$.

 Let now
$$\tilde v(x,t)=(1-t)^\frac{d}2(\nabla u,\pa_tu)((1-t)x,t),$$
so that $\tilde v(x,t)=\vec{v}(x+x(t),t)$ satisfies
$${\rm supp}\vec{v}(\cdot,t)\subset\{x:~|x-x(t)|\leq1\}.$$
The fact that $E>0$, the compactness of $\{\vec{v}(\cdot,t)\}$ and preservation of energy
imply that $|x(t)|\leq M$. This together with \eqref{nlw-rig1.10} and the fact that if
$$K_3=\{\vec{v}(x+x_0,t):~|x_0|\leq M\},$$
then $\overline{K_3}$ is compact, gives the Proposition.
\end{proof}

\vskip0.2cm

\noindent\underline{\bf Self-similar variables and the end of proof of rigidity theorem}

\vskip0.2cm
Since the lack of the $L^2$ conservation law,
we cannot use the conservation of the $L^2$-norm as in the
(NLS). At this point we introduce a new idea, inspired by the works of Giga-Kohn\cite{GK89} in the parabolic case and
Merle-Zaag\cite{MZ03} in the hyperbolic case, who studied the equations
$$(\pa_t^2-\Delta)u=|u|^{p-1}u,~1<p<1+\tfrac4{d-2},$$
in the radial case. In our case, {$p=1+\frac4{d-2}>1+\frac4{d-1}$}.
We thus introduce {self-similar variables} as following
$${y=\frac{x}{1-t},\quad ~s=\log\frac1{1-t},}\;\;\; \Longrightarrow\;\;\; x=(1-t)y=e^{-s}y,\; t=1-e^{-s}$$
and define
$${w(y,s;0)=(1-t)^\frac{d-2}2u(x,t)=e^{-\frac{(d-2)s}2}u(e^{-s}y,1-e^{-s}),}$$
which is defined for $0\leq s<\infty$ with
$${{\rm supp}w(\cdot,s;0)\subset\big\{y:\; e^{-s}|y|\leq e^{-s}\big\}=\big\{y:\; |y|\leq1\big\}.}$$
\vskip0.2cm

We will also
consider, for $\delta>0$, $u_\delta(x,t)=u(x,t+\delta)$ which also solves (FNLW) and its corresponding
$w$, which we will denote by $w(y,s;\delta)$.
Thus, we set
\begin{equation*}\left\{\begin{aligned}
&{y=\frac{x}{1+\delta-t},\quad\;~s=\log\frac1{1+\delta-t},}\\
&{w(y,s;\delta)=(1+\delta-t)^\frac{d-2}2u(x,t)=e^{-\frac{(d-2)s}2}u(e^{-s}y,1+\delta-e^{-s}).}
\end{aligned}\right.\end{equation*}
Here $w(y,s;\delta)$ is defined for $0\leq s<-\log\delta$
and satisfies
$${{\rm supp}~w(\cdot,s;\delta)\subset\Big\{|y|\leq\frac{e^{-s}-\delta}{e^{-s}}=\frac{1-t}{1+\delta-t}\leq1-\delta\Big\}.}$$
Hence $w(y,s;\delta)$ solves the equation in their domain
\begin{align*}
{\pa_s^2w=}&{\frac1\rho{\rm div}\big(\rho\nabla w-\rho(y\cdot\nabla w)y\big)-\frac{d(d-2)}4w+|w|^{\frac4{d-2}}w}\\
&{-2y\cdot\nabla\pa_sw-(d-1)\pa_sw,}\quad \text{where}\; \; {\rho(y)=(1-|y|^2)^{-1/2}.}
\end{align*}
Note that the elliptic part of this operator degenerates. In fact,
$${\frac1\rho{\rm div}\big(\rho\nabla w-\rho(y\cdot\nabla w)y\big)=\frac1\rho{\rm div}\big(\rho(I-y\otimes y)\nabla w\big),}$$
which is {elliptic} with smooth coefficients for {$|y|<1$}, but {degenerates} at {$|y|=1$}.

This new equation gives us a new set of formulas.
{ \bf The reason for introducing $\delta>0$ is that, on the supp $w(\cdot,s,\delta)$, the inequality $(1-|y|^2)\ge\delta$ holds,
 so we stay away from the degeneracy}.
  Now we  collect some straightforward bounds on $w(\cdot, \cdot, \delta)$ with $\delta>0$: \; $w\in H_0^1(B_1)$ with
\begin{equation}\label{self-simil1.1}
{\int_{B_1}\big(|\nabla w|^2+|\pa_sw|^2+|w|^{2^\ast}\big)dx\leq C.}
\end{equation}
Moreover, by Hardy's inequality for $H_0^1(B_1)$-functions, we have
\begin{equation}\label{self-simil1.2}
{\int_{B_1}\frac{|w(y)|^2}{(1-|y|^2)^2}dy\leq C, \quad\;  \text{uniformly in} \;\; \delta>0,\;~0<s<-\log\delta.}
\end{equation}
 Next, following \cite{MZ03}, we
introduce an energy, which will provide us with a Lyapunov functional for $w$.
\begin{align*}
{\tilde E(w(s;\delta))=}&{\frac12\int_{B_1}\big(|\nabla w|^2+|\pa_sw|^2-(y\cdot\nabla w)^2\big)\frac{dy}{(1-|y|^2)^\frac12}}\\
&{+\int_{B_1}\Big(\frac{d(d-2)}8w^2-\frac{d-2}{2d}|w|^{2^\ast}\Big)\frac{dy}{(1-|y|^2)^\frac12}<\infty, \quad \forall \;\delta>0.}
\end{align*}

For $\delta>0$, we have the following new formulas:
\begin{lemma} \label{lyapunov-1}  Let ~$0<s_1<s_2<\log1/\delta$,  we have
{\rm(i)}
$${\tilde E(w(s_2))-\tilde E(w(s_1))=\int_{s_1}^{s_2}\int_{B_1}\frac{|\pa_sw|^2}{(1-|y|^2)^{3/2}}dyds,}
\quad \Longrightarrow\;\; \tilde E \;  \nearrow. $$
{\rm(ii)} \begin{align*}
&{\frac12\int_{B_1}\big(\pa_sw)\cdot w-\frac{d+1}2w^2\big)\frac{dy}{(1-|y|^2)^{1/2}}\Big|_{s_1}^{s_2}}\\
=&{-\int_{s_1}^{s_2}\tilde E(w(s))ds+\frac1d\int_{s_1}^{s_2}\int_{B_1}\frac{|w|^{2^\ast}}{(1-|y|^2)^{1/2}}dsdy}\\
&{+\int_{s_1}^{s_2}\int_{B_1}\Big(|\pa_sw|^2+\pa_sw\cdot (y\cdot\nabla) w+\frac{\pa_sw\cdot w|y|^2}{1-|y|^2}\Big)\frac{dy}{(1-|y|^2)^{1/2}}.}
\end{align*}
{\rm(iii)} ${\lim\limits_{s\to\log1/\delta}\tilde E(w(s))=E(u_0,u_1)=E,}$ so that, by {\rm (i)},
$$-\frac{C}{\delta^{1/2}}\le {\tilde E(w(s))\leq E\;\;\; ~\text{for}\;\;\; ~0\leq s<\log1/\delta.}$$
\end{lemma}

\vskip0.2cm

\noindent\underline{\bf  The first improvement }
\vskip0.2cm

By \eqref{self-simil1.1} and the support property of $w$, we have
$$\int_0^1\int_{B_1}\frac{|\pa_sw|^2}{1-|y|^2}dyds\leq\frac{C}\delta.$$
Inspired by the above Lemma (i), we obtain:

\begin{lemma}[First improvement] \label{lyapunov-2}
$${\int_0^1\int_{B_1}\frac{|\pa_sw|^2}{1-|y|^2}dyds\leq C\log1/\delta.}$$
\end{lemma}
{\bf {Proof:}} Notice that
\begin{align*}
{-2\int_{\mathbb R^d}\frac{|\pa_sw|^2}{1-|y|^2}dy}
=&{\frac{d}{ds}\Big\{\int_{\mathbb R^d}\Big[\frac12|\pa_sw|^2+\frac12\big(|\nabla w|^2-(y\cdot\nabla w)^2\big)+\frac{d(d-2)}8w^2}\\
&\qquad\qquad\qquad\;\qquad{-\frac{d-2}{2d}|w|^{2^\ast}\Big]\cdot(-\log(1-|y|^2))dy\Big\}}\\
&+\int_{\mathbb R^d}[\log(1-|y|^2)+2]y\cdot\nabla w\pa_sw-\log(1-|y|^2)(\pa_sw)^2dy\\
&-2\int_{\mathbb R^d}|\pa_sw|^2dy.
\end{align*}
We next integrate in $s$, between $0$ and $1$, and drop the term next to last term by
sign. The proof is finished by using {Cauchy-Schwartz} and the {support property} of
$w(\cdot,\delta)$.

\begin{corollary}\label{cor-3.18}
$${\rm (a).}\qquad  {\int_0^1\int_{B_1}\frac{|w|^{2^*}}{(1-|y|^2)^{1/2}}dyds\leq C(\log1/\delta)^\frac12;}$$
$${\rm(b).}\qquad {\tilde E(w(1))\geq-C(\log1/\delta)^\frac12.}\quad\qquad \qquad\quad $$

\end{corollary}
\begin{proof}
{Part (a)} follows from  Lemma \ref{lyapunov-1}(ii), (iii), {Cauchy-Schwartz} and
Lemma \ref{lyapunov-2}. Note that we obtain the power ${1/2}$ on the right hand side by Cauchy-
Schwartz.

\vskip0.15cm

 To prove {Part (b)}, we first consider $\int_0^1\tilde{E}(w(s))ds$, which is bounded from below
 by  $-C(\log1/\delta)^\frac12$ by Part (a). The monotonicity of $\tilde{E}(s)$ in  Lemma \ref{lyapunov-1}(i)  shows
$$\tilde{E}(w(1))\ge{\int_0^1\tilde E(w(s))ds\geq-C(\log1/\delta)^\frac12.}$$
\end{proof}

\vskip0.2cm

\noindent\underline{\bf  The second improvement }
\vskip0.2cm

\begin{lemma}[Second improvement] \label{lyapunov-3}
$${\int_1^{\log1/\delta}\int_{B_1}\frac{|\pa_sw|^2}{(1-|y|^2)^{3/2}}dyds\leq C(\log1/\delta)^\frac12.}$$
\end{lemma}
{\bf{Proof:}} Using Lemma \ref{lyapunov-1} (i), (iii) and  Corollary \ref{cor-3.18} (b), we have
$${\int_1^{(\log1/\delta)^{3/4}}\int_{B_1}\frac{|\pa_sw|^2}{(1-|y|^2)^{3/2}}dyds
=\tilde{E}\Big(w\Big((\log1/\delta)^{3/4}\Big)\Big)-\tilde{E}(w(1))\le E+C(\log1/\delta)^{\frac12}.}$$

\begin{corollary}\label{cor-3.20}
There exists $\bar{s}_\delta\in(1,(\log1/\delta)^{3/4})$ such that
$${\int_{\bar{s}_\delta}^{\bar s_\delta+(\log1/\delta)^{1/8}}\int_{B_1}\frac{|\pa_sw|^2}{(1-|y|^2)^{3/2}}dyds\leq\frac{C}{(\log1/\delta)^{1/8}}.}$$
\end{corollary}
{\bf {Proof:}} Split ${(1,(\log1/\delta)^{3/4})}\subset(1,\log1/\delta)$ into  {$(\log1/\delta)^{5/8}$} disjoint
intervals of length{$(\log1/\delta)^{1/8}$}. Then Corollary follows from {Lemma \ref{lyapunov-3}}
 and {pigeonhole principle.}

\begin{remark}\;In Corollary \ref{cor-3.20}, the length of the $s$ interval tends to infinity,
while the bound goes to zero.\end{remark}

 It is easy to see that if ${\bar{s}_\delta\in(1,(\log1/\delta)^{3/4})}$ and
${\bar{s}_\delta=-\log(1+\delta-\bar t_\delta)}$, then
$${\Big|\frac{1-\bar t_\delta}{1+\delta-\bar t_\delta}-1\Big|\leq C\delta^\frac14\longrightarrow0,\; \;\;\text{as}\;\;\;
\delta\longrightarrow0.}$$
From this and the compactness of $\bar K$, we can find $\delta_j\to0$, so that
$$w(y,\bar s_{\delta_j}+s;\delta_j)\longrightarrow {w^\ast(y,s)\in C([0,S];\dot H_0^1\times L^2)},\;\;\text{for}
\;\;s\in[0,S],$$
 and $w^\ast$ solves our self-similar equation in $B_1\times[0,S]$ as
  \begin{align*}
{\pa_s^2w=}&{\frac1\rho{\rm div}\big(\rho\nabla w-\rho(y\cdot\nabla w)y\big)-\frac{d(d-2)}4w+|w|^{\frac4{d-2}}w}\\
&{-2y\cdot\nabla\pa_sw-(d-1)\pa_sw,}\quad \text{where}\; \; {\rho(y)=(1-|y|^2)^{-1/2}.}
\end{align*}
   The  Corollary
\ref{cor-3.20} shows
that $w$ must be independent of $s$. Also,  $E>0$ and our coercivity
estimates show that ${w^\ast\not\equiv0}$. Thus, ${w^\ast\in H_0^1(B_1)}$ solves
the {degenerate elliptic equation}:

$${\frac1\rho{\rm div}\big(\rho\nabla w^\ast-\rho(y\cdot\nabla w^\ast)y\big)-\frac{d(d-2)}4w^\ast
+|w^\ast|^\frac4{d-2}w^\ast=0.}$$

\vskip0.2cm

We next point out that $w^\ast$ satisfies the additional (crucial) estimates:
$${\int_{B_1}\frac{|w^*|^{2^\ast}}{(1-|y|^2)^{1/2}}dy+\int_{B_1}\frac{|\nabla w^\ast|^2-(y\cdot\nabla w^\ast)^2}{(1-|y|^2)^{1/2}}dy<\infty}.$$
Indeed, for the first estimate it suffices to show that, uniformly for  $j$ large, we have
$${\int_{\bar{s}_{\delta_j}}^{\bar s_{\delta_j}+\delta}\int_{B_1}\frac{|w(y,s;\delta_j)|^{2^\ast}}{(1-|y|^2)^{1/2}}dyds\leq C.}$$
But this  follows from (ii) above, together with the choice of $\bar s_{\delta_j}$, by the Corollary \ref{cor-3.20},
Cauchy-Schwartz and (iii).

\vskip0.2cm

 The proof of the second estimate follows from the first
one, (iii) and the formula for $\tilde E$.

\begin{remark}\;The conclusion of the proof is obtained by showing that a {$w^\ast$ in $H^1_
0(B_1)$,}
solving the {degenerate elliptic equation} with the additional bounds above, must
be zero. This will follow from a {unique continuation argument}.
\end{remark}
\vskip0.15cm

 Recall that, for
${|y|\leq1-\eta_0,~\eta_0>0}$, the linear operator is {uniformly elliptic}, with smooth coefficients
and that the non-linearity is critical.

\vskip0.15cm
 An argument going back to Trudinger's \cite{Tr68} shows that
$w^\ast$ is bounded on $\{|y|\leq1-\eta_0\}$ for each $\eta_0>0$. Thus, if we show that {$w^\ast\equiv0$
near $|y|=1$}, the standard {Carleman unique continuation} principle \cite{Ho85} will show
that {$w^\ast\equiv0$.}

\vskip0.15cm
Near $|y|=1$, our equation is modeled (in variables $(z,r)\in\R^{d-1}\times\R,~r>0$, near $r=0$) by
$${r^\frac12\pa_r(r^\frac12\pa_rw^\ast)+\Delta_zw^\ast+cw^\ast+|w^\ast|^\frac4{d-2}w^\ast=0.}$$

Our information on $w$ translates into $w^\ast\in H^1_0((0, 1]\times(|z|<1))$ and our crucial
additional estimates are:
$${\int_0^1\int_{|z|<1}|w^\ast(r,z)|^{2^\ast}\frac{dr}{r^{1/2}}dz+\int_0^1\int_{|z|<1}|\nabla_zw^\ast(r,z)|^2\frac{dr}{r^{1/2}}dz<\infty.}$$
\vskip0.15cm

We take advantage of the degeneracy of the equation. We ``{desingularize}"
the problem by letting $r=a^2,$ setting $v(a,z)=w^\ast(a^2,z)$, so that
$${\pa_av(a,z)=2a\pa_rw^\ast(r,z)=2r^\frac12\pa_rw^\ast(r,z).}$$
Our equation becomes:
$${\pa_a^2v+\Delta_zv+cv+|v|^\frac4{d-2}v=0,~0<a<1,~|z|<1,~v\big|_{a=0}=0,}$$
and the bounds give:
\begin{equation*}\left\{\begin{aligned}
&{\int_0^1\int_{|z|<1}|\nabla_zv(a,z)|^2dadz=\int_0^1\int_{|z|<1}|\nabla_zw^\ast(r,z)|^2\frac{dr}{r^{1/2}}dz<\infty,}\\
&{\int_0^1\int_{|z|<1}|\pa_av(a,z)|^2\frac{da}{a}dz=\int_0^1\int_{|z|<1}|\pa_rw^\ast(r,z)|^2drdz<\infty.}\end{aligned}\right.
\end{equation*}
Thus, ${v\in H_0^1((0,1]\times B_1)}$, but from the additional bound we see that
$${\pa_av(a,z)|_{a=0}\equiv0.}$$

\vskip0.15cm

We then extend $v$ by
$0$ to $a<0$ and see that the extension is an {$H^1$} solution to the same equation. By
{Trudinger's argument}, it is bounded.

\vskip0.3cm

 But since it vanishes for $a<0$, by Carleman's
{unique continuation theorem}, $v\equiv0$. Hence, {$w^\ast\equiv0$}, giving our {contradiction}.

%%%%%%%%%%%%%%%%%%%%%%%%%%%%%%%%%%%%%%%%%%%%%%%%%%%%%%%%%%%%%%%%%%%%%%%%%%%%%%%%%%%%%%%%%%%%%%%%%%%%%%%%%%%%%%%%%%%%%%%%%%%%%%%%%%%%%%%%%%%%%%

\begin{center}

\end{center}


\begin{thebibliography}{99}
%\addcontentsline{toc}{section}{References}


\bibitem{Au76} T. Aubin.
 \'{E}quations diff\'erentielles non lin\'eaires et probl\`eme de
  {Y}amabe concernant la courbure scalaire.
 J. Math. Pures Appl. (9), 55(3):269--296, 1976.


\bibitem{bahchemin} H. Bahouri and J. Chemin, {On global well-posedness for defocusing cubic wave equation},
Internat. Math. Res. Notices, 2006: 54873.


\bibitem{BG} H. Bahouri, P. G$\acute{e}$rard, High frequency
approximation of solutions to critical nonlinear wave equations,
Amer. J. Math., 121(1999), 131-175.


\bibitem{Bizon02}
{P. Bizo{\'n},}
\newblock Formation of singularities in {Y}ang-{M}ills equations.
\newblock { Acta Phys. Polon. B 33}, 7 (2002), 1893--1922.

\bibitem{BCT04}
P. Bizo{\'n}, T. Chmaj, and Z. Tabor.
\newblock On blowup for semilinear wave equations with a focusing nonlinearity.
\newblock { Nonlinearity}, 17(6):2187--2201, 2004.


\bibitem{Bo99a} J. Bourgain, Global well-posedness of defocusing 3D
critical NLS in the radial case. J. Amer. Math. Soc., 12(1999),
145-171.


\bibitem{brenner}
P.~Brenner, P.~Kumlin, {On wave equations with supercritical nonlinearities}, Arch. Math. \text{74} (2000), 129--146.


\bibitem{BW} P. Brenner, W. Von Wahl, Global classical solutions of nonlinear wave equations. Math. Z., 1981, 176: 87-121.

\bibitem{Bu2010} A. Bulut, Maximizers for the Strichartz
inequalities for the wave equation. Differential Integral Equations
23 (2010), 1035-1072.

\bibitem{Blut2012} A. Bulut, Global well-posedness and scattering
for the defocusing energy-supercritical cubic nonlinear wave
equation. J. Func. Anal. 263 (2012), 1609-1660. MR2948225.

\bibitem{Blut2011} A. Bulut, The radial defocusing
energy-supercritical cubic nonlinear wave equation in $\R^{1+5}$.
Nonlinearity, 27(2014), 1859-1877.


\bibitem{Blut20115} A. Bulut, The defocusing energy-supercritical cubic nonlinear wave
equation in dimension five. Trans. Amer. Math. Soc., 367(2015). 6017-6061.


\bibitem{Blut} A. Bulut, The defocusing cubic nonlinear wave
equation in the energy-supercritical regime. Contemporary
Mathematics, 581, Amer. Math. Soc., Providence, RI, 2012.


\bibitem{BCLPZ} A. Bulut, M. Czubak, D. Li, N. Pavlovic and X. Zhang, Stability and unconditional uniqueness of solutions for energy
critical wave equation in high dimensions.   Comm. Partial
Differential Equations,  38 (2013), 575-607.


\bibitem{Cazenave85a}
{T. Cazenave,}
\newblock Uniform estimates for solutions of nonlinear {K}lein-{G}ordon
  equations.
\newblock { J. Funct. Anal. 60}, 1 (1985), 36-55.

\bibitem{CaLi84}
{T.  Cazenave and P. L. Lions,}
\newblock Solutions globales d'\'equations de la chaleur semi lin\'eaires.
\newblock { Comm. Partial Differential Equations 9}, 10 (1984), 955--978.


\bibitem{Chemin}{J.-Y. Chemin},
{ About weak-strong uniqueness for the 3{D} incompressible
              {N}avier-{S}tokes system}, {Comm. Pure Appl. Math.},  64(2011), {1587-1598}.
	
\bibitem{Chen-Miao-Zhang-1}{Q. Chen, C. Miao and Z. Zhang},
 { On the uniqueness of weak solutions
for the 3D Navier-Stokes equations}, \newblock{Ann. Inst. Henri Poincare-Nonlinear Analysis}, 26(2009)2165-2180.

\bibitem{cct1}
M.~Christ, J.~Colliander, and T.~Tao, Ill-posedness for nonlinear Schr\"odinger and wave equations, arXiv: 0311048.


\bibitem{CW} M. Christ and M. Weinstein, Dispersion of small
amplitude solutions of the generalized Korteweg-de Vries equation.
J. Funct. Anal. 100 (1991), 87-109. MR1124294.


\bibitem{ChTZ93}
{D. Christodoulou and A. Tahvildar-Zadeh,}
\newblock On the asymptotic behavior of spherically symmetric wave maps.
\newblock { Duke Math. J. 71}, 1 (1993), 31-69.



\bibitem{Collot}
C. Collot,
\newblock Type II blow-up for the energy supercritical wave equation, 2014.
\newblock arXiv: 1407.4525v2.


\bibitem{Kenig4dWave} R . C\^ote, C. Kenig, A. Lawrie, W. Schlag, Profiles for the radial focusing 4d energy-critical wave equation,  arXiv:1402.2307. To appear in Comm. Math. Phys.


\bibitem{CoteZaag11P}
{R. C\^ote, and H. Zaag,}
\newblock Construction of a multi-soliton blow-up solution to the semilinear
  wave equation in one space dimension.
\newblock Comm. Pure and Appl. Math., 66(2013), 1541-1581.


\bibitem{Ding86}
{ W. Ding,}
\newblock On a conformally invariant elliptic equation on {${\R}^n$}.
\newblock { Comm. Math. Phys. 107}, 2 (1986), 331--335.

\bibitem{Dod1} B. Dodson, Global well-posedness and scattering for the radial, defocusing, cubic wave equation with almost sharp initial data, arXiv: 1604.04255.

\bibitem{Dod2} B. Dodson, Global well-posedness and scattering for the radial, defocusing, cubic wave equation with initial data in a critical Besove space, arXiv: 1608.02020.


\bibitem{DL} B. Dodson and A. Lawrie, Scattering for the radial 3d
cubic wave equation.  Analysis and PDE, 8(2015), 467-497.


\bibitem{DL14} B. Dodson and A. Lawrie, Scattering for radial, semi-linear,
super-critical wave equations with bounded critical norm. Arch. Rational Mech. Anal., 218(2015), 1459-1529.


\bibitem{DHKS} R. Donninger, H. Min, J. Krieger, W. Schlag, Exotic blowup solutions for the $u^5$ focusing wave equation in~$\R^3$, Michigan Math. J., 63(2014), 451-501.

\bibitem{DoKr12P}
{R. Donninger and J. Krieger,}
\newblock Nonscattering solutions and blowup at infinity for the critical wave
  equation.
\newblock Math. Anna., 357(2013), 89-163.



\bibitem{DS12}
R. Donninger and B. Sch{\"o}rkhuber.
\newblock Stable self-similar blow up for energy subcritical wave equations.
\newblock { Dyn. Partial Differ. Equ.}, 9(1):63--87, 2012.


\bibitem{DonningerSchorkhuber}
R. Donninger  and B. Sch{\"o}rkhuber,
\newblock Stable blow up dynamics for energy supercritical wave equations.
\newblock { Trans. Amer. Math. Soc. 366}, 4 (2014), 2167--2189.

\bibitem{DonningerSchorkhuber2}
R. Donninger and B. Sch{\"o}rkhuber,
\newblock On blowup in supercritical wave equations, 2014.
\newblock Comm. Math. Phy., 346(2016), 907-943.


\bibitem{DonningerSchorkhuber1}
R. Donninger and B. Sch{\"o}rkhuber,
Stable blowup for wave equations in odd space dimensions, To appear in Anna. de l'institut Henr. Poin. (C) Nonlinear Analysis.

\bibitem{DJKM} T. Duyckaerts, H. Jia, C.Kenig and F.Merle, Soliton resolution along a sequence of times for the focusing energy-critical wave equation, arXiv: 1601.01871.


\bibitem{DKM11} T. Duyckaerts, C.Kenig and F.Merle, Universality of
blow-up profile for small radial type II blow-up solutions of the
energy-critical equation. J. Eur. Math. Soc. 13(2011) 533-599.


\bibitem{DKM12JEMS} T. Duyckaerts, C. Kenig, F. Merle, Universality of the blow-up profile for small type II blow-up solutions
of energy-critical wave equation, J. Eur. Math. Soc. (JEMS) 14:5 (2012), 1389-1454.



\bibitem{DKM13} T. Duyckaerts, C. Kenig, F. Merle, {Classification of radial solutions of the focusing, energy-critical wave equation}, Cambridge Journal of Mathematics 1 (2013), no. 1, 75-144

\bibitem{DKM} T. Duyckaerts, C. Kenig and F.Merle, Scattering for radial, bounded solutions of focusing
supercritical wave equations. Int Math Res Notices IMRN, 2014, 224-258.


\bibitem{DuKeMe11P}
{T. Duyckaerts, C.  Kenig, and F. Merle,}
\newblock Profiles of bounded radial solutions of the focusing, energy-critical
  wave equation.
\newblock  Geometric and Functional Analysis, 22(2012), 639-698.


\bibitem{DKMprofile}T. Duyckaerts, C. Kenig, F. Merle, {Profiles for bounded solutions of dispersive equations, with applications to energy-critical wave and Schr\"odinger equations},
Commun. Pure Appl. Anal. 14 (2015), no. 4, 1275-1326



\bibitem{DM08} T. Duyckaerts and F. Merle, Dynamics of threshold
solutions for energy-critical wave equation. Int. Math. Res. Pap.
2008, art. ID rpn002, 67 pp. MR 2416069.


\bibitem{DuMe07}
T. Duyckaerts, and F. Merle,
 Dynamic of threshold solutions for energy-critical {NLS}.
 Geom. Funct. Anal., 18(2009), 1787-1840.

 \bibitem{DuMe10} T. Duyckaerts, and F. Merle, Scattering norm estimate near the threshold for energy-critical focusing semilinear wave equation.
 Indiana Univ. Math. J., 58(2009), 1971-2001.

 \bibitem{DuRoy} T. Duyckaerts, and T. Roy, Blow-up of the critical Sobolev norm for nonscattering radial solutions of supercritical wave equations on $\mathbb{R}^{3}$,


\bibitem{FT03}
G. Furioli and E. Terraneo.
Besov spaces and unconditional well-posedness for the nonlinear Schr\"odinger equation in
$\dot{H}^s(\R^n)$. Commun. Contemp. Math. 5 (2003), no. 3, 349--367.


\bibitem{FPT03} G. Furioli, F. Planchon, E. Terraneo,
Unconditional well-posedness for semilinear Schr\"odinger and wave equations in $H\sp s$.
(English summary) Harmonic analysis at Mount Holyoke (South Hadley, MA, 2001), 147--156,
Contemp. Math., 320, Amer. Math. Soc., Providence, RI, 2003.


\bibitem{gallagplanch} I. Gallagher and F. Planchon, {On global solutions to a dofocusing
semi-linear wave equation}, Revista Mathematic$\acute{a}$
Iberoamericana, 19, 2003, pp. 161-177


\bibitem{Germ1}{ P. Germain}, { Strong solutions and weak-strong uniqueness for the
              nonhomogeneous {N}avier-{S}tokes system}, {J. Anal. Math.}, Vol.{ 105}, {169--196}, 2008.
	
	
\bibitem{Germ2}{ P. Germain}, { Weak-strong uniqueness for the isentropic compressible
              {N}avier-{S}tokes system}, {J. Math. Fluid Mech.}, Vol. { 13}, {137--146}, 2011.


	
\bibitem{Germ3}{P. Germain}, { Multipliers, paramultipliers, and weak-strong uniqueness for
              the {N}avier-{S}tokes equations}, {J. Differential Equations}, Vol. { 226}, {373--428}, 2006.

	


\bibitem{GNN81} B. Gidas, W. Ni and L. Nirenberg, Symmetry of
positive solutions of nonlinear elliptic equatios in $\R^n$. In
Mathematical analysis and applications, Part A, volume 7 of Adv. in
Math. Suppl. Stud., pages 369-402. Academic Press, New York, 1981.



\bibitem{GK89} Y. Giga and R. V. Kohn. Nondegeneracy of blowup for semilinear heat equations.
Comm. Pure Appl. Math., 42(6):845-884, 1989.


\bibitem{GV1} J. Ginibre and G. Velo, On the global Cauchy problem for nonlinear Klein-Gordon equation. Math. Z., 1985, 189: 487-505.

\bibitem{GV2} J. Ginibre and G. Velo, On the  global Cauchy problem for nonlinear Klein-Gordon equation II. Ann. Inst. Henri. Poincar\'e Analyse nonli\'eaire, 1989, 6: 15-35.

\bibitem{GiV95} J. Ginibre and G. Velo, Generalized Strichartz
inequalities for the wave equation, J. Funct. Anal., 133(1995),
50-68.






\bibitem{Gri90} M. Grillakis, Regularity and asymptotic behaviour of the wave equation with a critical nonlinearity,
Ann. of Math., 132(1990), 485-509.

\bibitem{Gri92} M. Grillakis,  Regularity  for nonlinear wave equation with critical nonlinearity, Comm. Pure Appl. Math., 1992, 45: 749-774.


\bibitem{Gundlach99}
{ C. Gundlach, C.}
\newblock Critical phenomena in gravitational collapse.
\newblock { Living Rev. Relativ. 2\/} (1999), 1999--4, 58 pp. (electronic).


\bibitem{HZ}  M-A. Hamza, H. Zaag. Blow-up results for semilinear wave equations in the superconformal case. Discrete Contin. Dyn. Syst. Ser. B, 18(9):2315-2329, 2013.



\bibitem{HR12} M. Hillairet and P. Raphael, Smooth type II blow up solutions to the four dimensional energy-critical wave equation, Analysis and PDE's, 5(2012), 777-829.

\bibitem{Ho85} L. H\"ormander. The analysis of linear partial differential operators. III, volume
274 of Grundlehren der Mathematischen Wissenschaften [Fundamental
Principles of Mathematical Sciences]. Springer-Verlag, Berlin, 1985. Pseudodifferential operators.


\bibitem{IMN} S. Ibrahim, N. Masmoudi and K. Nakanishi, Scattering
threshold for the focusing nonlinear Klein-Gordon equation, Analysis
and PDE., 4 (2011), No3, 405-460.

\bibitem{Jen1}
J.~Jendrej. Construction of two-bubble solutions  for energy-critical wave equations, arXiv: 1602.01093.


\bibitem{moi15p-3}
J.~Jendrej.
\newblock Nonexistence of radial two-bubbles with opposite signs for the
  energy-critical wave equation.
\newblock { Preprint}, arXiv:1510.03965, 2015.

\bibitem{moi15p}
J.~Jendrej.
\newblock Construction of type {II} blow-up solutions for the energy-critical
  wave equation in dimension 5.
\newblock J. Funct. Anal., 272(2017), 866-917.

\bibitem{Jen} J. Jendrej, Bounds on the speed of type II blow-up for the energy critical wave equation in the radial case, Int Math Res Notices (2016)2016(21): 6656-6688.

\bibitem{JiaKenig} H. Jia, C. Kenig, {Asymptotic decomposition for semilinear wave and equivariant wave map equations}, preprint 2015, arXiv 1503.06715

\bibitem{Jor} K. J\"orgen, Das Anfangswert problem im Grossen f\"ur eine nichlineare Wellengleichungen. Math. Z., 77(1961): 295-308.

\bibitem{JL} D. Joseph and T. Lundgren. Quasilinear Dirichlet
problems deriven by positive sources. Arch. Rational Mech. Anal.,
49: 241-269, 1972/1973.

\bibitem{Kapi94} L. Kapitanski, Global and unique weak solution of
nonlinear wave equations, Math. Res. Lett., 1(1994), 211-223.


\bibitem{kapitanski:wayw}
L.~Kapitanski, {Weak and Yet Weaker Solutions of Semilinear Wave
Equations}, Comm. Part. Diff. Eq., \text{19} (1994), 1629--1676.

\bibitem{KeT98} M. Keel and T. Tao, Endpoint Strichartz estimates. Amer. J.
Math. 120:5(1998), 955-980.

\bibitem{Kenig01} C. Kenig, The concentration-compactness/Rigidity method for critical dispersive and wave equations. https://www.math.uchicago.edu/$\sim$cek/Kenig.pdf.

\bibitem{Kenig02} C. Kenig, Global well-posedness, scattering and blow up for the energy-critical, focusing, non-linear Schr\"odinger and wave equations. https://www.math.uchichago.edu/$\sim$cek/Kenigrev1.pdf.

\bibitem{KM1} C. Kenig and F. Merle, Global well-posedness, scattering and
blow-up for the energy critical focusing non-linear wave equation,
Acta Math., 201 (2008), 147-212.


\bibitem{KM2010} C. Kenig and F. Merle, Scattering for $\dot{H}^\frac12$ bounded solutions
to the cubic, defocusing NLS in $3$ dimensions. Trans. Amer. Math.
Soc. 362 (2010), 1937-1962. MR2574882.

\bibitem{KM2011} C. Kenig and F.Merle, Nondispersive radial solutions to
energy supercritical nonlinaer wave equations, with applications.
Amer. J. Math. 133 (2011), no.4, 1029-1065.

\bibitem{KM2011D} C. Kenig and F.Merle, Radial solutions to energy
supercritical wave equations in odd dimensions. Disc. Cont. Dyn.
Sys. A, 4 (2011) 1365-1381.


\bibitem{KPV93} C. Kenig, G. Ponce, and L. Vega. Well-posedness and scattering results for the
generalized Korteweg-de Vries equation via the contraction principle. Comm.
Pure Appl. Math., 46(4):527-620, 1993.


\bibitem{kenponcevega} C. E. Kenig, G. Ponce and L. Vega, {Global well-posedness for
semi-linear wave equations}, Communications in partial differential
equations 25 (2000), pp. 1741-1752



\bibitem{KV} R. Killip and M. Visan, The defocusing
energy-supercritical nonlinear wave equation in three space
dimensions, Trans. Amer. Math. Soc. 363 (2011), 3893-3934.

\bibitem{KV2011} R. Killip and M. Visan, The radial defocusing
energy-supercritical nonlinear wave equation in all space
dimensions. Proc. Amer. Math. Soc. 139 (2011), 1805-1817.

\bibitem{KNW} J.~Krieger, K.~Nakanishi and W. Schlag, Global dynamics away from the ground state for the energy-critical nonlinear wave equation,
Calculus of Variations and Partial Differential Equations, 44(2012), 1-45.



\bibitem{KS07}
J.~Krieger and W.~Schlag.
\newblock On the focusing critical semi-linear wave equation.
\newblock { Amer. J. Math.}, 129(3):843--913, 2007.

\bibitem{KrSc14}
J.~Krieger and W.~Schlag.
\newblock Full range of blow up exponents for the quintic wave equation in
  three dimensions.
\newblock { J. Math Pures Appl.}, 101(6):873--900, 2014.

\bibitem{KST09} J. Krieger, W. Schlag and D. Tataru, Slow blow-up
solutions for the $H^1(\R^3)$ critical focusing semilinear wave
equation. Duke Math. J., 147(1): 1-53, 2009.



\bibitem{lebeau}
G.~Lebeau,
{\em Optique non lin\'eaire et ondes sur critiques},
S\'eminaire: \'Equations aux D\'eriv\'ees Partielles, 1999--2000,
Exp. No. IV, 13 pp., S\'emin. \'Equ. D\'eriv. Partielles,
\'Ecole Polytech., Palaiseau, 2000.

\bibitem{Leray} { J. Leray}, {\em \'Etude de diverses \'equations int\'egrales non lin\'eaires et de quelques probl\`emes que pose l'hydrodynamique},
J. Math. Pures Appl., IX. S\'er. {12}, 1--82, 1933.


\bibitem{Le74} H. A. Levine, Instability and nonexistence of global solutions to nonlinear wave equations of
the form $Pu_{tt}=-Au + F(u)$", Trans. Amer. Math. Soc. 192 (1974), 1-21.



\bibitem{LZ09JFA} D. Li and X. Zhang, Dynamics for the energy critical nonlinear Schr\"odinger equation in high dimensions. J. Funct. Anal., 256(2009), 1928-1961.


\bibitem{LZ10} D. Li and X. Zhang, Dynamics for the energy critical
nonlinear wave equation in high dimensions, Tran. AMS., 363 (2011)
1137-1160.


\bibitem{LS} H. Lindblad and C. Sogge, On existence and scattering
with minimal regularity for semilinear wave equations, J. Funct.
Anal., 130 (1995), 357-426.

\bibitem{MM15} M. Majdoub, N. Masmoudi, On uniqueness for supercritical nonlinear wave and Schr\"odinger equations, Int Math Res Notices (2015) 2015(9): 2386-2405.


\bibitem{MP06} N. Masmoudi and F. Planchon, { On uniqueness for the
critical wave equation}.  Comm. Partial Differential Equations
\text{31} (2006), no. 7-9, 1099--1107.



\bibitem{MatanoMerle09}
{ H.  Matano  and F. Merle,}
\newblock Classification of type {I} and type {II} behaviors for a
  supercritical nonlinear heat equation.
\newblock { J. Funct. Anal. 256}, 4 (2009), 992--1064.

\bibitem{MatanoMerle11}
{ H. Matano  and F. Merle,}
\newblock Threshold and generic type {I} behaviors for a supercritical
  nonlinear heat equation.
\newblock { J. Funct. Anal. 261}, 3 (2011), 716--748.


\bibitem{May} {R. May}, {\em Extension d'une classe d'unicit\'e pour les \'equations de
              {N}avier-{S}tokes}, {Ann. Inst. H. Poincar\'e Anal. Non Lin\'eaire}, Vol. { 27}, {705--718}, 2010.



\bibitem{Merle96b}
{ F. Merle,}
\newblock Lower bounds for the blowup rate of solutions of the {Z}akharov
  equation in dimension two.
\newblock {\em Comm. Pure Appl. Math. 49}, 8 (1996), 765--794.

\bibitem{MerleRaphael08}
F. Merle and P. Rapha{\"e}l,
\newblock Blow up of the critical norm for some radial {$L\sp 2$} super
  critical nonlinear {S}chr\"odinger equations.
\newblock { Amer. J. Math. 130}, 4 (2008), 945--978.


\bibitem{MeRaSz10P}
{ F. Merle, P. Rapha{\"e}l, and J. Szeftel,}
\newblock The instability of {B}ourgain-{W}ang solutions for the $l^2$-critical
  {NLS}.  Amer. Jour. Math., 135(2013), 967-1017.




\bibitem{MZ03} F. Merle and H. Zaag. Determination of the blow-up rate for the semilinear
wave equation. Amer. J. Math., 125(5):1147-1164, 2003.


\bibitem{MZ05}
F. Merle and H. Zaag.
\newblock Determination of the blow-up rate for a critical semilinear wave
  equation.
\newblock { Math. Ann.}, 331(2):395--416, 2005.


\bibitem{MZ2}  F. Merle, H.; Zaag. Openness of the set of non-characteristic points and regularity of the blow-up curve for the 1 D semilinear wave equation. Comm. Math. Phys., 282(1):55-86, 2008.


\bibitem{MerZaa14b}
F. Merle and H. Zaag.
\newblock Dynamics near explicit stationary solutions in similarity variables
  for solutions of a semilinear wave equation in higher dimensions.
Trans. Amer. Math. Soc., 368(2016), 27-87.


\bibitem{MerZaa14a}
F. Merle and H. Zaag.
\newblock On the stability of the notion of non-characteristic point and
  blow-up profile for semilinear wave equations.
Comm. Math. Phys., 333(2015), 1529-1562.

\bibitem{MC} C. Miao, Modern Methods to the Nonlinear Wave Equations,
Monographs on Modern pure mathematics, No.133. Science Press, Beijing,  2010.4.

\bibitem{Miao-Zhang-Wave} C. Miao and B. Zhang, $H^s$-global well-posedness for semilinear wave equations,
 {\em J. Math. Anal. Appl}. 283(2003) 645-666.

\bibitem{MWZ} C. Miao, Y. Wu and X. Zhang, The defocusing
energy-supercritical nonlinear wave equation in even space
dimension. preprint.

\bibitem{MZZ} C. Miao, J. Zhang and J. Zheng, The defocusing energy-critical  wave
equation with a cubic convolution., Indiana University Math.
Journal, 63(2014), 993-1015.


\bibitem{MussoPistoia06}
{ Musso, M., and Pistoia, A.}
\newblock Sign changing solutions to a nonlinear elliptic problem involving the
  critical {S}obolev exponent in pierced domains.
\newblock { J. Math. Pures Appl. (9) 86}, 6 (2006), 510--528.

\bibitem{KS10} K. Nakanishi and W. Schlag, Invariant manifolds and dispersive Hailtonian evolution equations, Z\"urich Lectures in Advanced Mathematics, 2011.


\bibitem{Pech} H. Pecher, $L^p$ Absch\"atzungen und klassische L\"osungen f\"urnichr lineare wellengleichungen.I. Math. Z., 150(1976), 159-183.



\bibitem{P04}
F. Planchon, { On uniqueness for semilinear wave equations},
Math. Z. \text{244}  (2003),  no. 3, 587--599.

\bibitem{Rauch}  J. Rauch, The $u^5$-Klein-Gordon equation. Pitman Research Notes in Math. (Brezis H and Lions J L, eds), 1982, 53: 335-364.

\bibitem{Casey} C. Rodriguez, {Profiles for the radial focusing energy-critical wave equation in odd dimensions}, Adv. Diff. Equa., 21(2016), 505-570.


\bibitem{Ramos} J. Ramos, A refinement of the Strichartz inequality for the wave
equation with applications, Advances in Mathematics, 230(2012), 649-698.


\bibitem{triroy} T. Roy, {Global well-posedness for the radial
defocusing cubic wave equation on $\mathbb{R}^{3}$ and for rough
data}, Elec. Jour. Diff. Equa., 2007(2007), 166: 1-22.



\bibitem{triroy1} T. Roy, Adapted Linear-Nonlinear Decomposition and Global well-posedness for solutions
to the defocusing cubic wave equation on $\mathbb{R}^{3}$, Disc. Cont. Dynam. Systems A, 24(2009), 1307-1323.


\bibitem{RoyLog}
T. Roy,
\newblock Global existence of smooth solutions of a 3{D} log-log
  energy-supercritical wave equation.
\newblock { Anal. PDE } Vol. 2 (2009), 261--280.

\bibitem{RoyLogSchrod} T. Roy,
\newblock Scattering above energy norm of solutions of a loglog energy-supercritical Schr\"odinger equation
with radial data.
\newblock { Journal of Differential Equations}, 250 (2011), no. 1, 292--319.



\bibitem{ShaStr93} J. Shatah, M. Struwe, Regularity results for nonlinear wave equations, Ann. Math., 1993, 138: 505-518.


\bibitem{ShaStr94} J. Shatah, M. Struwe, Well posedness in the energy
space for semilinear wave equation with critical growth, Inter.
Math. Researth Notice, (1994), 303-309.


\bibitem{SS98} J. Shatah and M. Struwe. Geometric wave equations, volume 2 of Courant
Lecture Notes in Mathematics. New York University Courant Institute of
Mathematical Sciences, New York, 1998.


\bibitem{Shen1} R. Shen, I-method for Defocusing, Energy-subcritical Nonlinear Wave Equation, arXiv: 1205.4739.

\bibitem{Shen} R. Shen, On the energy subcritical nonlinear wave
equation in $\R^3$ with radial data. Analysis and PDE, 6 (2013),
1929-1987.


\bibitem{Shih} Shih Hsi-Wei, Some results on scattering for log-subcritical and log-supercritical nonlinear wave equations, { Anal. PDE } Vol. 6 (2013), No. 1, 1--24.

\bibitem{PSJ} P. Sj\"olin, Global maximal estimates for solutions to the Schr\"odinger equation, Studia Math., 110(2) 1994, 105-114.

\bibitem{Sof06}
A. Soffer.
\newblock Soliton dynamics and scattering.
\newblock In { International {C}ongress of {M}athematicians. {V}ol. {III}},
  pages 459--471. Eur. Math. Soc., Z\"urich, 2006.


\bibitem{sogge:wave}
C.~D.~Sogge, {Lectures on Nonlinear Wave Equations}, Monographs in Analysis
 II, International Press, 1995.



\bibitem{St81a} W. A. Strauss, Nonlinear scattering theory at low
energy, J. Funct. Anal., 41(1981), 110-133.

\bibitem{St81b} W. A. Strauss, Nonlinear scattering theory at low
energy sequel, J. Funct. Anal., 43(1981), 281-293.

\bibitem{Strauss} {W. Strauss}, { Nonlinear wave equations},
Conf. Board of the Math. Sciences, { 73}, Amer. Math. Soc., 1989.



\bibitem{Stru88} M. Struwe, Global regular solution to the $u^5$ Klein-Gordon equations. Ann Scu. Norm Sup. Pisa., 1988, 15: 495-513.

\bibitem{struwe99} M. Struwe. Uniqueness for critical nonlinear wave equations and wave maps via the energy inequality. Comm. Pure Appl. Math., 52(9): 1179-1188, 1999.




\bibitem{Struwe-IMRN} { M. Struwe}, {\em On uniqueness and stability for supercritical nonlinear wave
              and {S}chr\"odinger equations}, Int. Math. Res. Not. { 2006}, Art. ID 76737, 14 pp.


\bibitem{tao:lowreg}
T.~Tao, \emph{Low regularity semilinear wave equations}, Comm. PDE. \text{24} (1999), 599--630.



\bibitem{T07} T. Tao, Spacetime bounds for the energy-critical nonlinear wave equation in three
spatial dimensions, Dynamics of PDE, 3 (2006), 93-110.


\bibitem{TaoLog} Tao T.
\newblock Global regularity for a logarithmically supercritical defocusing
  nonlinear wave equation for spherically symmetric data.
\newblock { J. Hyperbolic Differ. Equ. 4} 2 (2007), 259--265.

\bibitem{Ta76}
G. Talenti,
 Best constant in {S}obolev inequality.
Ann. Mat. Pura Appl. (4), 110:353-372, 1976.


\bibitem{Tr68} N. Trudinger. Remarks concerning the conformal deformation of Riemannian
structures on compact manifolds. Ann. Scuola Norm. Sup. Pisa (3), 22:265-274, 1968.


\bibitem{Vega} L. Vega, Schr\"odinger equations: pointwise
convergence to the initial data, Proceeding of the American
Mathematical Society 102(1988), No.4, 874-878.

\bibitem{Wa} W. Wahl, $L^p$ decay rates for the homogeneous wave equations. Math. Z., 1971, 120: 93-106.

\bibitem{Wal} B. G. Walther, Some $L^p(L^\infty)$ and
$L^2(L^2)$-estimates for osciallatory Fourier transforms, in
Analysis of Divergence (Orono, ME, 1997), Appl. Numer. Harmon.
Anal., pp. 213-231, Birkh\"auser, Boston, MA, 1999.


\end{thebibliography}
\end{document}